\documentclass[11pt]{article}
\usepackage[margin=1in]{geometry}
\usepackage[dvipsnames]{xcolor}
\usepackage{mathtools}
\usepackage{amsfonts}
\usepackage{amssymb}

\usepackage{mathrsfs}
\usepackage{bm}
\usepackage{bbm}
\usepackage{eucal}
\usepackage{enumerate}
\usepackage{enumitem}
\usepackage{subcaption}
\usepackage{soul}
\usepackage{comment}
\usepackage{xspace}
\usepackage{mleftright}
\usepackage{dsfont}
\usepackage{booktabs}
\usepackage{newtxtext}
\usepackage{nicefrac}
\usepackage{relsize} \usepackage{amsmath,amssymb}
\usepackage{algorithm}
\usepackage{algorithmicx}
\usepackage{algpseudocode}
\algtext*{EndFor}
\algtext*{EndIf}
\algtext*{EndFunction}
\algnewcommand\INPUT{\item[\textbf{Input:}]}\algnewcommand\OUTPUT{\item[\textbf{Output:}]}

\usepackage[most]{tcolorbox}
\usepackage{tikz}
\usetikzlibrary{spath3} \usepackage{amsthm}
\usepackage{cite}
\pgfdeclarelayer{background}
\pgfsetlayers{background,main}

\DeclareFontFamily{U}{FdSymbolA}{}
\DeclareFontShape{U}{FdSymbolA}{m}{n}{<->FdSymbolA-Book}{}
\DeclareSymbolFont{fdsymbols}{U}{FdSymbolA}{m}{n}
\DeclareMathSymbol{\varheartsuit}{\mathord}{fdsymbols}{184}

\let\oldmathbb\mathbb
\renewcommand{\mathbb}[1]{\ifnum\pdfstrcmp{#1}{1}=0
    \mathds{1}\else
    \oldmathbb{#1}\fi
}

\usepackage[
  pagebackref,
  breaklinks,
  colorlinks,
  linkcolor=NavyBlue,
  citecolor=Blue,
  filecolor=Purple,
  urlcolor=Purple
]{hyperref}

\renewcommand*{\backref}[1]{}
\renewcommand*{\backrefalt}[4]{\ifcase #1 (No citations.)\or
  (Cited p.~#2.)\else
  (Cited pp.~#2.)\fi
}

\usepackage[capitalize,nameinlink,noabbrev]{cleveref}

\newcommand{\ENEdone}{}
\newcommand{\RAMdone}{}
\newcommand{\CCdone}{}
\usepackage{wasysym}

\author{Chris Camaño\thanks{Department of Computing and Mathematical Sciences, California Institute of Technology, Pasadena, CA 91125 USA (\email{ccamano@caltech.edu}, \email{eepperly@berkeley.edu},  \email{ram900@berkeley.edu}, \email{jtropp@caltech.edu})}\and Ethan N. Epperly\footnotemark[1] \and Raphael A. Meyer\footnotemark[1] \and Joel A. Tropp\footnotemark[1]}

\hypersetup{
  colorlinks = true,
  linkcolor  = Green,
  citecolor  = NavyBlue,
  filecolor  = Plum,
  urlcolor   = Plum
}

\newcommand{\email}[1]{\href{mailto:#1}{#1}}

\algrenewcommand\alglinenumber[1]{\sf\scriptsize\color{NavyBlue}{#1}}
\definecolor{cit}{rgb}{0.91,0.39,0.16}
\definecolor{dark-gray}{gray}{0.3}
\definecolor{dkgray}{rgb}{.3,.3,.3}
\definecolor{medgray}{rgb}{.5,.5,.5}
\definecolor{ltgray}{rgb}{.7,.7,.7}
\definecolor{dkblue}{rgb}{0,0,.5}
\definecolor{medblue}{rgb}{0,0,.75}
\definecolor{ltblue}{rgb}{0.97,0.97,1}
\definecolor{rust}{rgb}{0.5,0.1,0.1}
\definecolor{ltyellow}{rgb}{1,1,0.8}

\sethlcolor{ltyellow}

\DeclareMathOperator{\tr}{Tr}

\DeclareMathOperator{\diag}{diag}
\DeclareMathOperator{\orth}{orth}

\newcommand{\lowrank}[1]{\mleft\llbracket #1 \mright\rrbracket}
\newcommand{\norm}[1]{\mleft\| #1 \mright\|}
\newcommand{\snorm}[1]{\mleft\|\smash{#1}\mright\|}

\newcommand{\Id}{\mathbf{I}}

\newcommand{\bmat}[1]{\begin{bmatrix} #1 \end{bmatrix}}

\newcommand{\sbmat}[1]{\left[\begin{smallmatrix} #1 \end{smallmatrix}\right]}

\DeclareMathOperator{\Var}{Var}

\DeclareMathOperator{\E}{\mathbb{E}}

\newcommand{\bernoulli}{\text{\textsc{Bernoulli}}}

\newcommand{\order}{\mathcal{O}}

\newcommand{\defeq}{\coloneqq}
\newcommand{\eps}{\varepsilon}
\let\epsilon\eps

\newcommand{\set}[1]{\mathsf{#1}}
\newcommand{\e}{\mathrm{e}}

\renewcommand{\hat}[1]{\widehat{#1}}
\renewcommand{\tilde}[1]{\widetilde{#1}}

\let\originalchi\chi
\newcommand{\raisedchi}[2]{\raisebox{0.4ex}{$#1#2$}}
\renewcommand{\chi}{{\mathpalette{\raisedchi}{\originalchi}}}

\DeclareMathOperator*{\argmin}{argmin}

\newcommand{\actionbox}[1]{\begin{tcolorbox}[colback=white,colframe=black,width=\columnwidth,boxsep=5pt,arc=4pt]
    #1
  \end{tcolorbox}
}

\definecolor{ltyellow}{rgb}{1, 1, 0.9}
\usepackage{soul}
\sethlcolor{ltyellow}

\let\oldthebibliography\thebibliography
\renewcommand{\thebibliography}[1]{\oldthebibliography{#1}\setlength{\itemsep}{0pt}\setlength{\parskip}{0pt}}

\renewcommand{\eqref}{\cref}

\usepackage{xcolor}
\colorlet{deadgray}{black!55}

\usepackage{stmaryrd}

\newcommand{\lra}[2]{\left\llbracket #1 \right\rrbracket_{#2}}

\newcommand{\nystrompp}{\text{Nystr\"om\raisebox{0.35ex}{\relscale{0.75}++}}\xspace}
\newcommand{\XNysTrace}{XNysTrace\xspace}
\newcommand{\nystromppest}{\hat{\tr}_{\rm N\raisebox{0.35ex}{\relscale{0.75}++}}}
\newcommand{\xnystraceest}{\hat{\tr}_{\rm XN}}

\DeclareMathOperator{\GH}{H}
\newcommand{\iu}{\mathrm{i}}
\renewcommand{\d}{\mathrm{d}}
\renewcommand{\bar}{\overline}
\newcommand{\GramNystrom}{\textsc{GramNyström}\xspace}
\algrenewcommand\algorithmiccomment[1]{\hfill\textcolor{NavyBlue}{$\triangleright$ #1}}
\renewcommand{\defeq}{\coloneqq}

\newsavebox{\algsectionbox}
\newcommand{\algsection}[1]{\Statex
  \vspace{-0.35\baselineskip}\sbox{\algsectionbox}{{\scriptsize\textsc{#1}}}\makebox[\linewidth]{\leaders\hrule height 0.3pt\hfill
    \hspace{0.25em}\usebox{\algsectionbox}\hspace{0.25em}\leaders\hrule height 0.3pt\hfill\kern0pt
  }\vspace{0.05\baselineskip}}

\definecolor{MonteCarloRed}{HTML}{C02232}
\DeclareRobustCommand{\blackdashedline}{\tikz[baseline=-0.55ex]{\draw[
            black,
            line width=0.8pt,
            dash pattern=on 2.2pt off 1.5pt
        ] (0,0) -- (2.0em,0);
    }}

\DeclareRobustCommand{\reddashedline}{\tikz[baseline=-0.55ex]{\draw[
            MonteCarloRed,
            line width=0.8pt,
            dash pattern=on 2.2pt off 1.5pt
        ] (0,0) -- (2.0em,0);
    }} \newcommand{\QB}{\ensuremath{\mathsf{QB}}\xspace}
\newcommand{\bbR}{\ensuremath{\mathbb{R}}\xspace}
\newcommand{\bbC}{\ensuremath{\mathbb{C}}\xspace}
\newcommand{\bbF}{\ensuremath{\mathbb{F}}\xspace}

\newcommand{\onevec}{\ensuremath{\mathbb{1}}\xspace}
\newcommand{\indicator}{\onevec}

\let\vec\relax
\newcommand{\vec}[1]{\boldsymbol{#1}}

\newcommand{\vb}{\ensuremath{\boldsymbol{b}}\xspace}

\newcommand{\vg}{\ensuremath{\boldsymbol{g}}\xspace}
\newcommand{\vh}{\ensuremath{\boldsymbol{h}}\xspace}
\newcommand{\vi}{\ensuremath{\boldsymbol{i}}\xspace}
\newcommand{\vj}{\ensuremath{\boldsymbol{j}}\xspace}
\newcommand{\vk}{\ensuremath{\boldsymbol{k}}\xspace}
\newcommand{\vl}{\ensuremath{\boldsymbol{\ell}}\xspace}

\newcommand{\vs}{\ensuremath{\boldsymbol{s}}\xspace}

\newcommand{\vu}{\ensuremath{\boldsymbol{u}}\xspace}
\newcommand{\vv}{\ensuremath{\boldsymbol{v}}\xspace}

\newcommand{\vx}{\ensuremath{\boldsymbol{x}}\xspace}
\newcommand{\vy}{\ensuremath{\boldsymbol{y}}\xspace}

\newcommand{\vphi}{\ensuremath{\boldsymbol{\phi}}\xspace}

\newcommand{\vpsi}{\ensuremath{\boldsymbol{\psi}}\xspace}
\newcommand{\vomega}{\ensuremath{\boldsymbol{\omega}}\xspace}

\newcommand{\vone}{\ensuremath{\mathbf{1}}\xspace}
\newcommand{\vzero}{\ensuremath{\mathbf{0}}\xspace}

\newcommand{\mA}{\ensuremath{\mathbf{A}}\xspace}
\newcommand{\mB}{\ensuremath{\mathbf{B}}\xspace}
\newcommand{\mC}{\ensuremath{\mathbf{C}}\xspace}

\newcommand{\mE}{\ensuremath{\mathbf{E}}\xspace}

\newcommand{\mG}{\ensuremath{\mathbf{G}}\xspace}
\newcommand{\mH}{\ensuremath{\mathbf{H}}\xspace}
\newcommand{\mI}{\ensuremath{\mathbf{I}}\xspace}

\newcommand{\mL}{\ensuremath{\mathbf{L}}\xspace}
\newcommand{\mM}{\ensuremath{\mathbf{M}}\xspace}

\newcommand{\mQ}{\ensuremath{\mathbf{Q}}\xspace}
\newcommand{\mR}{\ensuremath{\mathbf{R}}\xspace}
\newcommand{\mS}{\ensuremath{\mathbf{S}}\xspace}
\newcommand{\mT}{\ensuremath{\mathbf{T}}\xspace}
\newcommand{\mU}{\ensuremath{\mathbf{U}}\xspace}
\newcommand{\mV}{\ensuremath{\mathbf{V}}\xspace}

\newcommand{\mX}{\ensuremath{\mathbf{X}}\xspace}
\newcommand{\mY}{\ensuremath{\mathbf{Y}}\xspace}
\newcommand{\mZ}{\ensuremath{\mathbf{Z}}\xspace}
\newcommand{\mLambda}{\ensuremath{\mathbf{\Lambda}}\xspace}
\newcommand{\mSigma}{\ensuremath{\mathbf{\Sigma}}\xspace}

\newcommand{\mOmega}{\ensuremath{\mathbf{\Omega}}\xspace}
\newcommand{\mPsi}{\ensuremath{\mathbf{\Psi}}\xspace}

\newcommand{\mrho}{\ensuremath{\boldsymbol{\rho}}\xspace}

\newcommand{\tA}{\ensuremath{\bm{\mathcal{A}}}\xspace}
\newcommand{\tB}{\ensuremath{\bm{\mathcal{B}}}\xspace}
\newcommand{\tC}{\ensuremath{\bm{\mathcal{C}}}\xspace}

\newcommand{\tG}{\ensuremath{\bm{\mathcal{G}}}\xspace}

\newcommand{\tM}{\ensuremath{\bm{\mathcal{M}}}\xspace}
\newcommand{\tN}{\ensuremath{\bm{\mathcal{N}}}\xspace}

\newcommand{\tT}{\ensuremath{\bm{\mathcal{T}}}\xspace}

\newcommand{\tW}{\ensuremath{\bm{\mathcal{W}}}\xspace}
\newcommand{\tX}{\ensuremath{\bm{\mathcal{X}}}\xspace}

\newcommand{\rC}{\ensuremath{\mathrm{C}}\xspace}
\newcommand{\rc}{\ensuremath{\mathrm{c}}\xspace}

\newcommand{\rK}{\ensuremath{\mathrm{K}}\xspace}

\newcommand{\cN}{\ensuremath{\mathcal{N}}\xspace}
\newcommand{\cO}{\ensuremath{\mathcal{O}}\xspace}

 \tikzset{
  leg/.style={thick},
  copy/.style={circle,draw,fill=white,inner sep=0.8pt},
  every node/.style={inner sep=0,outer sep=0}
}

\newcommand{\Rdtensor}[1]{(\mathbb{R}^{d})^{\otimes #1}}

\newcommand{\rmps}{\ensuremath{\mathsf{rMPS}}\xspace}
\newcommand{\mps}{\ensuremath{\mathsf{MPS}}\xspace}
\newcommand{\mpo}{\ensuremath{\mathsf{MPO}}\xspace}
\newcommand{\rmpss}{\ensuremath{\mathsf{rMPSs}}\xspace}
\newcommand{\mpss}{\ensuremath{\mathsf{MPSs}}\xspace}
\newcommand{\mpos}{\ensuremath{\mathsf{MPOs}}\xspace}

\newcommand{\trps}{\rmps test matrices\xspace}
\newcommand{\trp}{\rmps test matrix\xspace}
\newcommand{\rmpstest}{\trp}

\newcommand{\tnshow}[3][0pt]{\mathord{\vcenter{\hbox{\raisebox{#1}{\includegraphics[width=#3\linewidth,keepaspectratio]{figs2/#2.pdf}}}}}}

 \makeatletter
\def\hlinewd#1{\noalign{\ifnum0=`}\fi\hrule \@height #1 \futurelet
	\reserved@a\@xhline}
\makeatother

\newtheorem{theorem}{Theorem}
\newtheorem{proposition}[theorem]{Proposition}

\newtheorem{corollary}[theorem]{Corollary}

\newtheorem{importedtheorem}[theorem]{Imported Theorem}
\newtheorem{importedlemma}[theorem]{Imported Lemma}

\numberwithin{theorem}{section}
\numberwithin{equation}{section}

\theoremstyle{definition}
\newtheorem{definition}[theorem]{Definition}

\theoremstyle{remark}

\crefname{setting}{setting}{settings}
\Crefname{setting}{Setting}{Settings}
\crefname{problem}{problem}{problems}
\Crefname{problem}{Problem}{Problems}
\crefname{equation}{}{}
\Crefname{equation}{Equation}{Equations}
\Crefname{importedtheorem}{Imported Theorem}{Imported Theorems}
\Crefname{importedlemma}{Imported Lemma}{Imported Lemmas}

\makeatletter
\newtheorem*{rep@theorem}{\rep@title}
\newcommand{\newreptheorem}[2]{\newenvironment{rep#1}[1]{\def\rep@title{\Cref{##1}, restated}\begin{rep@theorem}}{\end{rep@theorem}}}
\makeatother
\makeatletter
\newtheorem*{rep@lemma}{\rep@title}
\newcommand{\newreplemma}[2]{\newenvironment{rep#1}[1]{\def\rep@title{\Cref{##1} Restated}\begin{rep@lemma}}{\end{rep@lemma}}}
\makeatother
\makeatletter
\newtheorem*{rep@definition}{\rep@title}
\newcommand{\newrepdefinition}[2]{\newenvironment{rep#1}[1]{\def\rep@title{\Cref{##1}, restated}\begin{rep@definition}}{\end{rep@definition}}}
\makeatother
\makeatletter
\newtheorem*{rep@corollary}{\rep@title}
\newcommand{\newrepcorollary}[2]{\newenvironment{rep#1}[1]{\def\rep@title{\Cref{##1}, restated}\begin{rep@corollary}}{\end{rep@corollary}}}
\makeatother
\newreptheorem{theorem}{Theorem}
\newreplemma{lemma}{Lemma}
\newrepdefinition{definition}{Definition}
\newrepcorollary{corollary}{Corollary}
  \usepackage{graphicx}
\title{Linear algebra at exponential scale \\ via tensor network dimension reduction}
\usepackage{svg}

\author{Chris Camaño\thanks{Department of Computing and Mathematical Sciences, California Institute of Technology, Pasadena, CA 91125 USA 
(\email{ccamano@caltech.edu}, \email{jtropp@caltech.edu})
}
\and
Ethan N. Epperly\thanks{Department of Mathematics, University of California Berkeley, Berkeley, CA 94720 USA 
(\email{eepperly@berkeley.edu})
}
\and
Raphael A. Meyer\thanks{Department of Statistics, University of California Berkeley and International Computer Science Institute, Berkeley, CA 94720 USA 
(\email{ram900@berkeley.edu})
}
\and
Joel A. Tropp\footnotemark[1]
}

\date{\today}

\begin{document}
\maketitle

\begin{abstract}
Many problems in modern scientific computing
are challenging because of a \emph{curse of dimension},
where their mathematical formulation involves objects whose
dimension is \emph{exponential} in the nominal
``size'' of the problem.
Tensor networks can provide a compact representation
for exponentially large vectors and matrices
that arise in applications, but these representations
do not always lead to reliable algorithms.
This paper develops and analyzes techniques
for randomized dimension reduction of tensor
network data. These techniques support a suite of efficient
algorithms for provably solving exponential-scale
linear algebra problems, including trace estimation
and eigenvalue approximation.
The paper includes several stylized illustrations
from quantum many-body physics with ambient dimension
up to $2^{200}$.
\end{abstract}

\section{Introduction}
\RAMdone
\ENEdone 
\CCdone

At the frontier of scientific computing, we encounter problems
that suffer from a \emph{curse of dimension}, where direct methods would require us to instantiate objects that scale exponentially with the nominal ``size'' of the problem \cite{e2021dawning}.
For instance,
\begin{enumerate}
    \item In high-dimensional function approximation, we may need to discretize a function $f : [0,1]^n \to \bbR$.
    Using $d$ points per spatial dimension, the discretization
    requires $d^n$ pieces of data.
    \item In probability theory, the joint probability distribution of $n$ random experiments, each with $d$ possible outcomes, is described by $d^n$ parameters.
    \item In quantum many-body physics, we represent the joint quantum state of $n$ interacting subsystems, each with local dimension $d$, by means of a vector in a Hilbert space with dimension $d^n$. \end{enumerate}
When $n$ is large, any algorithm that manipulates
$d^n$ parameters is doomed.
To address this challenge, we must design algorithms
that work with compact representations
of the problem data, the problem solution, and the intermediate objects
that arise during computations.

For each of the three examples above,
we can sometimes circumvent the exponential scaling
by employing \emph{tensor networks}
\cite{orus14,BR17,Rakhshan26,ahle24} to represent information. A tensor network models an exponentially large vector or matrix as the \emph{contraction} of several tensors of small order; see \cref{subsec:background} for background. Tensor networks have become essential tools for processing high-dimensional data in
applications spanning quantum physics \cite{Scho05,Vers08,Evenbly_2009,Schuch_2011,Paeckel_2019}, quantum chemistry \cite{olivares2015ab,Chan16,Zhai_2023,zhai26}, high-dimensional function approximation \cite{Khoromskij2011OdlogNA,Rip21,tind24,Rod24cheb},
machine learning 
\cite{novikov2015tensorizing,Stoudenmire2016,izmailov18,yang24loretta},
and fluid dynamics \cite{kiffner2023tensor,peddinti2024quantum,holscher2025quantum,gourianov2025tensor}.

Although tensor networks provide a compact way of representing exponential-scale data, the cost of performing computations on a tensor network may still be prohibitive. To avoid this dilemma, researchers have proposed methods
for \emph{randomized dimension reduction}
of tensor networks~\cite{daas2021,kressner23streaming,bucci24,saibaba2025improved,cam25,peng26,meng2026recursive,tang25sketchtom,camano25,daas25,cazeaux2026linear}.
This approach applies a randomized embedding
to the high-dimensional tensor network data to produce a low-dimensional summary.
After the embedding, we can process the summary data with
traditional algorithms
to solve a range of computational problems.
This strategy has been very fruitful. Nevertheless, the existing randomized algorithms for tensor networks
often lack end-to-end guarantees, and several popular methods can
fail dramatically~\cite{meyer_hutchinsons_2025,meyer25,cam25}.

This paper advances the theory and practice of
randomized dimension reduction for tensor networks.
Our approach involves a type of random tensor network,
called a \emph{random matrix product state} (\rmps),
that is easily implemented with standard libraries for
tensor network computations.
Rakhshan \& Rabusseau \cite{rakhshan20,rakhshan2022}
introduced \rmps dimension reduction maps
and showed that they preserve the geometry of a finite point set.
We prove through a stronger argument that \rmps dimension reduction preserves the geometry of an entire \emph{subspace} of points, thereby unlocking the application to linear algebra.

Using \rmps dimension reduction, we develop a suite of efficient algorithms that provably
solve exponential-scale linear algebra problems. For matrix data with tensor network structure,
we obtain powerful new techniques for eigenvalue approximation,
trace estimation, and low-rank matrix approximation.
As applications, we demonstrate that these algorithms
can help us solve two problems from quantum many-body physics:
estimating the partition function for a system at finite temperature
and estimating the von Neumann entropy for an open quantum system.

\subsection{Empirical evidence: Randomized linear algebra at exponential scale}
\label{subsec:intro_empirical}

\ENEdone
\CCdone

This section outlines 
a stylized example of an exponential-scale problem that becomes accessible through randomized dimension reduction for tensor networks.
In quantum mechanics, we describe the energetic properties of $n$ interacting subsystems, each with local dimension $d$, by an exponentially large Hamiltonian matrix $\mH \in \bbC^{d^n\times d^n}$.
In typical applications, \mH has a concise description as a tensor network.
The core quantity governing the thermal properties of the system is the \emph{partition function}:
\begin{equation} \label{eq:partition-function}
Z(\beta)=\tr\bigl(\e^{-\beta \mH}\bigr) \quad \text{for given inverse temperature $\beta > 0$.}
\end{equation}
This section illustrates a technique for estimating the partition function~\eqref{eq:partition-function}
using \emph{stochastic trace estimation}.
The numerical experiments focus on the periodic transverse-field Ising model (PTFIM) Hamiltonian $\mH \in \bbR^{2^n\times 2^n}$, which is real-valued and has local dimension $d=2$ at each of $n$ sites; for details, see \cref{eq:ptfim}.

We can approximate the partition function~\eqref{eq:partition-function} using
a tensor network extension of the \XNysTrace algorithm; see \cref{subsec:xnystrace}.
The original matrix implementation of this algorithm stores the Hamiltonian matrix \mH as a dense or sparse array, and it estimates the partition function by approximating the action $\vomega \mapsto \e^{-\beta \mH}\vomega$ of the matrix exponential on dense, unstructured random vectors $\vomega \in \bbR^{d^n}$.
As \cref{fig:rnla_fails_at_scale} demonstrates, the usual approach rapidly becomes infeasible when the tensor order $n$ becomes large.
Merely storing the matrix \mH exceeds estimates of the total storage capacity of all enterprise data centers on Earth \cite{seagate24} once the system size reaches about \(n=40\) for dense storage and about \(n=70\) for sparse storage.

\begin{figure}
    \centering
\includegraphics[width=1\linewidth]{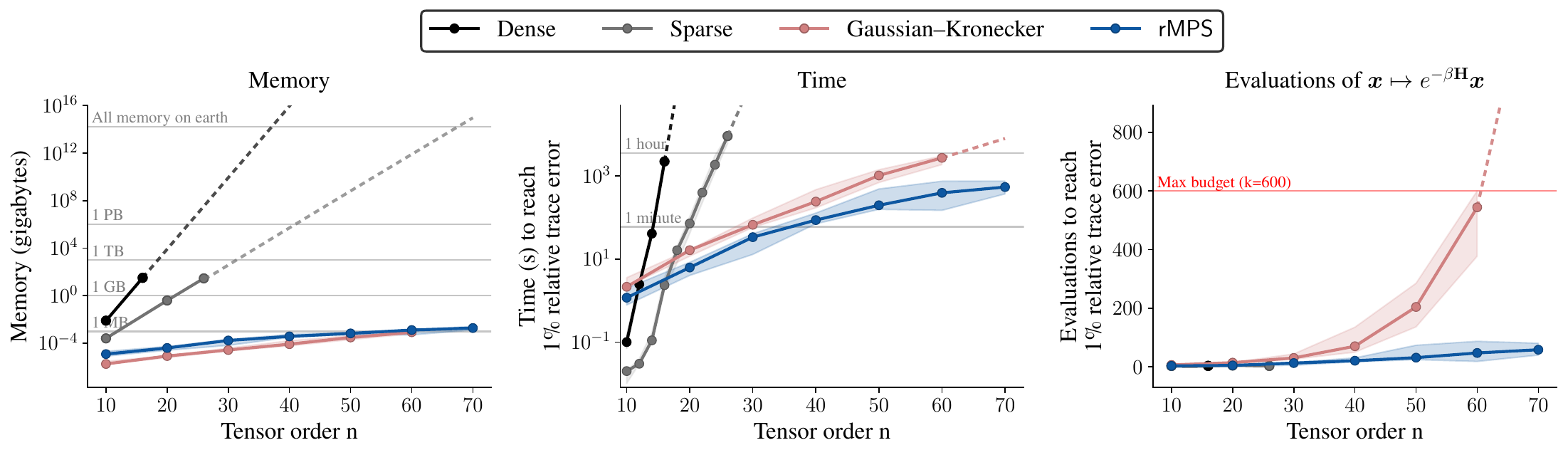}
\caption{\textbf{Exponential-scale linear algebra.}
    Computational resources to estimate the partition function $Z(\beta)=\tr(\e^{-\beta \mH})$ of the PTFIM Hamiltonian \(\mH\in\bbR^{2^n \times 2^n}\) \cref{eq:ptfim} with increasing order $n$
using the \XNysTrace algorithm implemented with dense, sparse (\texttt{CSR}), and tensor network data formats with both \rmps and Gaussian--Kronecker test vectors.
    From left to right, the panels report the memory, wall-clock time, and number of applications of the primitive $\vomega \mapsto \e^{-\beta \mH} \vomega$ required to estimate $Z(\beta)$ to 1\% relative error.
    Markers track the median over 10 trials, and shaded regions bound the $10\%$ and $90\%$ quantiles.
Dense and sparse matrix computations quickly expend the available memory, and
    Gaussian--Kronecker vectors show exponential growth in the number of evaluations.
    \rmps test vectors yield better results, requiring only $9$ minutes and $58$ evaluations to estimate the partition function when $n = 70$.
    See \cref{subsec:intro_empirical}.
}
\label{fig:rnla_fails_at_scale}
\end{figure}

Fortunately, we can do much better by exploiting tensor network structure.
In place of the unstructured random vectors used by the standard \XNysTrace algorithm, we use \rmps vectors (\cref{subsec:rmps}).
In place of the dense or sparse matrix formats, we represent the PTFIM Hamiltonian \mH using the \mpo tensor network format (\cref{subsec:mpo_mps}), and we use an off-the-shelf method known as TDVP1-GSE \cite{YA20} to compute the action $\vomega \mapsto \e^{-\beta \mH}\vomega$ on \rmps vectors.
See \cref{app:experiment_details} for more experimental details.

\Cref{fig:rnla_fails_at_scale} shows the results.
For tensor order $n = 70$, the \rmps-based implementation of \XNysTrace can estimate the partition function to 1\% relative error using megabytes of storage and
minutes
of computation.
By using \rmpss and efficient tensor network methods to store and apply the Hamiltonian \mH, we can extend the \XNysTrace algorithm to perform trace estimation at exponential scale.

As a baseline, we compare against the performance of Gaussian--Kronecker vectors that have the form $\vomega = \vphi^{(1)} \otimes \cdots \otimes \vphi^{(n)}$ with iid Gaussian $\vphi^{(i)} \sim \cN_{\bbR}(\vzero,\Id_2)$.
Gaussian--Kronecker vectors are widely used in the literature \cite{cam25,camano25,meyer_hutchinsons_2025,ahle20,daas25,meng2026recursive,saibaba2025improved,bujanovic25}, and they also enable us to implement \XNysTrace at exponential scale.
Even so, the Gaussian--Kronecker distribution suffers from a fatal flaw:
to achieve the same results as \rmps test vectors, the number of invocations of the map $\vomega \mapsto \e^{-\beta \mH}\vomega$ grows \emph{exponentially} in the tensor order $n$.
In several other contexts, our experimental work confirms the ineptitude of Kronecker-structured vectors for randomized linear algebra tasks; see \cref{subsec:krvecs} for discussion.

\subsection{Background: Tensor networks and matrix product states}
\label{subsec:background}
\RAMdone
\CCdone

The central object of our study is a particular type of tensor network
called either a \emph{matrix product state} (\mps) \cite{Baxter68,Fan92,Ost95,verstraete04mpdo,perez2006matrix}
or a \emph{tensor train} ($\mathsf{TT}$) \cite{Khoromskij2011OdlogNA,ose11}. 
We use the former term in this article.
Algebraically, an \mps is a tensor $\tX \in \bbF^{d\times d\times d\times \cdots \times d} \cong (\bbF^d)^{\otimes n}$ that admits the representation
\begin{equation}
\label{eq:mps}
\tX(i_1,\ldots,i_n)
= \sum_{\alpha_1=1}^{\chi}\cdots\sum_{\alpha_{n-1}=1}^{\chi}
  \tG_1(i_1,\alpha_1)
  \tG_2(\alpha_1,i_2,\alpha_2)\tG_3(\alpha_2,i_3,\alpha_3)\cdots 
  \tG_n(\alpha_{n-1},i_n).
\end{equation}
Here, $\bbF$ denotes the field of real (\(\bbF=\bbR\)) or complex (\(\bbF=\bbC\)) numbers.
The constituent tensors (``cores'') have dimensions \(\tG_1 \in \bbF^{d \times \chi}\), \(\tG_n \in \bbF^{\chi \times d}\), and \(\tG_p \in \bbF^{\chi \times d \times \chi}\) for $1 < p < n$.
The parameter $\chi \ge 1$ is called the \emph{bond dimension} or  the \emph{$\mathsf{MPS/TT}$ rank}, while the parameter $d\ge 1$ is called the \emph{physical dimension} or \emph{local dimension}.
The generalization to non-uniform bond dimensions $\chi_i$ or physical dimensions $d_i$ is possible, but this work does not pursue the extension.

The \mps representation is useful for exponentially large problems because it allows us to store a high-dimensional object with $d^n$ degrees of freedom using only $nd\chi^2$ parameters. 
Provided the bond dimension $\chi$ is small, the \mps representation of \tX can lead to significant reductions
in storage costs. Throughout this article, it is helpful to think of an \mps as either a $d\times\cdots\times d$ tensor \tX or as a vector \vx with dimension $d^n$.
We switch between these perspectives without comment.

To reason about matrix product states and other tensor networks, Penrose's \emph{tensor diagram notation} \cite{pen71} proves to be an immensely useful tool. 
In this notation, each tensor is depicted as a convex body with one edge emanating for each mode of the tensor. 
For instance, an order-3 tensor $\tT = [t_{ijk}]_{i,j,k=1}^d\in \bbF^{d\times d\times d}$
would have three edges, one for each index $i$, $j$, and $k$.
Using this notation, a vector, matrix and tensor may be represented as follows:
\vspace{-0.5em}
\begin{equation*}
    \tnshow{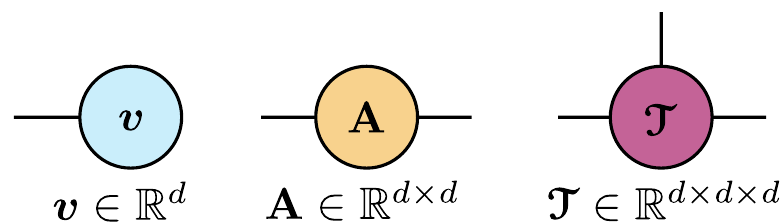}{.5} 
\end{equation*}

\begin{table}[t]
\centering
\caption{\textbf{(Penrose tensor diagrams)}. Examples of tensor diagrams for basic linear algebra operations.}

\renewcommand{\arraystretch}{1.35}
\setlength{\tabcolsep}{10pt}
\begin{tabular}{@{} l c c @{}}
\toprule
Problem & Algebraic formula & Tensor diagram \\
\midrule
Vector inner product
& $\langle \vu,\vv \rangle = \sum_{i=1}^d \bar{u_i} v_i$
& $\hspace*{-0.8em}\tnshow{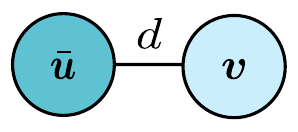}{0.177}$ \\[0.5ex]

Matrix--vector product
& $\mA\vv = \left[\sum_{j=1}^d a_{ij} v_j\right]_{i=1}^d$
& $\hspace*{-2.3em}\tnshow{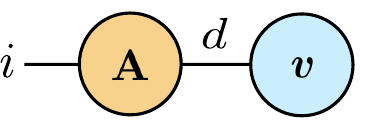}{0.2145}$ \\[0.5ex]

Matrix--matrix product
& $\mA\mB = \left[\sum_{k=1}^d a_{ik} b_{kj}\right]_{i,j=1}^d$
& $\hspace*{-0.6em}\tnshow{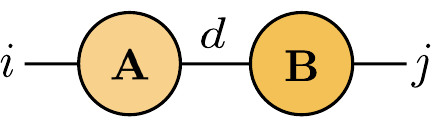}{0.25}$ \\[0.5ex]

Matrix Kronecker product
& $\mA \otimes \mB = \left[a_{ij} b_{k\ell}\right]_{i,j,k,\ell=1}^d$
& $\hspace*{-.73em}\tnshow{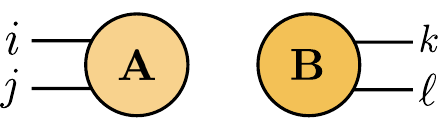}{0.255}$ \\[0.5ex]

\bottomrule
\end{tabular}
\label{tab:penrose-basics}
\end{table}

The next most fundamental operation in tensor diagram notation is \emph{contraction}, where two tensors are merged by multiplying the tensor elements and summing along one or more contracted edges.
Given an order-$t$ tensor $\tA \in (\bbF^d)^{\otimes t}$, an order-$s$ tensor $\tB \in (\bbF^d)^{\otimes s}$, and $r$ pairs of edges to contract, the result is a $(t+s-2r)$-mode tensor $\tC \in (\bbF^d)^{\otimes(t+s-2r)}$ with entries

\begin{equation*}
    \tC(\alpha_1,\ldots,\alpha_{t-r},\beta_1,\ldots,\beta_{s-r})
    = \sum_{\gamma_1=1}^{d}\cdots\sum_{\gamma_r=1}^{d} \tA(\alpha_1,\ldots,\alpha_{t-r},\gamma_1,\ldots,\gamma_r)\tB(\beta_1,\ldots,\beta_{s-r},\gamma_1,\ldots,\gamma_r).
\end{equation*}
In tensor diagram notation, contraction is denoted by joining the edges for the contracted modes:
\begin{equation*}
    \tnshow[2ex]{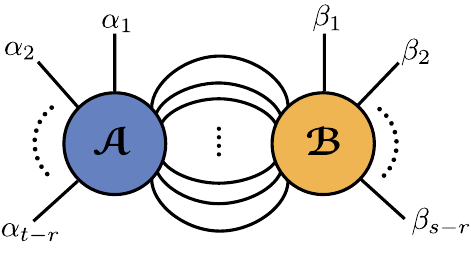}{.32}
    =\tnshow[2.1ex]{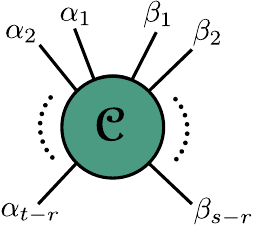}{.183} 
      \vspace{-1em}
\end{equation*}
In the trivial case when $r = 0$, no edges are drawn between the tensor bodies and no modes are contracted, resulting in the tensor product $\tC = \tA \otimes \tB$.
Throughout this article, we may label edges in a tensor diagram using either an index label (e.g., 
$i$ or $\alpha$) or the dimension (e.g., $d$ or $\chi$).
Index labels always emanate from an edge, while dimension labels appear adjacent to an edge.
\Cref{tab:penrose-basics} shows how to express several familiar operations from linear algebra diagrammatically as tensor contractions.
As another example, tensor contractions allow us to express the \mps representation \cref{eq:mps} graphically:
\begin{equation}
\label{eq:mps_diagram_intro}
  \tX = 
  \tnshow[5.9ex]{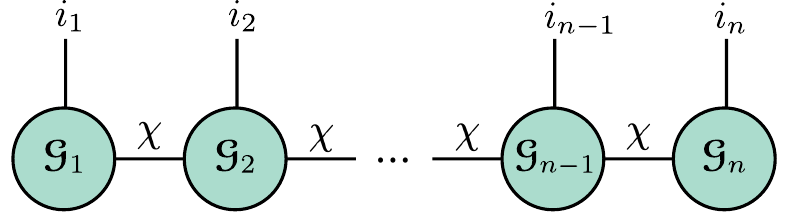}{.525}
  \in \bbF^{d^n}.
  \vspace{-2.4em}
\end{equation}
Going forward, we appeal to tensor diagrams to simplify the presentation of otherwise unwieldy tensor contractions.
See \cite[$\mathsection$3.83]{Ballard25},\cite[$\mathsection$2]{BR17}, or \cite{Rakhshan26,taylor24,ahle24} for additional guidance on using tensor diagram notation for representing tensors and operations on them.

\subsection{Random matrix product states}
\label{subsec:rmps}
\RAMdone

Randomized dimension reduction offers a powerful tool
for solving large-scale linear algebra problems.
In a typical algorithm, we collect information about a large matrix $\mA \in \bbF^{N\times N}$ by computing its action $\mA\vomega_1,\ldots,\mA\vomega_{\ell}$ on a small family of random probe vectors $\vomega_1,\ldots,\vomega_{\ell}$.
It is common to choose the probe vectors iid from the standard Gaussian distribution $\cN_{\bbF}(\vzero,\Id_N)$, and we often aggregate the probe vectors into a \emph{test matrix} $\mOmega = k^{-1/2}[\vomega_1,\ldots,\vomega_{k}] \in \bbF^{N\times k}$.

We plan to design randomized linear algebra algorithms
for matrices
with exponential dimension $N = d^n$ that enjoy tensor network structure.
To that end, we ask:
\actionbox{\centering What is the ``right'' analogue of a standard Gaussian vector for tensor network computations?}

This paper proposes the \emph{random matrix product state} as a natural choice.  
Random matrix product states were first studied systematically in physics \cite{garnerone10statistical,garnerone10typicality,haag23typical,lancien22correlationlength,lami25anticoncentration,lami25clifford}, and they were first used for dimension reduction
by Rakhshan \& Rabusseau \cite{rakhshan20,rakhshan2022}.

\begin{definition}[\rmps and \rmps test matrix] \label{def:rmps}
    Fix the scalar field $\bbF$, the physical dimension $d$,
    the bond dimension $\chi$, and the tensor order $n$.
A \textit{random matrix product state} (\rmps) is an \mps, as in \eqref{eq:mps} or~\eqref{eq:mps_diagram_intro},
    whose cores $\tG_p$ have independent Gaussian elements drawn from
$\cN_{\bbF}(0,\sigma_p^2)$.
    When $1 < p < n$, the variance $\sigma_p^2 = 1/\chi$;
    when $p = 1$ or $p = n$, the variance $\sigma_p^2 = 1/\sqrt{\chi}$.
This \rmps vector is an element of $\bbF^{d^n}$.

    An \emph{\trp} is a random matrix of the form
    $$
    \mOmega \coloneqq \frac{1}{\sqrt{k}}
    \begin{bmatrix}
    \mid & \mid &        & \mid  \\
    \vomega_1 & \vomega_2 & \cdots & \vomega_k \\
    \mid & \mid &        & \mid  
    \end{bmatrix} = \frac{1}{\sqrt{k}}
    \begin{bmatrix}
    \tnshow{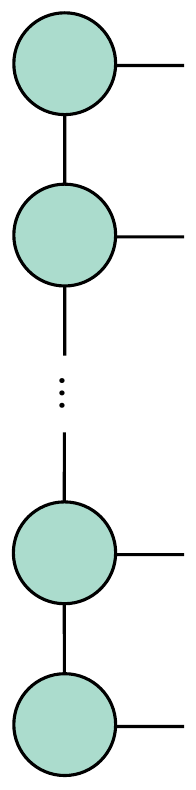}{.025} & \tnshow{rmpsdemo}{.025} & \cdots & \tnshow{rmpsdemo}{.025}\\
    \end{bmatrix}
    \in \bbF^{d^n \times k}.$$
The columns $\vomega_1,\ldots,\vomega_k \in \bbF^{d^n}$ are iid \rmps vectors.  By construction, the \trp is isotropic; see~\eqref{eq:isotropy}.
\end{definition}

To ensure that \rmps vectors and test matrices can serve in place
of their Gaussian counterparts,
how should we set the bond dimension $\chi$?
The main thesis of this paper can be summarized as follows:
\actionbox{\centering \textbf{Thesis (\emph{informal}).} \rmps vectors and test matrices work reliably in randomized \\ linear algebra algorithms if and only if the bond dimension $\chi \gtrsim n$.}
\noindent 
\Cref{subsec:fourth_moment_comparison,sec:rmps-moments}
assign a precise meaning to this thesis. What is its significance?
When $\chi \asymp n$, storing an \rmps vector requires only $\order(dn^3)$ numbers, dramatically less than the $\order(d^n)$ cost of storing a standard Gaussian vector.
Therefore, we recognize an opportunity to solve exponential-scale
linear algebra problems by designing algorithms that employ \rmps probe vectors. 

When $n$ is large---as is typical in quantum physics applications---the cost of storing an \rmps may still be burdensome.
Unfortunately, our analysis also shows that $\chi \gtrsim n$ is \emph{necessary} for \rmps vectors
to exhibit good performance,
at least in the worst case.
At the other extreme,
an \rmps with $\chi = 1$ is a Gaussian--Kronecker vector,
and
\Cref{fig:rnla_fails_at_scale} has already warned us that
Gaussian--Kronecker vectors behave poorly in some cases;
see \cref{subsec:krvecs} for discussion.

\subsection{Core aspects of the theoretical analysis}
\label{subsec:fourth_moment_comparison}

We can employ an \rmps test matrix \mOmega as a dimension reduction map
to compress a high-dimensional vector $\vx \in \bbF^{d^n}$
to a low-dimensional vector $\mOmega^*\vx \in \bbF^k$.
On average, the \rmps test matrix preserves the squared length of an arbitrary vector:
\begin{equation} \label{eq:isotropy}
    \E[\norm{\mOmega^*\vx}_2^2] = \norm{\vx}_2^2
    \quad\text{for each $\vx \in \bbF^{d^n}$.}
\end{equation}
This property is known as \emph{isotropy}.
The expression~\eqref{eq:isotropy} holds because
the covariance matrix of an \rmps \vomega equals
the identity matrix: $\E[\vomega\vomega^*] = \Id$.
Note that the covariance matrix tabulates all of the \emph{second-order moments} of the \rmps distribution.

To understand the dimension reduction properties of
an \rmps test matrix, it suffices to compute \emph{fourth-order moments} of an \rmps vector $\vomega$~\cite{tropp25,cam25},
which are collected in the \emph{fourth-moment tensor} $\tM \coloneqq \E[\vomega\vomega^* \otimes \vomega\vomega^*]$.
The main technical result of this paper is an exact formula for the fourth-moment tensor of an \rmps.
This formula requires some preparation, so we reserve the details until \cref{ssec:rmps_fourth_moment,sec:complex}. We can summarize the analysis as follows.

\actionbox{
    When the bond dimension \(\chi \ll n\), the fourth-moment tensor of a \rmps differs sharply from the fourth-moment tensor of a standard Gaussian vector. As \(\chi/n \to \infty\), the fourth-moment tensor of an \rmps approaches
    the fourth-moment tensor of a standard Gaussian vector.
}
\noindent 
This calculation substantiates the claim from \cref{subsec:rmps} that bond dimension $\chi \gtrsim n$ is \emph{necessary and sufficient} for an \rmps vector to serve in place of a standard Gaussian vector.

As an example, consider the task of estimating the trace of an exponentially large matrix $\mA \in \bbF^{d^n\times d^n}$ that
we access through the matrix--\mps product $\vomega \mapsto \mA \vomega$.
To estimate the trace, we can generate an \rmps and evaluate the quadratic form $\vomega^* (\mA\vomega)$.
The isotropy property $\E[\vomega\vomega^*] = \Id$ shows that this quadratic form provides an unbiased estimator for the trace:
\begin{equation*}
    \E[\vomega^* (\mA \vomega)] = \tr(\mA).
\end{equation*}
The fourth-moment formula for an \rmps
(\cref{ssec:rmps_fourth_moment,sec:complex})
yields a bound on the variance of the
quadratic form trace estimate:

\begin{theorem}[Variance of \rmps quadratic forms]
\label{thm:rmps_variance_intro}
    Let $\mA \in \bbF^{d^n\times d^n}$ be any square matrix, and let $\vomega$ be an \rmps with bond dimension $\chi$.
    When $\bbF = \bbR$, the variance of the quadratic form trace estimator satisfies
\begin{subequations} \label{eq:rmps-variance}
    \begin{equation}
        \Var(\vomega^\top\mA\vomega) \le 2 \left(1 + \frac{1}{\chi}\right)^{n-1} \norm{\mA}_{\rm F}^2 + \left[ 3\left(1 + \frac{2}{\chi}\right)^{n-1} - 2\left(1 + \frac{1}{\chi}\right)^{n-1} - 1\right] \norm{\mA}_*^2. \end{equation}
When $\bbF=\bbC$,
    \begin{equation}
        \Var(\vomega^*\mA\vomega) \le   \norm{\mA}_{\rm F}^2 + 2\left[ \left(1 + \frac{1}{\chi}\right)^{n-1} - 1\right] \norm{\mA}_*^2.
    \end{equation}
    \end{subequations}
These bounds are attained when $\mA$ is the matrix of all ones.
\end{theorem}

\begin{figure}[t]
    \centering
    \includegraphics[width=1\linewidth]{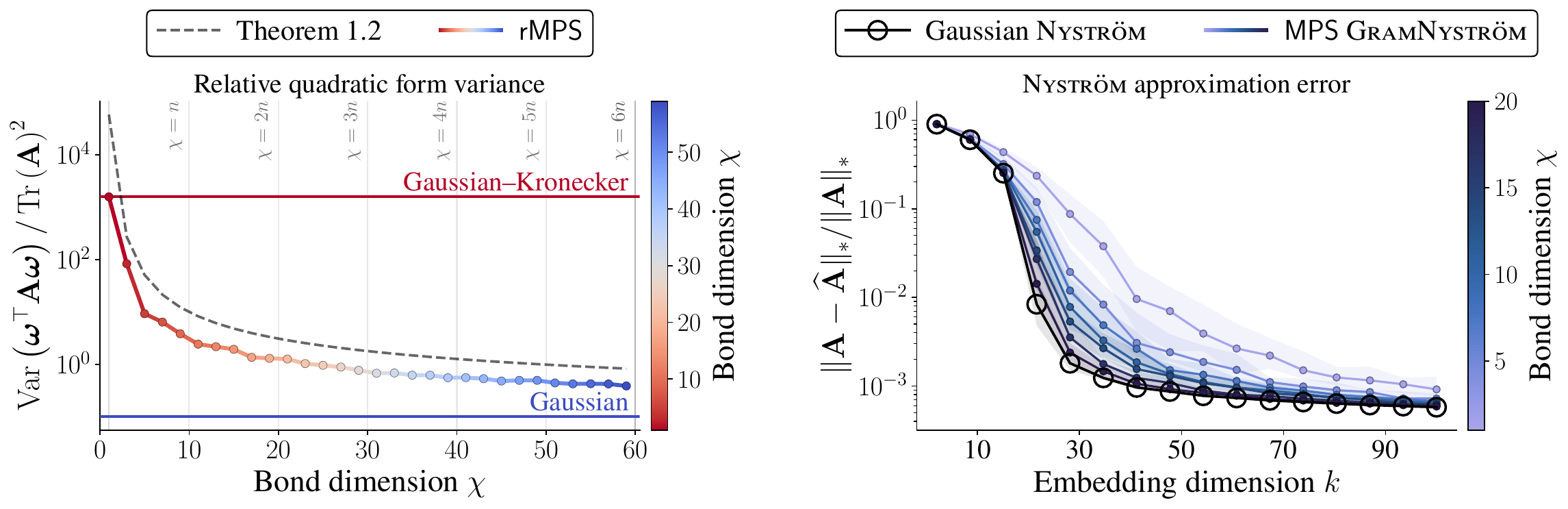}
    \caption{\textbf{Convergence of \rmpss to Gaussian behavior.}
    (\textit{Left})
    Normalized variance of quadratic form trace estimator \(\Var(\vomega^*\mA\vomega)/\tr(\mA)^2\) evaluated across \rmps bond dimension \(\chi\in[1,100]\).
    As the \rmps bond dimension \(\chi\) approaches \(n\), the variance of the \rmps quadratic form converges to that of a dense Gaussian random vector. (\textit{Right})
    Relative error \(\snorm{\mA-\hat\mA}_* / \snorm{\mA}_*\) for a randomized Nystr\"om approximation $\widehat{\mA}$
    as a function of approximation rank. In both panels, the matrix is chosen adversarially: \(\mA=\mQ\mQ^*+10^{-5}\mI\in\bbR^{2^{10}\times2^{10}}\)
    where \mQ consists of the first 20 columns of the orthogonal Walsh--Hadamard matrix.  Sample variances are computed using 500 random samples;
    the markers track the median over 20 trials; shaded regions are bounded by the $10\%$ and $90\%$ quantiles.
See \cref{subsec:fourth_moment_comparison}.
}
    \label{fig:convergence-to-gaussian}
\end{figure}

It is instructive to compare the variance bounds from \cref{thm:rmps_variance_intro}
with the variance of the quadratic form trace estimate
$\vg^*(\mA\vg)$ based on a standard Gaussian vector
\(\vg\sim \cN_{\bbF}(\vzero,\Id)\).
For an Hermitian matrix \(\mA\),
\begin{equation*}
    \E[ \vg^* \mA \vg] = \tr(\mA) \quad\text{and}\quad
    \Var(\vg^* \mA\vg) = \rK_\bbF\|\mA\|^2_{\rm F}
\end{equation*}
with $\rK_{\bbR} = 2$ and $\rK_{\bbC} = 1$.
For large enough $\chi\gtrsim n$,
the bounds \cref{eq:rmps-variance} simplify to
\[
    \Var(\vomega^\top \mA \vomega)
    \le
    \left(\rK_{\bbF}+\cO\!\left(\frac n\chi\right)\right)\|\mA\|_{\rm F}^2
    +
    \cO\!\left(\frac{n}{\chi}\right)
    \bigl(
        \|\mA\|_{\rm F}^2+\|\mA\|_*^2
    \bigr).
\]
In other words,
the variance of a \rmps quadratic form $\vomega^*\mA\vomega$ differs from its Gaussian counterpart $\vg^* \mA \vg$ by a factor that vanishes as $\chi/n \to \infty$.
See the left panel of \cref{fig:convergence-to-gaussian}
for numerical evidence.

The variance bound for quadratic forms (\cref{thm:rmps_variance_intro}) also yields error bounds for approximate least-squares regression and matrix low-rank approximation with \rmps dimension reduction; see \cref{sec:quadratic-form-condition}. As a numerical example,
the right panel of \cref{fig:convergence-to-gaussian} shows that the randomized Nystr\"om algorithm for low-rank approximation (\cref{sec:nystrom}) exhibits similar performance when implemented with Gaussian test matrices or with \rmps test matrices
that have sufficient bond dimension. 

\subsection{Roadmap}

\Cref{sec:related} summarizes related work.
\cref{sec:algs-applications} explains how to implement randomized linear algorithms at exponential scale using \rmps test matrices, and it evaluates these algorithms on stylized problems from quantum many-body physics.
\Cref{sec:rmps-moments} calculates the fourth-moment tensor of an \rmps
vector, and it derives variance bounds for quadratic form induced by an \rmps.
\Cref{sec:quadratic-form-condition} shows how the variance bound
supports the mathematical analysis of randomized algorithms based
on \rmps test matrices.

\subsection{Notation} \label{sec:notation}
\RAMdone
We work over a scalar field $\bbF$, which will be either $\bbR$ or $\bbC$.
A random variable from the complex normal distribution $\cN_{\bbC}(0,\sigma^2)$ takes the form $g_1+\iu g_2$ for independent $g_1,g_2 \in \cN_{\bbR}(0,\sigma^2/2)$.
The variance of a complex random variable is $\Var(x) = \E[|x|^2] - |\E[x]|^2$.
To state some results in generality, we define constants $\rK_{\bbR} = 2$ and $\rK_{\bbC} = 1$.
The adjoint of a matrix $\mA$ is written $\mA^*$.
We write \(\|\mA\|_{(p)}\) for the Schatten \(p\)-norm of a matrix $\mA$.
The Frobenius and nuclear norms are \(\norm{\mA}_{\rm F} \defeq \norm\mA_{(2)}\) and \(\norm{\mA}_{*} \defeq \norm\mA_{(1)}\).
The symbol $\lowrank{\mA}_r$ denotes an Eckart--Young--Mirsky best rank-$r$ approximation of the matrix $\mA$,
which is optimal simultaneously for all Schatten norms
but may not be uniquely determined.

We use the shorthand $[n]\coloneqq\{1,\ldots,n\}$.
Scalars are written in italic, such as $a\in\bbF$.  
Vectors are in bold lowercase,
while matrices are in bold uppercase,
such as $\vx\in\bbF^{n}$ or $\mX\in\bbF^{m \times n}$.
Tensors are written in a calligraphic font: $\tT \in \bbF^{d\times \cdots \times d}$.
Entries of vectors, matrices, and tensors
are expressed with functional notation
$\tT(i,j,k)$ or with subscripts $t_{ijk}$.
Throughout this article, we are flexible in treating a tensor $(\bbF^d)^{\otimes n}$ as either a $d\times \cdots \times d$ array or a vector of length $d^n$. 

For nonnegative quantities \(a,b\), we write \(a\lesssim b\) if \(a\leq \rC b\) for a universal constant \(\rC>0\). We write \(a\gtrsim b\) if \(b\lesssim a\), and \(a\asymp b\) if \(\rC_1b\leq a\leq \rC_2b\) for universal constants \(0<\rC_1\leq \rC_2<\infty\). \section{Related work} \label{sec:related}
\ENEdone
\CCdone

Dimension reduction for tensor networks is a young but expanding topic of interest.
This section compares several approaches that have been proposed in the literature.

\subsection{Random Kronecker vectors and Khatri--Rao dimension reduction}
\label{subsec:krvecs}
\ENEdone
\CCdone
The simplest construction for a structured high-dimensional random vector is the random Kronecker vector $\vomega = \vomega_1 \otimes \cdots \otimes \vomega_n \in \bbF^{d^n}$, obtained by taking the tensor product of independent random vectors $\vomega_i \in \bbF^d$, each drawn from an isotropic distribution, such as the standard Gaussian distribution.
To obtain a test matrix $\mOmega \in \bbF^{d^n\times k}$, we can stack random Kronecker vectors columnwise.
The resulting object is called a Khatri--Rao test matrix.
Random Kronecker vectors and Khatri--Rao test matrices are widely studied and deployed in algorithms \cite{cam25,camano25,daas25,feldman22,saibaba2025improved,bujanovic25,bujanovic2021norm,sun15,chen21,ahle20,Khanna_2018,meyer_hutchinsons_2025}.
They are closely connected to the present work because an \rmps with bond dimension $\chi = 1$ is a Gaussian--Kronecker vector,
and an \rmpstest with $\chi = 1$ is a Khatri--Rao test matrix.

Unfortunately, while random Kronecker vectors are easy to generate and manipulate, they suffer from a severe computational limitation known as \emph{overwhelming orthogonality} \cite{meyer25,cam25}.
For each tensor-structured vector $\vx = \bigotimes_{i=1}^n \vx_i$, the squared inner product is exponentially small
\begin{equation*}
    |\langle \vx,\vomega\rangle|^2 \le \rC^{-n} \norm{\vx}^2_2 \quad \text{for some universal constant } \rC > 1
\end{equation*}
except with exponentially small probability.  
As a consequence, randomized algorithms based on Kronecker vectors sometimes require resources that are exponential in the tensor order $n$.  This issue afflicts algorithms for trace estimation \cite{meyer_hutchinsons_2025,meyer25}, dimension reduction \cite{ahle20}, and other problems \cite{meyer25}.

Random Kronecker vectors sometimes perform well in practice despite their exponentially bad worst-case behavior \cite{cam25,camano25,saibaba2025improved,daas25}.
However, as the examples in this paper demonstrate, the deficiencies of random Kronecker vectors manifest in many applications; see the numerical experiments documented in \cref{fig:rnla_fails_at_scale,fig:convergence-to-gaussian,fig:GH_trace_estimation,fig:nystrom,fig:trace-comparison}.
This paper advocates for random \mpss with bond dimension $\chi > 1$ as a more powerful alternative to random Kronecker vectors.

\subsection{\rmpss: History and existing analysis} \label{sec:rakhshan-rabusseau}
\CCdone

Random matrix products have been used for decades in computational quantum many-body physics to initialize iterative algorithms like DMRG.
The statistical properties of random matrix product states were first systematically studied by researchers in theoretical quantum information and quantum many-body physics \cite{garnerone10statistical,garnerone10typicality,haag23typical,lancien22correlationlength,lami25anticoncentration,lami25clifford}.
These researchers study different random \mps constructions that have unit norm and, similar to us, provide evidence that 
a random matrix product state ``behaves randomly''
when the bond dimension that scales polynomially
in the system size.
Early works \cite{garnerone10statistical,garnerone10typicality} showed that $\chi \gtrsim n^2$ is sufficient for random matrix product states to achieve an equidistribution property known as \emph{quantum typicality}; the recent paper \cite{lami25anticoncentration} shows that $\chi \asymp n$ is sufficient for the \emph{inverse participation ratio} $\E[\sum_{i=1}^{d^n} |\omega(i)|^{2k}]$ to be close to the value for a uniformly random unit vector.
To the best of our understanding, the results in the physics literature are not sufficient to analyze randomized matrix algorithms for random \mpss with $\chi \asymp n$.

The \trp construction (\cref{def:rmps}) was introduced by Rakhshan \& Rabusseau \cite{rakhshan20,rakhshan2022} under the name \emph{tensorized random projection}.
They use \trps to embed high-dimensional, tensor-structured data into a lower dimensional space while approximately preserving distances (as in the Johnson--Lindenstrauss lemma).
For the analysis, they proved that real \rmpss satisfy the variance bound
\begin{equation} \label{eq:rr-bound}
    \Var\bigl((\vomega^\top \vx)^2\bigr) \le \left[3 \left( 1 + \frac{2}{\chi}\right)^{n-1} - 1\right] \norm{\vx}_2^4.
\end{equation}
The bound~\eqref{eq:rr-bound} also follows from \cref{thm:rmps_variance_intro} with the matrix $\mA = \vx\vx^\top$.
Our research extends Rakhshan \& Rabusseau's work in three  directions:
\begin{enumerate}
    \item We provide an exact computation of the fourth-moment tensor of an \rmps vector $\vomega$.
    This result leads to sharper bounds for the variance of the quadratic form $\vomega^\top \mA\vomega$.
    The improvement is essential for analyzing trace estimators (\cref{sec:trace-guarantees,sec:nystrompp-analysis}).
    \item We exploit our variance bound (\cref{thm:rmps_variance_intro}) to prove that \trps are \emph{subspace injections} (\cref{sec:osi}).  This result implies rigorous error bounds for low-rank approximation and least-squares regression algorithms implemented with \trps.
    \item We apply algorithms based on \trps to stylized problems in quantum many-body physics.
\end{enumerate}
In addition, our work adopts a technical approach distinct from that of Rakhshan \& Rabusseau.
Our analysis relies on tensor diagram manipulations and a probabilistic interpretation of the fourth-moment tensor.
Rakhshan \& Rabusseau, by contrast, develop an argument based on matricization and iterated computation of expectations.
The two strategies yield complementary insights.

\subsection{Sketching for tensors and tensor networks}
\CCdone
Numerous research articles propose and analyze dimension reduction maps for tensor-structured data,
and they investigate how these maps can accelerate tensor and tensor network computations; for example, see \cite{pham13,larsen2022,bharadwaj23,bharadwaj2024,ma22,ahle20,daas2021,cazeaux2026linear,sun15,camano25,daas25,kressner23streaming,bucci24,saibaba2025improved,meng2026recursive,tang25sketchtom,cam25}.
Some notable approaches include the TensorSketch \cite{pham13}, leverage score sampling \cite{larsen2022,bharadwaj23,bharadwaj2024}, and tree tensor network sketching \cite{ma22,ahle20,YZK25}.

A major technique similar to ours is known as \emph{tensor train sketching} \cite{daas2021,Hur23}.
Here,
the test matrix is chosen to be a single \emph{block \mps} (also known as a block tensor train), which has an additional leg of dimension $k$ on the last core:
\begin{equation} \label{eq:tt-sketching}
    \mOmega =\tnshow[5.4ex]{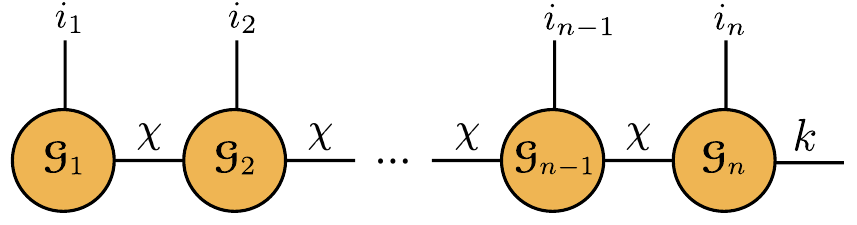}{.525}
      \vspace{-2em}
\end{equation}
That is, tensor train sketching uses exactly one tensor network to form the entire test matrix.
In contrast, the \rmpstest is a collection of \(k\) independent \rmps tensor networks.
Tensor train sketching has been used to accelerate \mps/$\mathsf{TT}$ compression \cite{daas2021}, \mps/$\mathsf{TT}$ construction from sparse data \cite{Hur23}, and quantum Monte Carlo \cite{YZK25}.
Analysis of this construction appears in the concurrent work of Cazeaux et al.\ \cite{cazeaux2026linear}.

Several existing works use random tensor networks
for dimension reduction and obtain provable guarantees for downstream linear algebra tasks,
such as regression and low-rank approximation.
The paper of Ahle et al.\ \cite{ahle20}
studies dimension reduction for
polynomial kernel matrices in the context of machine learning.
Their construction uses a \emph{random tensor tree} in place of an \rmps.
Their theoretical results are incomparable with ours because they only consider dimension reduction for Kronecker-structured vectors \(\vx^{\otimes n}\),
but they develop stronger ``subspace embedding'' guarantees.
Our work has a further advantage because the \rmps construction is compatible with the wide set of software libraries for \mps/$\mathsf{TT}$ computations.

We also mention the concurrent research of Cazeaux et al.\ \cite{cazeaux2026linear}, who consider a new ``$\mathsf{TT}$Stack'' random test matrix obtained by concatenating many block \mpss \cref{eq:tt-sketching}.
Similar to our work, they conclude that bond dimension \(\chi \gtrsim n\) is required for several randomized linear algebra tasks, and they prove variance bounds for $\mathsf{TT}$Stacks that are similar to \cref{thm:rmps_variance_intro}.
The $\mathsf{TT}$Stack and \rmps test matrices have complementary applications; 
we leave a full empirical and theoretical comparison to future work.

 \section{Algorithms and applications} \label{sec:algs-applications}

This section presents algorithms that use \rmpss to solve linear algebra problems that arise in scientific computing. For each method, we develop implementations specialized to the tensor network setting, discuss applications in quantum many-body physics, and provide theoretical and experimental results.

\Cref{subsec:black_box,subsec:mps-artihmetic} formalize the computational environment and the basic operations we perform on \mpss.
In \cref{subsec:trace_estimation}, we discuss the \mps Girard--Hutchinson trace estimator.
Then, in \cref{sec:nystrom}, we introduce the \mps \GramNystrom algorithm low-rank approximation of a psd matrix, with applications to computing the top eigenvalues and von Neumann entropy of a reduced density matrix.
\Cref{sec:nystrompp} then discusses variance-reduced trace estimators (\mps \nystrompp and \mps \XNysTrace).

\actionbox{
\begin{center}
Code for all experiments and supporting tensor network algorithms can be found at \\
\url{https://github.com/chriscamano/TNrandNLA}
\end{center}
}

\subsection{Design methodology: Black-box access by matrix--\mps products}
\label{subsec:black_box}

Many standard linear algebra algorithms \cite[chs.~10--11]{GoVa13} interact with a matrix \mA only through the matrix--vector product operation $\vx \mapsto \mA\vx$.
The matrix--vector paradigm is powerful because it allows us to design algorithms that operate efficiently whenever we have a mechanism for computing \(\vx \mapsto \mA\vx\) quickly. 

Throughout this section, we consider an exponentially large matrix \(\mA\in\bbF^{d^n \times d^n}\) that is compactly represented and supports the efficient computation of \emph{matrix--\mps products}:

\actionbox{
    \textbf{Matrix--\mps access model.}
We interact with the matrix  $\mA \in \bbF^{d^n\times d^n}$ only through a subroutine (``oracle'') whose input is an \mps $\vx \in \bbF^{d^n}$ of bond dimension at most $\chi$ and whose output is $\mA\vx$, represented as an \mps of bond dimension at most $\bar\chi$.
}
\noindent
Matrix--\mps products account for the majority of the computational cost in all our algorithms.
As such, we view the matrix--\mps product as our computational atom, and we design algorithms with the goal of minimizing the number of products required. 

To simplify the presentation, this section assumes that the matrix \mA is square and self-adjoint ($\mA^* = \mA$).
Extensions to rectangular and non-self-adjoint matrices are often possible \cite{woodruff2014sketching,martinsson20randomized,cam25}; see \cref{sec:more-algs}.
Many of our algorithms involve an \emph{\mps-column matrix}, a tall matrix $\mX \in \bbF^{d^n\times k}$ whose columns \(\vx_1,\ldots,\vx_k\) are \mpss.
Given matrix--\mps access to \mA, we can apply it to an \mps-column matrix by applying the matrix column-wise:
\begin{equation*}
    \mA\mX = \bmat{ \mA \vx_1 & \cdots & \mA \vx_k }.
\end{equation*}
In particular,  when $\mOmega \in \bbF^{d^n \times k}$ is an \rmpstest, we can form $\mA\mOmega$ using \(k\) matrix--\mps accesses.
The next two subsections describe settings where this access model is appropriate.
\subsubsection{Access model, example 1: Applying a matrix product operator}
\label{subsec:mpo_mps}

\ENEdone \CCdone \RAMdone
The \mps representation provides a compact expression for an exponentially large vector $\vx \in \bbF^{d^n}$.
Analogously, the matrix product \emph{operator} (\mpo) representation \cite{verstraete04mpdo,Ose10} allows us to represent an exponentially large matrix $\mA \in \bbF^{d^n \times d^n}$.
The tensor diagram for an \mpo is
\begin{equation*}
    \mA =\tnshow[-2.5ex]{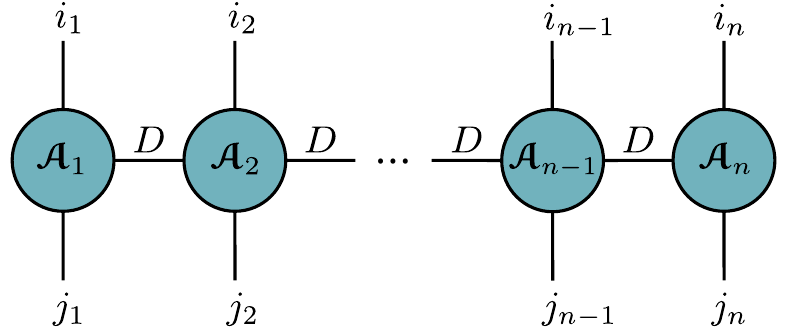}{.45}\in \bbF^{d^n\times d^n}.
\end{equation*}
As with an \mps, the horizontal indices in an \mpo are called \emph{bond indices}, and their dimension is the \mpo bond dimension $D$.
If \mA is an \mpo of bond dimension $D$ and \vx is an \mps of bond dimension $\chi$, then the \mpo--\mps product $\mA\vx$ is an \mps of bond dimension $\bar\chi \le D\chi$.
There are many algorithms for evaluating the \mpo--\mps product \cite{verst04,Stoud10,ose11,ma2024,camano25,mil2026} and compressing it to an \mps of near-minimal bond dimension; see \cite[\S4.4]{camano25} for an overview and comparison.

\subsubsection{Access model, example 2: Imaginary time evolution}
\label{subsubsec:time_evo}

\ENEdone\RAMdone\CCdone
An important operator in quantum many-body physics is the \emph{imaginary time evolution operator} $\mA = \e^{-\tau \mH}$ associated with a self-adjoint \emph{Hamiltonian operator} $\mH \in \bbF^{d^n\times d^n}$.
Its action $\mA \vx = \e^{-\tau \mH}\vx$ encodes the solution $\vx(t)$ to the time-dependent Schr\"odinger equation 
\begin{equation*}
    \iu\frac{\d}{\d t}\vx(t)=\mH\vx(t) \quad \text{with initial condition }\vx(0)=\vx
\end{equation*}
at the \emph{imaginary} time value $t = -\iu \tau$.
There is a mature collection of tensor network algorithms that implement the action of the imaginary time evolution operator on an \mps \cite{verstraete04mpdo,White_2004,Dargel_2012,Zaletel_2015,Haegeman_2011,Haeg16,Paeckel_2019}.
In this work, we use the GSE-TDVP1 algorithm\cite{Haegeman_2011,Haeg16,YA20}.

\subsection{Arithmetic with \mpss}
\label{subsec:mps-artihmetic}

To solve exponentially large problems with tensor structure, this paper extends several algorithms originating in the matrix setting.
Unfortunately, many standard operations on matrices become expensive on \mpss.
Thus, to develop algorithms that are as efficient as possible, we design algorithms that use only three operations: matrix--\mps products, \mps--\mps inner products, and \mps linear combinations.

\paragraph{Inner products.}
Consider two \mpss \(\vx,\vy\in\bbR^{d^n}\) with respective bond dimensions $\chi, \chi'$.  Assuming that $\chi \geq \chi'$, the inner product $\langle \vx, \vy \rangle$ costs \(\cO(nd\chi^2\chi')\) arithmetic operations \cite[\S4.2]{ose11}.
Consequently, given two \mps-column matrices $\mX,\mY \in \bbF^{d^n\times k}$ whose columns respectively have bond dimensions $\chi, \chi'$ with $\chi \ge \chi'$, we can form the cross matrix $\mX^*\mY$ by entrywise evaluation in $\order(k^2nd\chi^2\chi')$ operations:
\begin{equation*}
    \mX^* \mY = \bmat{ \langle \vx_1, \vy_1\rangle & \cdots & \langle \vx_1, \vy_k\rangle \\ \vdots & \ddots & \vdots \\
    \langle \vx_k, \vy_1\rangle & \cdots & \langle \vx_k, \vy_k\rangle} \in \bbF^{k\times k}.
\end{equation*}

\paragraph{Linear combinations.}
Let \(\vx_1,\ldots,\vx_k\) be \mpss with bond dimensions \(\chi_1,\ldots,\chi_k\).
The linear combination
$
    \sum_{i=1}^k \alpha_i \vx_i
$
is an \mps of bond dimension at most \(\sum_{i=1}^k\chi_i\).
Thus, if $\mX\in\bbF^{d^n\times k}$ is an \mps-column matrix whose columns have bond dimension $\chi$ and $\mR\in\bbF^{k\times k}$ is a dense matrix,
then \(\mX\mR\) is an \mps-column matrix whose columns have bond dimension at most \(k\chi\).
In practice, we can often compress the columns of \(\mX\mR\) to a smaller bond dimension \cite{ose11,daas2021,daas25}. 
Even with compression, \mps linear combinations can be slow and memory intensive, so we design algorithms that avoid linear combinations when possible.

\subsection{Trace estimation via the \mps Girard--Hutchinson estimator}
\label{subsec:trace_estimation}

\CCdone
\RAMdone
This section presents an algorithm for approximating the trace of a square matrix \mA in the matrix--\mps access model.
In the non-tensor setting, \emph{stochastic trace estimators}~\cite{girard87,hutchinson89,meyer2021hutch++,epperly24trace,persson22,CH23a} are popular tools that estimate the trace from a small number of matrix--vector products using random vectors.
The simplest technique is the \emph{Girard--Hutchinson estimator} \cite{girard87,hutchinson89}, which takes the form
\begin{equation} \label{eqn:gh}
    \GH_\ell(\mA) = \frac1\ell \sum_{i=1}^\ell \vomega_i^* (\mA\vomega_i^{\vphantom{*}})
    \qquad
    \text{for iid copies } \vomega_1,\ldots,\vomega_\ell\sim \vomega \text{ of an isotropic random vector } \vomega.
\end{equation}
In the standard (non-tensor) setting, the Girard--Hutchinson estimator is typically implemented with unstructured random probe vectors,
such as draws from the distribution $\cN_{\bbF}(\vzero,\Id)$.
For any choice of isotropic probe vectors $\vomega_i$, the Girard--Hutchinson estimator is unbiased: \(\E[\GH_\ell(\mA)] = \tr(\mA)\).
The variance decreases slowly at the Monte Carlo rate $\Var(\GH_\ell(\mA)) \sim 1/\ell$, so
the Girard--Hutchinson estimator is primarily used with a small number of probe vectors $\ell$ to reach an estimate of modest quality.
To attain higher accuracy, variance reduction techniques can be quite helpful; see \cref{sec:nystrompp}.

Since generating a single dense Gaussian vector requires \(\cO(d^n)\) operations, we cannot directly apply the Girard--Hutchinson algorithm in the tensor network setting.
Instead, we propose choosing the random vectors $\vomega_1,\ldots \vomega_\ell$ to be \rmpss.
We refer to the resulting algorithm as the \mps Girard--Hutchinson estimator; see \cref{alg:mps-hutch}.
The dominant cost of the algorithm is $\ell$ matrix--\mps products.

\begin{algorithm}[t]
\caption{\textsc{Girard--Hutchinson} in the \mps access model}
\label{alg:mps-hutch}
\begin{algorithmic}[1]
\Require MPS matrix--vector access to psd operator \(\mA \in \bbF^{d^n \times d^n}\), number of queries \(\ell\)
\Ensure Trace estimate \(\GH_\ell(\mA) \approx \tr(\mA)\)

\State Draw independent \rmpss $\vomega_1,\ldots,\vomega_{\ell}$
\State Compute $\vy_1 = \mA\vomega_1,\ldots,\vy_\ell = \mA\vomega_\ell$
\Comment{\(\ell\) matrix--\mps products}
\State \Return
$\tfrac{1}{\ell} \sum_{i=1}^\ell \vomega_i^*\vy_i^{\vphantom{*}}$ \Comment{\(\ell\) \mps inner products}
\end{algorithmic}
\end{algorithm}

\subsubsection{Theoretical guarantees} \label{sec:trace-guarantees}

\ENEdone 
\RAMdone
\CCdone
For a standard Gaussian probe distribution \vomega, the variance of the Girard--Hutchinson estimator~\cref{eqn:gh} is 
\begin{equation} \label{eq:gaussian-gh}
    \Var(\GH_\ell(\mA)) = \frac{1}{\ell} \cdot \rK_{\bbF} \norm{\mA}_{\rm F}^2,
\end{equation}
where \(\rK_{\bbR}=2\) and \(\rK_{\bbC}=1\).
For an \rmps probe distribution, we can bound the variance of the associated Girard--Hutchinson estimator using our analysis of the quadratic form (\cref{thm:rmps_variance_intro}).

\begin{corollary}[Girard--Hutchinson estimator]
    \label{cor:hutch-frob-bound}
    Fix a matrix \(\mA\in\bbF^{d^n \times d^n}\), and let \(\GH_\ell\) denote the Girard--Hutchinson estimator~\eqref{eqn:gh} implemented using an \rmps probe distribution of bond dimension \(\chi\).
    Then  \begin{align*}
        \Var(\GH_\ell(\mA)) &\le \frac{1}{\ell} \cdot \left\{ 2 \left(1 + \frac{1}{\chi}\right)^{n-1} \norm{\mA}_{\rm F}^2 + \left[ 3\left(1 + \frac{2}{\chi}\right)^{n-1} - 2\left(1 + \frac{1}{\chi}\right)^{n-1} - 1\right] \norm{\mA}_*^2 \right\}; &&(\bbF=\bbR) \\
        \Var(\GH_\ell(\mA)) &\le \frac{1}{\ell} \cdot \left\{  \norm{\mA}_{\rm F}^2 + 2\left[ \left(1 + \frac{1}{\chi}\right)^{n-1} - 1\right] \norm{\mA}_*^2 \right\}. &&(\bbF=\bbC)
    \end{align*}
\end{corollary}
\noindent Let us note a few features of this bound,
which are also visible in \cref{fig:GH_trace_estimation}.

\paragraph{The failure of Kronecker vectors.}
\RAMdone
Several authors \cite{avron14,ahle20,bujanovic2021norm,meyer_hutchinsons_2025} have studied the Girard--Hutchinson estimator where the random probe $\vomega$ follows the Gaussian--Kronecker distribution. This method coincides with the \rmps construction with $\chi = 1$. 
\Cref{cor:hutch-frob-bound} indicates that Gaussian--Kronecker probe vectors may induce an estimator whose variance is \emph{exponentially large} in the tensor order $n$, and worst-case examples confirm this fear \cite{meyer_hutchinsons_2025,ahle20}. 
Furthermore, under mild technical conditions, \emph{any} matrix--vector algorithm restricted to Kronecker probe vectors requires exponential work to approximate the trace of a general matrix \cite{meyer25}.
Thus, practitioners should avoid Kronecker probes ($\chi = 1$) for trace estimation and related tasks when $n$ is large.

\paragraph{Transition to Gaussian behavior.}
\RAMdone
\CCdone
The core insight of this paper is that an \rmps behaves somewhat like a Gaussian vector when $\chi \gtrsim n$.
\Cref{cor:hutch-frob-bound} supports this proposition for the Girard--Hutchinson trace estimator.
Indeed, for any scaling parameter \(\gamma > 0\), taking bond dimension \(\chi=\gamma n\) results in the bounds
\begin{align*}
    \Var(\GH_\ell(\mA)) 
    &\le \frac{1}{\ell} \cdot \left\{ 2\e^{\nicefrac{1}{\gamma}} \norm{\mA}_{\rm F}^2 + 3(\e^{\nicefrac{2}{\gamma}}-1)\norm{\mA}_*^2 \right\}; 
    &&(\bbF=\bbR) \\
    \Var(\GH_\ell(\mA)) &\le \frac{1}{\ell} \cdot \left\{ \norm{\mA}_{\rm F}^2 + 2(\e^{\nicefrac{1}{\gamma}}-1) \norm{\mA}_*^2 \right\}. 
    &&(\bbF=\bbC)
\end{align*}
When \(\chi\geq n\) so that \(\gamma\geq1\), all coefficients are bounded by an absolute constant independent of \(n\).
That is, we avoid the exponential cost that resulted from the choice \(\chi=1\).

As \(\chi\) grows larger, the term \(\e^{1/\gamma}\) approaches 1, making the variance bounds approach those of the standard Gaussian vector in \cref{eq:gaussian-gh}.
This may tempt us to increase \(\chi\) as a way to decrease the variance of this estimator.
However, the computational cost of the Girard--Hutchinson estimator scales linearly with the number of probe samples \(\ell\) and at least quadratically with \(\chi\).
Provided that \(\chi\gtrsim n\), increasing the number of samples \(\ell\) offers a more efficient strategy for reducing the variance.

\paragraph{\rmpss can still be much worse than standard Gaussians.}
\
\RAMdone
For Gaussian probes, the variance of the estimator scales only with the Frobenius norm of \mA, while the variance with \rmps probes scales with both the Frobenius norm and the nuclear norm.
When the singular values of \mA decay slowly, the nuclear norm of \mA may be as much as \(d^{n/2}\) times larger than the Frobenius norm, suggesting that \mps Girard--Hutchinson may achieve much higher error than its Gaussian counterpart in these cases.
Of course, standard Gaussian probes are not available for exponential-scale computations,
so the Gaussian baseline only has conceptual value. See \cref{subsec:app_hutch} for an empirical comparison of \rmpss to Gaussian probes.

\paragraph{Relative-error guarantees.} 
\ 
\RAMdone
Fix $\eps \in (0,1)$.
When the matrix \mA is psd, the \mps Girard--Hutchinson estimator achieves the following relative error guarantee
\begin{equation} \label{eq:worst-case-trace}
    |\hat{\tr} - \tr(\mA)| \le \varepsilon \tr(\mA) \quad \emph{\text{with probability }} \geq 2/3 \quad
    \emph{\text{ for all psd }} \mA
\end{equation}
using $\ell \asymp (1 + \rK_{\bbF} / \chi)^{n-1}/\varepsilon^2$ matrix--\mps products.
This result follows by combining \cref{cor:hutch-frob-bound} with Chebyshev's inequality and the norm comparison \(\norm{\mA}_{\rm F} \leq \norm{\mA}_{*}\).
We conclude that the sample complexity $\ell$ may be exponential in $n$ when the bond dimension $\chi$ is constant, but it is independent of $n$ when $\chi \gtrsim n$.
Choosing $\mA$ to be the $d^n\times d^n$ matrix of all ones shows that $\ell \asymp (1 + \rK_{\bbF} / \chi)^{n-1}/\varepsilon^2$ matrix--vector products is also \emph{necessary} for this algorithm to achieve the guarantee \cref{eq:worst-case-trace} for a worst-case matrix \mA.

\subsubsection{Example: Partition function calculations}
\label{subsec:app_hutch}

\begin{figure}
    \centering
\includegraphics[width=\linewidth]{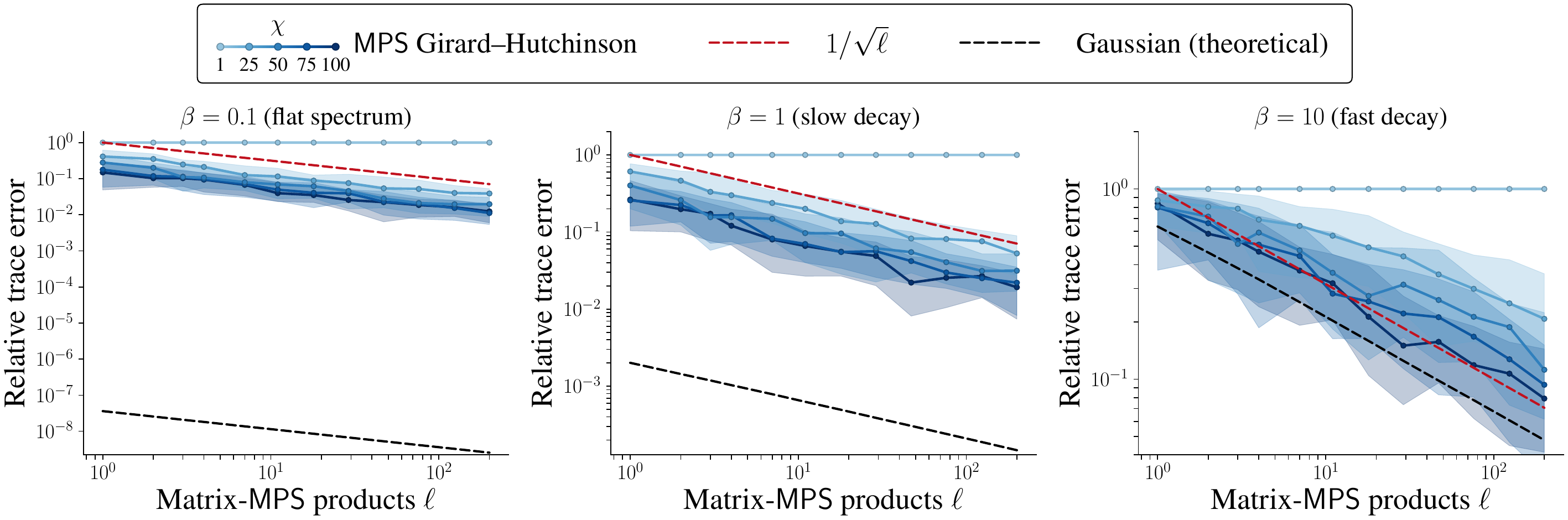}
\caption{
\textbf{
\rmps Girard--Hutchinson: Dependence on spectral decay.} \mps Girard--Hutchinson estimator (\cref{alg:mps-hutch})
used to estimate to the partition function \(Z(\beta)=\tr(\e^{-\beta \mH})\) of the classical 1D Ising Hamiltonian \cref{eq:open_ising} with $n=50$ sites and \(\beta\in\{0.1,1,10\}\). The markers track the median relative error over 60 trials; shaded regions connect the $25\%$ and $75\%$ quantiles.
The black line \((\blackdashedline)\) shows the predicted behavior for standard Gaussian probes, and the red line \((\reddashedline)\) illustrates the Monte Carlo error decay \(1/\sqrt{\ell}\).
See \cref{subsec:app_hutch} for interpretation and conclusions.
}
\label{fig:GH_trace_estimation} 
\end{figure}

\RAMdone 
To illustrate the performance of the \mps Girard--Hutchinson trace estimator,
we compute estimates for the partition function $Z(\beta) = \tr(\e^{-\beta \mH})$
of the classical 1D Ising model with Hamiltonian
\begin{equation}
\label{eq:open_ising}
  \mH=-\sum_{i=1}^{n-1}\mZ_i\mZ_{i+1}\in\bbC^{2^n\times 2^n}
\qquad
\text{with }\mZ_i=\mI_2^{\otimes(i-1)}\otimes
\begin{bmatrix}
1&0\\
0&-1
\end{bmatrix}
\otimes\mI_2^{\otimes(n-i)}.  
\end{equation}
Since $\e^{-\beta\mH}$ can be written explicitly in terms of hyperbolic trigonometric functions \cite{ising25}, the matrix $\e^{-\beta\mH}$ admits an \mpo representation with bond dimension $D=2$, facilitating fast matrix--\mps products.
We instantiate this matrix with tensor order \(n=50\) at three inverse temperatures \(\beta\in\{0.1,1,10\}\).
Larger values of \(\beta\) result in matrices \(\e^{-\beta \mH}\) with faster eigenvalue decay, as evident in the following norm ratios:
\[
    \frac{\norm{\e^{-\beta \mH}}_{\rm F}}{\norm{\e^{-\beta \mH}}_*} \approx \begin{cases}
        4 \cdot 10^{-8}, & \beta=0.1 \\
        2 \cdot 10^{-3}, & \beta=1 \\
        0.7, & \beta = 10.
    \end{cases}
\]
Decreasing $\beta$ makes the spectrum flatter, resulting in a trace estimation problem that it is easier for the Girard--Hutchinson method.

For this partition function computation, \Cref{fig:GH_trace_estimation} shows the performance of the \mps Girard--Hutchinson trace estimator~\eqref{eqn:gh}
with varying bond dimensions $1\le \chi \le 100$.
The estimator always fails when $\chi = 1$, but the accuracy improves as $\chi$ increases.
For sufficiently large $\chi$, the error in the trace estimates follows a $C_{\chi,\beta}/\sqrt{\ell}$ convergence profile for a prefactor $C_{\chi,\beta} > 0$ that decreases in both $\chi$ and $\beta$.
We observe that increasing the bond dimension $\chi$ always improves the quality of the estimator.
The dashed black line indicates the predicted relative error~\cref{eq:gaussian-gh} for standard Gaussian probes, although that computation is infeasible.
For a flat spectrum ($\beta = 0.1$), \rmpss suffer errors that are orders of magnitude worse than a Gaussian probe would achieve.  When the spectrum decays more quickly ($\beta = 10$), their performance is comparable.

\subsection{Low-rank approximation via the \mps Nystr{\"o}m method}
\label{sec:nystrom}

In many applications, we encounter matrices that are well-approximated by a low-rank matrix~\cite{UT19}.
For these matrices, we can construct a low-rank approximation to compress the data
or to estimate dominant eigenvalues, eigenvectors, singular values, or singular vectors \cite[\S4]{tropp2023}.
This section develops an efficient algorithm that reports a low-rank approximation of a
psd matrix $\mA \in \bbF^{d^n\times d^n}$ with exponential dimension under the matrix--\mps
access model.
\Cref{sec:more-algs} comments on extensions to non-psd matrices.

The \emph{randomized Nystr\"om method} is a well-established algorithm for computing a low-rank approximation of a psd matrix \cite{Williams00,gittens16,tropp17b}.
Given a psd input matrix \(\mA\in\bbF^{N\times N}\) and a \emph{sketching dimension} $k$, the standard algorithm draws a random test matrix \(\mOmega\in\bbF^{N\times k}\), computes the sketch \(\mY=\mA\mOmega \in \bbF^{N\times k}\), forms the core matrix \(\mC=\mOmega^* \mY\in\bbF^{k\times k}\), and builds the following approximation to the matrix \mA:
\begin{equation}
    \hat\mA
    = \mA \langle\mOmega\rangle
    \defeq \mY\mC^+\mY^*
    = \mA\mOmega(\mOmega^*\mA\mOmega)^+\mOmega^*\mA.
    \label{eq:nystrom-basic}
\end{equation}
Most implementations report the eigenvalue decomposition of the approximation:
\begin{equation} \label{eq:nystrom-factored}
    \mA\langle\mOmega\rangle = \mU\mLambda\mU^*.   
\end{equation}
Here, $\mU \in \bbF^{N\times k}$ is a tall orthonormal matrix whose columns are the eigenvectors of $\mA\langle\mOmega\rangle$, and $\mLambda \in \bbR_+^{k\times k}$ is the diagonal matrix that lists the associated eigenvalues.

To construct the Nystr\"om approximation of an exponentially large psd matrix $\mA$,
we advocate choosing \mOmega to be an \rmps test matrix.
Even with an \rmps test matrix, existing implementations of the Nystr\"om approximation \cite{szlam14,tropp2017a}
are not tractable for \mps data because of how they manipulate \mps-column matrices.
To resolve this issue, we propose a new implementation of the Nystr\"om approximation, called \GramNystrom, that is efficient in the \mps setting.

\subsubsection{The \GramNystrom algorithm for fast, numerically reliable \mps Nystr\"om approximation} \label{sec:nystrom-impl}

\RAMdone
\CCdone
To compute the Nystr\"om approximation \cref{eq:nystrom-basic} in factored form \cref{eq:nystrom-factored}, we propose the following \mps \GramNystrom method (\cref{alg:mps-nystrom}).
Conceptually, the algorithm proceeds in four steps. 
\begin{enumerate}
    \item Compute the sketch \(\mY = \mA\mOmega\) using \(k\) matrix--\mps products.
    \item Construct the core matrix \(\mC = \mOmega^*\mY\) and the Gram matrix \(\mS = \mY^*\mY\) using \mps inner products.
    \item Compute the Cholesky factorization \(\mC=\mR\mR^*\) and eigendecomposition \(\mR^{-*}\mS\mR^{-1} = \mV\mLambda\mV^*\)
    \item \emph{(Optional)} Compute \(\mU = \mY(\mR^{-1}\mV\mLambda^{-1/2})\) via \mps linear combinations. 
\end{enumerate}
With some algebra, we can verify that \(\mU\mLambda\mU^* = \mA\langle\mOmega\rangle\) in exact arithmetic.
The fourth step is the most expensive because it requires \mps linear combinations,
but we can eliminate this step if we only need the eigenvalue matrix \mLambda. If \mOmega is an \rmps test matrix with bond dimension \(\chi\), then eigenvalue extraction via \(\mps\) \(\GramNystrom\) costs \(k\) matrix--\mps products plus \(\cO(k^2nd\,\smash{\overline\chi}^3+k^3)\) additional operations.
We remind the reader that $\overline{\chi}$ is the bond dimension of the matrix--\mps product $\mA\vomega$.

In principle, these four steps could be used to form the Nystr\"om approximation induced by any random test matrix \mOmega,
but this sequence of steps is particularly well-suited to the \mps setting.
Indeed, this procedure manages to avoid the SVD of a tall $d^n\times k$ matrix that
arises in most Nystr{\"o}m implementations;
we achieve the same effect through the eigendecomposition of a small $k\times k$ matrix.

\begin{algorithm}[t]
\caption{\GramNystrom: Nystr\"om approximation in the \mps access model }
\label{alg:mps-nystrom}
\begin{algorithmic}[1]
\Require MPS matrix--vector access to psd matrix $\mA\in\bbF^{N\times N}$, embedding dimension $k$, flag $\texttt{just\_eigs}$
\Ensure 
Factors \(\mU,\mLambda\) of the Nystr\"om approximation \(\widehat\mA = \mU\mLambda\mU^*\), or just the eigenvalue matrix \(\mLambda\)

\State Draw a Gaussian \trp
$
\mOmega\in\bbF^{N\times k}
$
\State $\mY \gets \mA\mOmega$
\Comment{\(k\) matrix--\mps products}

\State $\mG \gets \mOmega^\ast\mOmega,\ \mC \gets \mOmega^\ast\mY$
\State $\mT \gets \texttt{chol}(\mG,\texttt{'upper'})$
\State $\mB \gets \mT^{-\ast}\mC\mT^{-1}$ 
\Comment{Apply inverses via triangular solves}
\State $\nu \gets \max\bigl(0,\sqrt{\varepsilon_{\mathrm{mach}}}\,\lambda_{\max}(\mB)-\lambda_{\min}(\mB)\bigr)$
\Comment{Calculate shift for stability}

\State $\mR \gets \texttt{chol}(\mC+\nu\mG,\texttt{'upper'})$
\Comment{\(\mC+\nu\mG = \mOmega^*\mA_\nu\mOmega\)}

\State $\mS \gets \mY^\ast\mY+2\nu\mC+\nu^2\mG$
\Comment{$\mS=\mOmega^\ast\mA_\nu^2\mOmega=(\mY+\nu\mOmega)^\ast(\mY+\nu\mOmega)$}
\State $\mE \gets \mR^{-\ast}\mS\mR^{-1}$
\Comment{Apply inverses via triangular solves}
\State $[\mV,\mLambda] \gets \texttt{eigh}(\mE)$
\State $\mLambda \gets \max(\mLambda-\nu\mI,0)$
\Comment{Undo diagonal shift}

\State \textbf{if} $\texttt{just\_eigs}$ \textbf{then return }\mLambda

\State $\mY \gets \mY + \nu \mOmega$
\Comment{Optionally compress columns of $\mY$}
\State $\mU \gets \mY(\mR^{-1}\mV\mLambda^{-1/2})$
\Comment{\mps-column matrix--matrix product}
\State \Return $\mU,\mLambda$
\end{algorithmic}
\end{algorithm}

The four steps above can fail dramatically in finite-precision arithmetic when the core matrix \mC is approximately low-rank.
To improve the stability of \mps\GramNystrom, we follow a standard practice used in the matrix setting \cite{szlam14,tropp2017a} and apply a numerical shift.
Rather than approximating \mA directly, we approximate the shifted matrix \(\mA_{\nu} \defeq \mA + \nu\mI\) for a shift parameter \(\nu\) proportional to the square root of the machine precision \(\eps_{\rm mach}\).
Afterward, we subtract \(\nu\) from the computed eigenvalues.
In \cref{alg:mps-nystrom}, we use the shift \(\nu = \sqrt{\eps_{\rm mach}} \lambda_{\max}(\mB) - \lambda_{\min}(\mB)\) where
\begin{equation*}
    \mB = \mT^{-*} (\mOmega^*\mY) \mT^{-1} \quad \text{where } \mT^*\mT = \mOmega^*\mOmega \text{ is a Cholesky factorization.}
\end{equation*}
This shift is chosen to ensure that \mB and \mS are always numerically positive definite.
Since \GramNystrom requires a shift of size \(\cO(\sqrt{\eps_{\rm mach}})\),
it can only approximate eigenvalues up to about 8 digits of accuracy in double precision (\(\eps_{\rm mach} \approx 10^{-16}\)), where traditional Nystr\"om methods can reach 15 or 16 digits. While \mps \GramNystrom is faster than standard Nystr\"om implementations, it is somewhat less accurate.

\subsubsection{Theoretical guarantees}
\RAMdone
\CCdone
The \mps Nystr{\"o}m method is equipped with rigorous guarantees on the
quality of approximation.
These results demonstrate that \trps achieve guarantees similar to Gaussian test matrices \cite[Cor.~8.8]{tropp2023}, which are the accepted standard in the matrix setting \cite{tropp2017a,halko11,epperly26sharp}.
Here is our main theorem:
\begin{theorem}[Nystr\"om approximation] \label{thm:nystrom}
    Fix a psd matrix \(\mA \in \bbR^{d^n \times d^n}\) and target rank \(r \leq d^n\).
    Let $\mOmega \in \bbR^{d^n \times k}$ be an \rmpstest with embedding dimension \(k \gtrsim (1+\nicefrac{\rK_{\bbF}}\chi)^{n-1}r\), and introduce the Nystr\"om approximation \(\hat\mA = \mA\langle\mOmega\rangle\).
    Recall that \(\lra \mA r\) denotes a best rank-\(r\) approximation to \mA in the sense of Eckart--Young--Mirsky.
    Then the following bounds hold with probability at least \(0.99\):
    \begin{enumerate}[label=(\alph*)]
        \item 
        \label{item:trp-nystrom-trace-bound}
        \textbf{\textit{Nuclear norm error.}} \(\snorm{\mA - \hat\mA}_* \lesssim \snorm{\mA - \lra \mA r}_*\).
        \item
        \label{item:trp-nystrom-frob-bound}
         \textbf{\textit{Frobenius norm error.}} \(\snorm{\mA - \hat\mA}_{\rm F}^2 \lesssim \snorm{\mA - \lra\mA r}_{\rm F}^2 + \frac1r \norm{\mA - \lra \mA r}_*^2\).
    \end{enumerate}
    In particular, if \(\chi \gtrsim n\) then embedding dimension \(k \gtrsim r\) suffices to achieve these guarantees.
\end{theorem}

We establish this theorem in \cref{sec:nystrom-analysis}.
The nuclear norm bound (\cref{thm:nystrom}\ref{item:trp-nystrom-trace-bound}) can be deduced from existing results in the literature.
The Frobenius norm bound (\cref{thm:nystrom}\ref{item:trp-nystrom-frob-bound}) is new; it relies on the variance bound for quadratic forms (\cref{thm:rmps_variance_intro}).
The nuclear-norm term in the right-hand side of this bound appears to be necessary \cite{tropp2023,persson22}.
The new Frobenius-norm guarantee plays an essential role in the analysis of variance-reduced trace estimators; see \cref{sec:nystrompp,sec:nystrompp-analysis}.

\subsubsection{Application: Quantum entropy estimation} \label{sec:entropy}

\RAMdone
We illustrate the performance of the \mps\GramNystrom method (\cref{alg:mps-nystrom}) on a stylized problem, where we estimate the \emph{von Neumann entropy} of a reduced density matrix.
Estimating von Neumann entropy and other measures of entanglement is a core topic in quantum many-body physics and quantum information theory \cite{Brydges19,huang2020predicting,feldman22}.

To formalize this problem, consider a quantum state $\vpsi\in \bbC^{N_{\rm A}}\otimes \bbC^{N_{\rm B}}$ describing two possibly interacting subsystems $\rm A$ and $\rm B$.
The von Neumann entropy is
\begin{equation}
\label{eq:vn_entropy}
S(\mrho_{\rm A})=-\tr\bigl(\mrho_{\rm A}\log(\mrho_{\rm A})\bigr)
=-\sum_{i=1}^{N_{\rm A}}\lambda_i(\mrho_{\rm A})\log\bigl(\lambda_i(\mrho_{\rm A})\bigr).
\end{equation}
We instate the convention $0\log 0 = 0$.  The \emph{reduced density matrix}
is defined as
$$
\mrho_{\rm A}=\tr_{\rm B}\bigl(\vpsi\vpsi^*) \in \bbC^{N_{\rm A}\times N_{\rm A}}.
$$
It is generated using the partial trace operation $\tr_{\rm B} \colon \bbC^{(N_{\rm A} N_{\rm B})\times (N_{\rm A} N_{\rm B})} \to \bbC^{N_{\rm A}\times N_{\rm A}}$.

Suppose that we have matrix--\mps access to $\mrho_{\rm A}$, and we use that access to build an \mps\GramNystrom approximation $\hat{\mrho}_{\rm A}\approx \mrho_{\rm A}$.  
This immediately suggests we use the approximation
\begin{equation*}
    S(\mrho_{\rm A}) \approx S(\hat{\mrho}_{\rm A}\,) = -\sum_{i=1}^{N_{\rm A}} \lambda_i(\hat{\mrho}_{\rm A}\,) \log(\lambda_i(\hat{\mrho}_{\rm A}\,)).
\end{equation*}
In reference to the work of Persson et al.\ \cite{persson25,persson23funNystrom}, we call $S(\hat{\mrho}_{\rm A}\,)$ the \emph{funNystr\"om estimate} of $S(\mrho_{\rm A})$.
Since we are not using the eigenvectors of our Nystr\"om approximation, we avoid computing \mps linear combinations, making this algorithm highly efficient.

To demonstrate this computational pipeline, we consider a model scenario where $\vpsi$ is the ground state of a 1D chain of $200$ quantum spins interacting under the periodic transverse field Ising model (PTFIM):
\begin{equation}
\label{eq:ptfim}
\mH =
-\sum_{i=1}^{n}\mZ_i\mZ_{\operatorname{mod}(i,n)+1}
-h\sum_{i=1}^{n}\mX_i
\in \bbR^{2^n\times 2^n}, 
\end{equation}
where $\mZ_i\in\bbR^{2^n\times 2^n}$ is defined in \cref{eq:open_ising} and $\mX_i = \Id_2^{\otimes (i-1)} \otimes \mX \otimes \Id_2^{\otimes (n-i)}$, in which $\mX=\sbmat{0 & 1 \\ 1 & 0}$ is the Pauli $\sigma_x$ matrix.
We set the external field strength $h=1.1$, which places the model near its critical point. 
The ground state of \mH is computed as a \mps using the DMRG2 algorithm \cite{white92,Sch11} using a convergence tolerance of $10^{-8}$. 
We set the subsystems $\rm A$ and $\rm B$ to be the left and right halves of the chain, leading to a reduced density matrix $\mrho_{\rm A}$ of dimension $2^{100}\times 2^{100}$.
For this choice of Hamiltonian, the eigenvalues of $\mrho_{\rm A}$ can be computed exactly from the \mps ground state \cite[\S5.1.1]{orus14}, providing a baseline for comparison.
In other settings, the density matrix $\mrho_{\rm A}$ can often be described as an \mpo or matrix product density operator \cite{verstraete04mpdo}.
These kinds of matrices are natural candidates for the funNystr\"om estimator, since they admit efficient matrix--MPS products, but it is challenging to compute their eigenvalues directly.

\begin{figure}[t]
    \centering
\includegraphics[width=\linewidth]{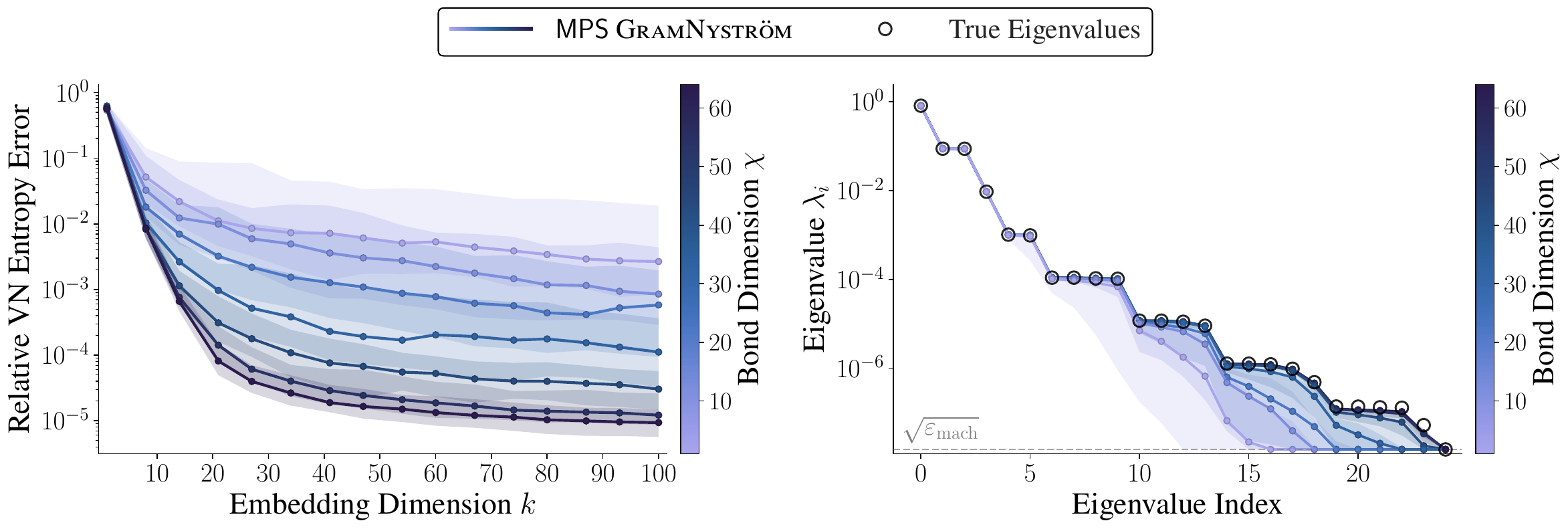}
\caption{\textbf{Nyström approximation at exponential scale.}
Estimates via funNystr\"om for the von Neumann entanglement entropy \cref{eq:vn_entropy} of the half-chain reduced density matrix $\boldsymbol{\rho}_{A}\in\mathbb{C}^{2^{100}\times 2^{100}}$ associated with the ground state of the $200$-site periodic transverse-field Ising Hamiltonian \cref{eq:ptfim}.
The markers track the median over 60 trials; shaded regions are bounded by the $10\%$ and $90\%$ quantiles.
(\textit{Left}) Relative entropy errors $|S(\mrho_{\rm A})- S(\hat{\mrho}_{\rm A})|/|S(\mrho_{\rm A})|$ as a function of the embedding dimension $k$ for each bond dimension $\chi\in\{1,2,4,8,16,32,64\}$. 
(\textit{Right}) Approximate eigenvalues of $\hat{\mrho}_{\rm A}$ for rank-$100$ approximation produced by \mps \GramNystrom for several values of $\chi$ compared with the exact eigenvalues.
See \cref{sec:entropy}.}
\label{fig:nystrom}
\end{figure}

\Cref{fig:nystrom} reports the relative error of the estimated von Neumann entropy and the approximate eigenvalues of $\mrho_{\rm A}$ obtained from \cref{alg:mps-nystrom} using \trps with a range of bond dimensions $\chi\in[1,64]$ and embedding dimensions $k\in[1,100]$.
See \cref{app:experiment_details} for details.
The left panel shows that the relative error of the entropy estimate decreases monotonically both in the bond dimension \(\chi\) and the embedding dimension $k$. 
The right panel shows the eigenvalues of $\hat{\mrho}_{\rm A}$ for fixed embedding dimension $k = 100$ and varying bond dimensions $\chi$.
\rmps test matrices with low bond dimension \(\chi\approx1\) struggle to recover
the 10th eigenvalue, while \rmps test matrices with larger bond dimensions approximate all the eigenvalues up to the $\sqrt{\varepsilon_{\rm mach}}$ floor.
Both panels support the core message that choosing a sufficiently large bond dimension is essential to get the best results with \rmps dimension reduction.

\subsection{\texorpdfstring{Variance-reduced trace estimators: \mps \nystrompp and \mps \XNysTrace}{Variance-reduced trace estimators: MPS Nystrom++ and MPS XNysTrace}}
\label{sec:nystrompp}

\RAMdone
The Girard--Hutchinson estimator suffers from a high variance, which typically limits it to one or two digits of accuracy in practice.
To address this shortcoming, several recent works show how to reduce the variance of the estimator
by employing the trace of a low-rank matrix approximation as a control variate,
resulting in trace estimates that can be orders of magnitude more accurate
\cite{GSO17,meyer2021hutch++,persson22,epperly24trace,CH23a,madhusudan26flextrace}.

This paper develops \mps extensions of two estimators:
\emph{\nystrompp} \cite{persson22} and \emph{\XNysTrace} \cite{epperly24trace}.
Both estimators compute the trace of a psd matrix \mA by combining the Girard--Hutchinson estimator with a Nystr\"om approximation.
As such, we can lean on our implementations of \mps Girard--Hutchinson and \mps Nystr\"om to design efficient \mps-based variations of \nystrompp and \XNysTrace.
In practice, \mps \XNysTrace tends to produce more accurate results, but \mps \nystrompp is more amenable to analysis.

\begin{algorithm}[t]
\caption{\textsc{\mps Nyström++}}
\label{alg:mps-npp-gram}
\begin{algorithmic}[1]
\Require MPS matrix--vector access to psd operator $\mA\in\bbF^{N\times N}$, total query budget $t$
\Ensure Trace estimate $ \nystromppest\approx \tr(\mA)$

\State Draw Gaussian \trps $\mOmega\in\bbF^{N\times \lfloor t/2\rfloor}$ and $\mPsi\in\bbF^{N\times \lceil t/2\rceil}$
\State $\mY \gets \mA\mOmega$
\State $\mG \gets \mOmega^\ast\mOmega,\mC \gets \mOmega^\ast\mY$
\State $\mT \gets \texttt{chol}(\mG,\texttt{'upper'})$
\State $\mB \gets \mT^{-\ast}\mC\mT^{-1}$
\Comment{Apply inverses via triangular solves}
\State $\nu \gets \max\bigl(0,\sqrt{\varepsilon_{\mathrm{mach}}}\,\lambda_{\max}(\mB)-\lambda_{\min}(\mB)\bigr)$
\Comment{Calculate shift for stability}
\State $\mR \gets \texttt{chol}(\mC+\nu\mG,\texttt{'upper'})$
\Comment{\(\mC+\nu\mG = \mOmega^*\mA_\nu\mOmega\)}
\State $\mS \gets \mY^\ast\mY+2\nu\mC+\nu^2\mG$
\Comment{$\mS=\mOmega^\ast\mA_\nu^2\mOmega=(\mY+\nu\mOmega)^\ast(\mY+\nu\mOmega)$}
\State $\mE \gets \mR^{-\ast}\mS\mR^{-1}$
\State $\mU \gets \mR^{-\ast}(\mY^\ast\mPsi+\nu\cdot \mOmega^\ast\mPsi)$ 
\State \Return $
\tr(\mE)
+
\sum_{i=1}^{\lceil t/2\rceil}
\left[
\mPsi(:,i)^\ast(\mA\mPsi(:,i))
+
\nu\norm{\mPsi(:,i)}_2^2
-
\norm{\mU(:,i)}_2^2
\right]
-\nu N
$
\Comment{\mps inner products}
\end{algorithmic}
\end{algorithm}

\subsubsection{\texorpdfstring{\mps \nystrompp}{\mps-Nystrom++}}

The basic variance reduction scheme \cite{GSO17,meyer2021hutch++} for trace estimation works as follows.
First, construct a (randomized) approximation $\hat\mA$ of the input
matrix $\mA$, and compute $\tr(\hat\mA)$.
To refine, we apply the Girard--Hutchinson estimator~\eqref{eqn:gh}
to approximate the trace of the residual $\mA - \hat\mA$:
\begin{equation*}
    \tr(\mA) = \tr(\hat\mA) + \tr(\mA - \hat\mA) \approx \tr(\hat\mA) + \GH_\ell(\mA - \hat\mA) = \tr(\hat\mA) + \frac{1}{\ell} \sum_{i=1}^\ell \vpsi_i^* (\mA - \hat\mA)\vpsi_i^{\vphantom{*}}.
\end{equation*}
In this expression, the $\vpsi_i$ are probe vectors for the Girard--Hutchinson estimator.

The \nystrompp algorithm is an instantiation of
this strategy, adapted for a psd input matrix $\mA$.
This algorithm employs the randomized Nystr{\"o}m method~\cref{eq:nystrom-basic}
to construct a factored rank-$k$ approximation
$\hat\mA = \mA\langle\mOmega\rangle$ of the input matrix
using a $k$-column test matrix $\mOmega$.
Explicitly, the \nystrompp estimator is
\begin{equation*}
    \nystromppest \defeq \tr(\mA\langle\mOmega\rangle) + \frac{1}{\ell} \sum_{i=1}^\ell \vpsi_i^* (\mA - \mA\langle\mOmega\rangle)\vpsi_i^{\vphantom{*}}.
\end{equation*}
The algorithm requires a total of $t = k+\ell$ matrix--vector products with \mA,
which usually dominate the computational cost.
Given a fixed budget, Persson et al.\ \cite{persson22} suggest the choice
$k = \ell = \lfloor t/2 \rfloor$, sharing the matrix--vector products evenly
between the Nystr{\"o}m approximation and the trace estimate for the residual. If \(\mOmega,\mPsi\) are \rmps test matrices with bond dimension \(\chi\), then \mps \nystrompp uses \(t=k+\ell\) matrix--\mps products plus \(\cO(t^2nd\,\bar\chi^3+t^3)\) downstream operations.

We extend this approach to the tensor setting, arriving at the \mps \nystrompp estimator.
This algorithm constructs the Nystr{\"o}m approximation $\mA\langle\mOmega\rangle$
using an \rmpstest $\mOmega$, and it performs a Girard--Hutchinson estimate of the residual trace using \rmps probe vectors \(\vpsi_i\).
For numerical reliability, we apply the procedure to the shifted matrix
$\mA_\nu = \mA + \nu \Id$ and subtract $\tr(\nu \Id)=\nu N$ at the last step.

\subsubsection{\texorpdfstring{\mps \XNysTrace}{ \mps-XNysTrace}}
\label{subsec:xnystrace}

\RAMdone
When the eigenvalues of \mA decay rapidly, the accuracy of a variance-reduced trace estimator
primarily depends on
the accuracy of the low-rank approximation $\hat\mA \approx \mA$.
As such, given a fixed budget of matrix--vector products, we should endeavor to make the rank of the Nystr\"om approximation $\hat\mA$ as large as possible.
The XNysTrace estimator \cite{epperly24trace} accomplishes this goal using a leave-one-out design.
Let $\mOmega \in \bbF^{N\times k}$ 
be an isotropic test matrix with iid random columns, and let $\mOmega_{-i}$ denote the ``downdated'' test matrix where the $i$th column $\vomega_i$ has been deleted.
The \XNysTrace estimator takes the form
\begin{equation*}
    \xnystraceest \coloneqq \frac{1}{t} \sum_{i=1}^t \left[ \tr(\mA\langle \mOmega_{-i}\rangle) + \vomega_i^*(\mA - \mA\langle \mOmega_{-i}\rangle) \vomega_i^{\vphantom{*}}  \right] \quad \text{where } \mOmega = \frac{1}{\sqrt{t}} \begin{bmatrix} \vomega_1 & \cdots & \vomega_t \end{bmatrix}.
\end{equation*}
This estimator has a single parameter $t$, eliminating the need to separately choose the rank $k$ and number of test vectors $\ell$ in the \nystrompp estimator.
It is formed by constructing a rank-($t-1$) approximation using the downdated test matrix $\mOmega_{-i}$ and using the remaining vector $\vomega_i$ to estimate the residual trace $\tr(\mA - \mA\langle \mOmega_{-i}\rangle)$.
The estimator then averages over all choices $i$ to reduce variance.
Using downdating formulas for the Nystr\"om approximation \cite{epperly24trace,mtmepperly25,ET24}, this estimator can be implemented using only $t$ matrix--vector products.

Leveraging the \mps Girard--Hutchinson, \mps Nystr\"om, and \mps\nystrompp methods, it is straightforward to implement \mps version of \XNysTrace in the tensor network setting.
Our pseudocode in \Cref{alg:mps-xnystrace} most directly mirrors the implementation in Epperly's thesis \cite[Prog.~14.3]{mtmepperly25}. If \(\mOmega\) is an \rmps test matrix with bond dimension \(\chi\), then \mps \XNysTrace uses \(t\) matrix--\mps products plus \(\cO(t^2nd\,\bar\chi^3+t^3)\) downstream operations.

\begin{algorithm}[t]
\caption{\textsc{\mps XNysTrace}}
\label{alg:mps-xnystrace}
\begin{algorithmic}[1]
\Require MPS matrix--vector access to psd operator $\mA\in\bbF^{N\times N}$, query budget $t$
\Ensure Trace estimate $\xnystraceest\approx\tr(\mA)$

\State Draw a Gaussian \trp\ 
$\mOmega \in\bbF^{N\times t}$
\State $\mY \gets \mA\mOmega$
\State $\mG \gets \mOmega^\ast\mOmega,
       \mC \gets \mOmega^\ast\mY$
\State $\mR \gets \operatorname{chol}(\mG,\texttt{'upper'})$
\State $\mB \gets \mR^{-\ast}\mC\mR^{-1}$
\Comment{Apply inverses via triangular solves}

\State $\nu \gets
\max(
0,\,
\sqrt{\varepsilon_{\mathrm{mach}}}\,\lambda_{\max}(\mB)
-\lambda_{\min}(\mB)
)$
\Comment{Stability shift}

\State $\mS \gets \mY^\ast\mY+2\nu\mC+\nu^2\mG$
\Comment{$\mS=\mOmega^\ast\mA_\nu^2\mOmega$}

\State $\mL \gets \operatorname{chol}(\mC+\nu\mG,\texttt{'lower'})$
\Comment{\(\mC+\nu\mG = \mOmega^*\mA_\nu\mOmega\)}
\State $\mM \gets \mL^{-\ast}\mL^{-1}$
\Comment{$\mM=(\mOmega^*\mA_\nu\mOmega)^{-1}$}
\State $\mV\gets\mS\mM$
\State\Return
$
\tr(\mV)
+
\frac{1}{t}
\sum_{i=1}^{t}
\left[
(
t-\langle \mM(:,i),\mV(:,i)\rangle
)/
\mM_{ii}
\right]
-\nu N
$
\end{algorithmic}
\end{algorithm}
\subsubsection{Theoretical guarantees}
\RAMdone
This section outlines a theoretical analysis of the \mps \nystrompp estimator. Implemented with \rmpss of sufficiently large bond dimension, the estimator achieves
approximation guarantees that compare to the \nystrompp estimator
with a standard Gaussian probe vectors.
Unfortunately, the analysis of \mps \XNysTrace is technically involved,
so this paper does not pursue the extension.

Our first result shows that \mps \nystrompp achieves \emph{spectral accuracy}.
That is, the accuracy of the method is controlled by the error
of the best rank-$r$ approximation of the input matrix.

\begin{theorem}[\mps \nystrompp: Spectral accuracy] \label{thm:nystrompp-spectral-rmps}
    Let $\mA \in \bbF^{d^n\times d^n}$ be psd, fix a target rank $r$ and an oversampling parameter $\gamma > 2$, and consider the \rmps \nystrompp estimator with $k$ matvecs for low-rank approximation and $\ell$ matvecs for residual trace estimation, where $k, \ell \gtrsim r$.
    Then for bond dimension $\chi = \gamma n$, it holds with probability at least 0.98 that
\begin{equation*}
|\nystromppest - \tr(\mA)|^2 \lesssim \frac{1}{r} \snorm{\mA - \lowrank\mA_r}_{\rm F}^2 + \left(\frac{1}{r^2} + \frac{1}{\gamma r}\right) \norm{\mA - \lowrank{\mA}_r}_*^2.
    \end{equation*}
\end{theorem}

We establish this bound in \cref{sec:nystrompp-analysis}.
The theorem states that the error of \mps \nystrompp decays quickly when
the eigenvalues of the input matrix \mA decay quickly.
As with the Girard--Hutchinson estimator, one can improve the quality of the \nystrompp estimate by increasing the matrix--\mps product budget $t = k+\ell$ or the bond dimension $\chi$.
Our numerical experience suggests that the former choice is better.

Next, we provide worst-case bounds without assumptions on the spectrum of the psd matrix \mA.
Variance-reduced trace estimators such as \nystrompp and \XNysTrace can achieve relative error \(\eps\)
using only \(\cO(1/\eps)\) matrix--vector products \cite{persson22,epperly24trace}.
This rate is optimal \cite{meyer2021hutch++}, and it improves upon the \(\cO(1/\eps^2)\) rate achieved by the Girard--Hutchinson method.
Our second error bound shows that \rmps \nystrompp attains this error scaling, although we must increase the bond dimension to $\chi \asymp n / \varepsilon$.
This appears to be the first guarantee showing how to achieve the \(\cO(1/\eps)\) rate for trace estimation for an exponentially large, tensor-structured matrix \mA.

\begin{theorem}[\mps \nystrompp: Worst-case guarantee]
    \label{thm:nystrompp}
    Let $\mA \in \bbF^{d^n\times d^n}$ be psd, fix an error level $\varepsilon \in (0,1/2)$, and consider the \rmps \nystrompp estimator with $k$ matvecs for low-rank approximation and $\ell$ matvecs for residual trace estimation.
    Then if $\chi \asymp n/\varepsilon$ and $k=\ell \asymp 1/\varepsilon$, it holds that 
\begin{equation*}
        |\nystromppest - \tr(\mA)| \le \varepsilon \tr(\mA) \quad \text{with probability at least } 0.98.
    \end{equation*}
\end{theorem}

\noindent 
This result follows when we combine \cref{thm:nystrompp-spectral-rmps} with the bound $\norm{\mA - \lowrank{\mA}_r}_{\rm F}^2 \lesssim r^{-1} \norm{\mA}_*^2$ for psd \mA; see \cref{sec:nystrompp-analysis} for details.
In practice, the recommendation $\chi = n/\varepsilon$ seems pessimistic, and we recommend setting $\chi \le n$ in practice.

\begin{figure}[t]
    \centering
\includegraphics[width=1\linewidth]{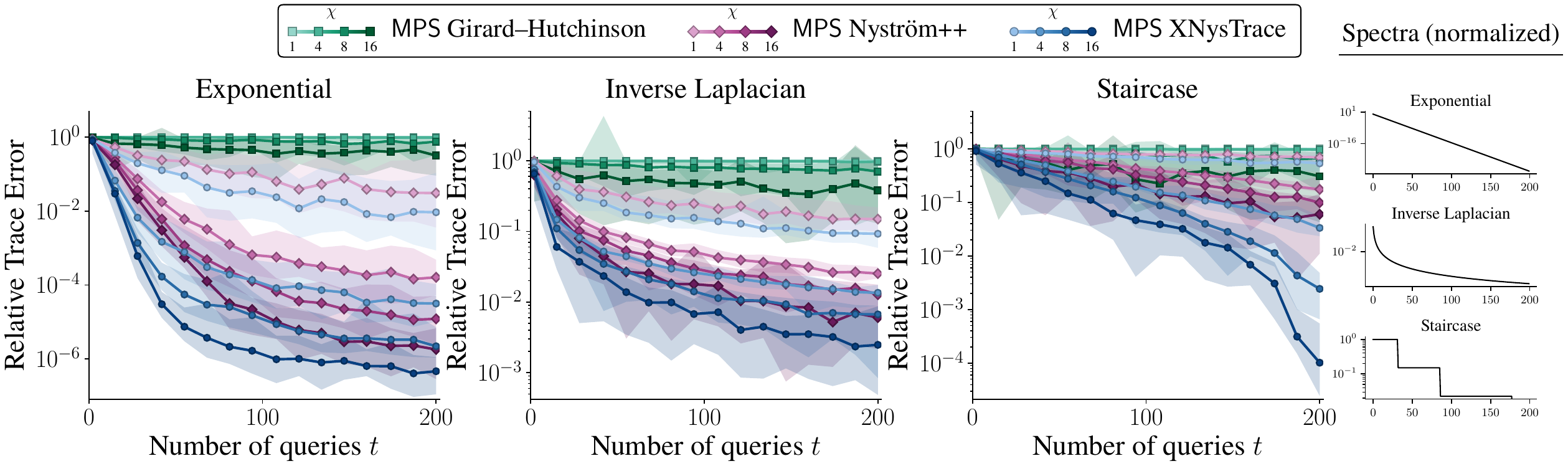}
\caption{\textbf{Comparison of \mps trace estimators with variance reduction.}
\RAMdone\CCdone
Trace estimates for three psd \(\mpos\) \(\mA\in\bbR^{2^{50}\times2^{50}}\) with several spectral decay profiles; the (normalized) spectra appear in the rightmost panel. We compare the \mps Girard--Hutchinson, \nystrompp, and \XNysTrace estimators (\cref{alg:mps-hutch,alg:mps-npp-gram,alg:mps-xnystrace}), and we test bond dimensions \(\chi\in\{1,4,8,16\}\).
Markers track the median relative error  over 60 independent trials;
the shaded regions show the 10\% to 90\% quantile range. Across all instances, the variance-reduced trace estimators improve on \mps Girard--Hutchinson,
and \mps \XNysTrace outperforms \mps \nystrompp.
Increasing the bond dimension improves the accuracy of all estimators.
See \cref{ssec:subsec:mpo_mps}.
}
    \label{fig:trace-comparison}
\end{figure}

\subsubsection{Evaluation: Comparison of trace estimators}
\label{ssec:subsec:mpo_mps} 
\ENEdone 
\RAMdone
\CCdone
\Cref{fig:trace-comparison} compares the Girard--Hutchinson, \nystrompp, and \XNysTrace estimators,
each one implemented with \rmpss.
We test the methods on three psd \mpos \(\mA\in\bbR^{2^{50} \times 2^{50}}\)
with designated eigenvalue decay rates, chosen to model instances arising in applications.
\begin{enumerate}
\item \textbf{Exponential:} 
The diagonal matrix $\mA$ with $\mA(i,i) = \alpha^i$, which is an \mpo with bond dimension $D = 1$.
We set $\alpha = 0.7$ in our experiments.
This matrix is numerically low-rank.
\item \textbf{Inverse-Laplacian:}
The matrix $\mA=(\operatorname{tridiag}(-1,2,-1))^{-1}$, which is the inverse of the 1-dimensional discrete Laplacian.
It has an \mpo representation of bond dimension $D = 5$ \cite{kazeev2012low}.
This matrix has eigenvalues that decay at a polynomial rate,
and it is farther from a low-rank matrix.
\item \textbf{Staircase:}
The diagonal matrix of the form $$\mA = \diag(\underbrace{h_1,\ldots,h_1}_{S_1 \text{ times}},\underbrace{h_2,\ldots,h_2}_{S_2 \text{ times}},\ldots,\underbrace{h_R,\ldots,h_R}_{\text{$S_R$ times}}).$$
The spectrum of \mA is clustered, and resembles a staircase with $R$ steps. 
In our experiments, \(R=4\), with step lengths \((S_r)_{r=1}^4=(64,64,128,256)\) and step heights \((h_r)_{r=1}^4=(1,0.1,0.03,0.01)\), followed by a zero tail.
This matrix has an exact representation as an \mpo of bond dimension $2R-1$ \cite[pg.~2]{ryzhakov2022}.
The stairs are hard for Nystr\"om-based trace estimators to resolve.
\end{enumerate}
For each test matrix, the advantages of variance-reduced trace estimation
carry over to the tensor network setting. Replacing \mps Girard--Hutchinson (\cref{alg:mps-hutch})
with either \mps Nystr\"om++ (\cref{alg:mps-npp-gram}) or \mps XNysTrace
(\cref{alg:mps-xnystrace}) reduces the relative trace error
$
    |\tr(\mA)-\widehat{\tr}(\mA)|/|\tr(\mA)|
$
by several orders of magnitude. 
Furthermore, \mps\XNysTrace reliably outperforms \mps\nystrompp, particularly on challenging examples examples such as the staircase \mpo.
Thus, we recommend deploying \mps\XNysTrace in applications, despite our current lack of theoretical guarantees.

Within the range of parameters tested,
increasing the bond dimension \(\chi\) consistently improves the trace estimate.
Nonetheless, there are diminishing returns: once \(\chi\) is large enough, we recommend increasing the number of matrix--\mps accesses \(t\) instead of increasing \(\chi\).

\subsection{More algorithms: Regression and low-rank approximation for general matrices} \label{sec:more-algs}
\ENEdone 
\RAMdone
\CCdone
Let us mention a few other randomized linear algebra algorithms that can be implemented with \trps.
For tensor-structured regression problems, one may implement the sketch-and-solve algorithm \cite{sarlos06,drineas11} using an \trp,
with potential applications to linear system and eigenvalue solvers for tensor networks \cite{dolgov2012tt,kressner2014low,Bucci_sTTGMRES}.
To approximate general (non-psd) matrices, one may implement randomized low-rank approximation algorithms
such as the randomized SVD \cite{halko11} or the generalized Nystr\"om method \cite{Naka2020,tropp2017a,Clarkson09,woolfe08} with an \trp;
possible applications include compression of (sums and products of) tensor networks \cite{daas2021,daas25,camano25}.
Theoretical justification for these algorithms follows from the subspace injection results in \cref{sec:osi} together with relevant theorems from \cite{cam25}. 
For brevity, we will not elaborate further on these algorithms, except for some incidental remarks on the randomized SVD (\cref{sec:rsvd}).
 \section{The fourth-moment tensor of an \rmps} \label{sec:rmps-moments}

\ENEdone 
\RAMdone
\CCdone
This section develops an \emph{exact} description of the fourth-moment tensor of a real \rmps $\vomega \in \bbR^{d^n}$.
As a consequence, we can extract bounds on the variance of a quadratic form $\vomega^\top\mA\vomega$, proving \cref{thm:rmps_variance_intro} and providing theoretical support for the algorithms and applications discussed in \cref{sec:algs-applications}.
The results of this section are fundamental to \cref{sec:quadratic-form-condition}, where we use them to analyze these randomized linear algebra algorithms.
To simplify the presentation, we focus on the real case $\bbF = \bbR$.
In \cref{sec:complex}, we explain how the complex case $\bbF = \bbC$ follows
from a similar argument.

\subsection{Preliminaries and notation}

\ENEdone 
\RAMdone
\CCdone
As discussed in the introduction, an \mps $\vomega$ can be interpreted either as a tensor of size $d\times d \times \cdots \times d$ or a vector of length $d^n$.
To index such an object, we can either use $n$ individual indices $i_1 ,\ldots,i_n \in [d]$ or a single multi-index $\vi = (i_1,\ldots,i_n) \in [d]^n$.
We denote the entry of $\vomega$ at this position as either
\begin{equation*}
    \vomega(i_1,\ldots,i_n) \quad \text{or} \quad \vomega(\vi).
\end{equation*}
The \emph{fourth-moment tensor} of a real \rmps $\vomega$ is
\begin{equation}
    \tM \coloneqq \E[\vomega^{\otimes 4}] \in \bbR^{d^n \times d^n \times d^n \times d^n}.
\end{equation}
It is indexed by a four-tuple of multi-indices \((\vi, \vj, \vk, \vl)\).
Entrywise, its values are
\begin{equation*}
    \tM(\vi,\vj,\vk,\vl) = \E[\vomega(\vi) \vomega(\vj) \vomega(\vk) \vomega(\vl)].
\end{equation*}
Once we determine the fourth-moment tensor of an \rmps, we can compute the variance of a quadratic form $\vomega^\top\mA\vomega$.
Indeed,
\begin{equation*}
    \Var(\vomega^\top\mA\vomega) = \E[(\vomega^\top\mA\vomega)^2] - (\E[\vomega^\top\mA\vomega])^2 = \E[\langle \mA^{\otimes 2}, \vomega^{\otimes 4}\rangle] - \tr(\mA)^2 = \langle \mA^{\otimes 2},\tM\rangle - \tr(\mA)^2. 
\end{equation*}
Throughout this section, we make extensive use of tensor diagram notation (tensor diagrams) to simplify calculations. 
We refer the reader to \Cref{subsec:background} for a brief summary and \cite[$\mathsection$3.83]{Ballard25} for a detailed treatment.

\subsection{Core ingredient: The copy tensor}
\label{sec:copy}

\RAMdone
\ENEdone 
\CCdone
The \emph{copy tensor} is a tensor that can be used to simplify tensor diagrams \cite{ahle24,meng2026recursive}.
\begin{definition}[Copy tensor]
    The \emph{copy tensor of order \(t\)} with base dimension $d$, denoted \(\Delta^{(t)}\in\Rdtensor{t}\), is the tensor with entries
    \[
        \Delta^{(t)}(\alpha_1,\ldots,\alpha_t) = \indicator_{\{\alpha_1=\alpha_2=\cdots=\alpha_t\}} \quad \text{for any } \alpha_1,\ldots,\alpha_t \in [d].
    \]
\end{definition}
In special cases, the copy tensor reduces to familiar objects.
The order-one copy tensor is the all-ones vector 
\[
    \Delta^{(1)}(\alpha) = 1 \quad \text{for all } \alpha \in [d],
\]
and the order-two copy tensor is the identity matrix
\[
    \Delta^{(2)}(\alpha_1,\alpha_2) = \indicator_{\{\alpha_1 = \alpha_2\}} = \Id(\alpha_1,\alpha_2).
\]
In tensor diagram notation, the order-$t$ copy tensor is denoted as a white circle with $t$ legs, one for each index:
\begin{equation}
   \Delta^{(t)} =
   \begin{tikzpicture}[baseline=0ex]
       \foreach \x/\name/\lab/\hasLine in {
        0/i/\alpha_1/1,
        0.7/j/\alpha_2/1,
        1.4/k/\alpha_3/1,
        2.1/cdots/\cdots/1,
        2.8/l/\alpha_t/1}{
        \coordinate (\name) at (\x,0);
        
        \ifnum\hasLine=1
            \draw[leg] (\name) -- ++(0,0.4) node[above=2pt] {$\lab$};
        \else
            \path (\name) -- ++(0,0.4) node[above=2pt] {$\lab$};
        \fi
    }
\draw[leg] (i) .. controls +(0,-0.55) and +(0,-0.55)
              .. node[midway,copy,scale=1.8](copy_node) {} (l);
\draw[leg] (j) .. controls +(0,-0.55) and +(0,-0.55)
              .. node[midway,copy,scale=1.8] {} (cdots);
    \draw[leg] (k) -- (copy_node) 
;
\end{tikzpicture}
\end{equation}
Copy tensors help us express many tensors and tensor operations in a concise way.
For instance, suppose we are given the order-four tensor $\tT \in \Rdtensor{4}$, with entries \(\tT(i,j,k,\ell) = \indicator_{\{i=k\}}\indicator_{\{j=\ell\}}\).
Since we can factor
\[\tT(i,j,k,\ell) = \Delta^{(2)}(i,k) ~\cdot~ \Delta^{(2)}(j,\ell),\]
we can write \tT as the tensor product of two copy tensors.
As a tensor diagram, we write
\begin{equation} \label{eq:copy-indicator}
    \tT = \underbrace{
\begin{tikzpicture}[baseline=-1ex]
  \foreach \x/\name/\lab in {0/i/i,0.7/j/j,1.4/k/k,2.1/l/\ell}{
    \coordinate (\name) at (\x,0);
    \draw[leg] (\name) -- ++(0,-0.4) node[below=2pt] {$\lab$};
  }
  \draw[leg] (i) .. controls +(0,0.55) and +(0,0.55)
              .. node[midway,copy,scale=1.8] {} (k);
  \draw[leg] (j) .. controls +(0,0.55) and +(0,0.55)
              .. node[midway,copy,scale=1.8] {} (l);
\end{tikzpicture}
}_{\indicator_{\{i=k\}}\indicator_{\{j=\ell\}}}
\end{equation}
The copy tensor can also be used to represent sums of tensors.
Indeed, suppose that $\tB = \tA^{(1)} + \tA^{(2)} + \cdots + \tA^{(r)}$ is the sum of $r$ tensors $\tA^{(i)} \in \bbR^{d_1 \times \cdots \times d_t}$.
The sum of tensors is not an operation that is directly expressible using tensor diagrams, but there is a workaround using the copy tensor.
Namely, introduce the stacked tensor $\tA_{\rm stack} \in \bbR^{d_1 \times \cdots \times d_t \times r}$ by concatenating the summands along a new $(t+1)$st dimension:
\[
    \tA_{\rm stack}(:,:,\cdots,:,\alpha) = \tA^{(\alpha)} \quad \text{for } \alpha \in [r].
\]
Then the sum is 
\begin{equation*}
    \tB = \sum_{\alpha=1}^r \tA^{(\alpha)} = \sum_{\alpha=1}^r \tA_{\rm stack}(:,:,\cdots,:,\alpha) \cdot 1 = \sum_{\alpha=1}^r \tA_{\rm stack}(:,:,\cdots,:,\alpha) \Delta^{(1)}(\alpha).
\end{equation*}
We have inserted the copy tensor of order one, which is the all-ones vector.
We now recognize the sum $\tB = \sum_{\alpha=1}^r \tA^{(\alpha)}$ as the contraction of the stacked tensor $\tA_{\rm stack}$ with the order-one copy tensor:
\[
    \tB = \tnshow{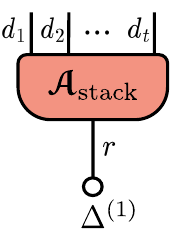}{0.13}
\]

Finally, let us note an important property of the copy tensor, often referred to as the \emph{fusion property}:

\begin{importedtheorem}[Copy tensors fuse together, \protect{\cite[Fig.~13,~14]{biamonte2011categorical}}]
    \label{impthm:copy-tensors-fuse}
    Let \(\Delta^{(t)}\) and \(\Delta^{(s)}\) be copy tensors with orders $t$ and $s$
    and common base dimension \(d\).
    Then the contraction of \(\Delta^{(t)}\) and \(\Delta^{(s)}\) along a single mode yields the copy tensor \(\Delta^{(t+s-2)}\in(\bbR^d)^{\otimes(t+s-2)}\).
\end{importedtheorem}
For example, for copy tensors $\Delta^{(4)}$ and $\Delta^{(3)}$ with common base dimension $d$, contracting the last mode of $\Delta^{(4)}$ with the first mode of $\Delta^{(3)}$ yields a copy tensor of order five:
\[
\begin{tikzpicture}[baseline=0ex]
       \foreach \x/\name/\lab/\hasLine in {
        0/i/\alpha_1/1,
        0.7/j/\alpha_2/1,
        1.4/k/\alpha_3/1,
        2.1/l/~/0,
        2.8/m/~/0}{
        \coordinate (\name) at (\x,0);
        \ifnum\hasLine=1
            \draw[leg] (\name) -- ++(0,0.4) node[above=2pt] {$\lab$};
        \else
            \path (\name) -- ++(0,0.4) node[above=2pt] {$\lab$};
        \fi
    }
\path[spath/save=copy11] (i) .. controls +(0,-0.55) and +(0,-0.55) .. node[midway,copy,scale=1.8,style=thick](copy_node1){} (m);
  \begin{pgfonlayer}{background}
      \draw[leg, spath/split at keep start={copy11}{0.5}, spath/use=copy11];
      \path[spath/save=copy12] (j) .. controls +(0,-0.55) and +(0,-0.55) .. (l);
      \draw[leg, spath/split at keep start={copy12}{0.5}, spath/use=copy12];
      \draw[leg] (k) -- (copy_node1) node[below left=4pt]{\(\Delta^{(4)}\)};
  \end{pgfonlayer}
\foreach \x/\name/\lab/\hasLine in {
        1.7/i2/~/0,
        2.4/j2/\alpha_4/1,
        3.1/k2/\alpha_5/1}{
        \coordinate (\name) at (\x,0);
        \ifnum\hasLine=1
            \draw[leg] (\name) -- ++(0,0.4) node[above=2pt] {$\lab$};
        \else
            \path (\name) -- ++(0,0.4) node[above=2pt] {$\lab$};
        \fi
    }
  \path[spath/save=copy2] (k2) .. controls +(0,-0.55) and +(0,-0.55) .. node[midway,copy,scale=1.8,style=thick](copy_node2){} (i2);
  \begin{pgfonlayer}{background}
      \draw[leg, spath/split at keep start={copy2}{0.5}, spath/use=copy2];
      \draw[leg] (j2) -- (copy_node2) node[below right=4pt]{\(\Delta^{(3)}\)};
  \end{pgfonlayer}
\draw[leg] (copy_node1) .. controls +(0, -0.55) and +(0.,-0.55) .. node[midway,below=1ex](){$d$} (copy_node2);
    \end{tikzpicture}
   =
   \begin{tikzpicture}[baseline=0ex]
       \foreach \x/\name/\lab/\hasLine in {
        0/i/\alpha_1/1,
        0.7/j/\alpha_2/1,
        1.4/k/\alpha_3/1,
        2.1/cdots/\alpha_4/1,
        2.8/l/\alpha_5/1}{
        \coordinate (\name) at (\x,0);
        
        \ifnum\hasLine=1
            \draw[leg] (\name) -- ++(0,0.4) node[above=2pt] {$\lab$};
        \else
            \path (\name) -- ++(0,0.4) node[above=2pt] {$\lab$};
        \fi
    }
\draw[leg] (i) .. controls +(0,-0.55) and +(0,-0.55)
              .. node[midway,copy,scale=1.8](copy_node) {} (l);
\draw[leg] (j) .. controls +(0,-0.55) and +(0,-0.55)
              .. node[midway,copy,scale=1.8] {} (cdots);
    \draw[leg] (k) -- (copy_node) 
    node[below=4pt]{\(\Delta^{(5)}\)}
    ;
\end{tikzpicture}
\]

\subsection{The fourth moment of a Gaussian tensor} \label{sec:gaussian}
\CCdone
\RAMdone
\ENEdone 
We begin our journey by studying the fourth-moment tensor \(\tN_d\) of a Gaussian vector \(\vg\in\bbR^d\).
We will then generalize to the fourth-moment tensor \(\tN_{d_1 \times \cdots \times d_t}\) of a dense Gaussian tensor \(\tG\in\bbR^{d_1 \times \cdots \times d_t}\).
This perspective will be useful because, for \(t \in \{2,3\}\), the tensor
\(\tN_{d_1 \times \cdots \times d_t}\) is the fourth moment of a \emph{single core}
of an \rmps. 
Finally, in preparation for computing the fourth-moment
tensor of an \rmps, we evaluate the fourth-moment tensor of the tensor product of Gaussian vectors.

\subsubsection{Special case: A Gaussian vector} \label{sec:gaussian-vector}

\RAMdone
\ENEdone 
\CCdone
Let $\tN_d = \E[\vg^{\otimes 4}]$ denote the fourth-moment tensor of a standard Gaussian vector $\vg\in \bbR^d$.
The entries of $\tN_d$ take a simple form:
\begin{equation} \label{eq:gaussian-tensor-indicators}
\tN_d(i,j,k,\ell)
= \indicator_{\{i = j\}} \indicator_{\{k = \ell\}} + \indicator_{\{i = k\}} \indicator_{\{j = \ell\}} + \indicator_{\{i = \ell\}} \indicator_{\{j = k\}}.
\end{equation}
This expression can either be obtained by a direct calculation or by the Isserlis--Wick theorem \cite{iss18}. 
Equivalently, \(\tN_d\) is the sum of three tensors represented using copy tensors:
\begin{equation*}
\tN_d
=
\underbrace{
\begin{tikzpicture}[baseline=-0.5ex]
  \foreach \x/\name/\lab in {0/i/i,0.7/j/j,1.4/k/k,2.1/l/\ell}{
    \coordinate (\name) at (\x,0);
    \draw[leg] (\name) -- ++(0,-0.4) node[below=2pt] {$\lab$};
  }
  \draw[leg] (i) .. controls +(0,0.55) and +(0,0.55)
              .. node[midway,copy,scale=1.8] {} (j);
  \draw[leg] (k) .. controls +(0,0.55) and +(0,0.55)
              .. node[midway,copy,scale=1.8] {} (l);
\end{tikzpicture}
}_{\indicator_{\{i=j\}}\indicator_{\{k=\ell\}}}
\,\,\,+\,\,\,
\underbrace{
\begin{tikzpicture}[baseline=-0.5ex]
  \foreach \x/\name/\lab in {0/i/i,0.7/j/j,1.4/k/k,2.1/l/\ell}{
    \coordinate (\name) at (\x,0);
    \draw[leg] (\name) -- ++(0,-0.4) node[below=2pt] {$\lab$};
  }
  \draw[leg] (i) .. controls +(0,0.55) and +(0,0.55)
              .. node[midway,copy,scale=1.8] {} (k);
  \draw[leg] (j) .. controls +(0,0.55) and +(0,0.55)
              .. node[midway,copy,scale=1.8] {} (l);
\end{tikzpicture}
}_{\indicator_{\{i=k\}}\indicator_{\{j=\ell\}}}
\,\,\,+\,\,\,
\underbrace{
\begin{tikzpicture}[baseline=-0.5ex]
  \foreach \x/\name/\lab in {0/i/i,0.7/j/j,1.4/k/k,2.1/l/\ell}{
    \coordinate (\name) at (\x,0);
    \draw[leg] (\name) -- ++(0,-0.4) node[below=2pt] {$\lab$};
  }
  \draw[leg] (i) .. controls +(0,0.55) and +(0,0.55)
              .. node[midway,copy,scale=1.8] {} (l);
  \draw[leg] (j) .. controls +(0,0.2) and +(0,0.2)
              .. node[midway,copy,scale=1.8] {} (k);
\end{tikzpicture}
}_{\indicator_{\{i=\ell\}}\indicator_{\{j=k\}}}
\end{equation*}
That is, we decompose \(\tN_d\) as the sum of the tensors
\begin{align*}
    \tW_d^{(1)} &\in \bbR^{d\times d\times d \times d} && \text{with entries } \tW_d^{(1)}(i,j,k,\ell) = \indicator_{\{i = j\}} \indicator_{\{k = \ell\}}, \\
    \tW_d^{(2)} &\in \bbR^{d\times d\times d \times d} && \text{with entries } \tW_d^{(2)}(i,j,k,\ell) = \indicator_{\{i = k\}} \indicator_{\{j = \ell\}}, \text{ and} \\
    \tW_d^{(3)} &\in \bbR^{d\times d\times d \times d} && \text{with entries } \tW_d^{(3)}(i,j,k,\ell) = \indicator_{\{i = \ell\}} \indicator_{\{ j = k\}}.
\end{align*}

In order to tap into the power of tensor diagrams, we want to represent \(\tN_d\) using a single tensor network.
So, we follow the approach discussed in \cref{sec:copy} where we stack these three tensors together and contract with a copy tensor.
That is, we define the order-5 \emph{Wick tensor} \(\tW_d\in\bbR^{d \times d \times d \times d \times 3}\) by stacking together the tensors $\tW_d^{(\alpha)}$.  That is,
\begin{equation} \label{eq:wick-tensor}
    \tW_d(:,:,:,:,\alpha) \defeq \tW_d^{(\alpha)} \quad \text{for } \alpha=1,2,3.
\end{equation}
The Gaussian fourth-moment tensor \(\tN_d\) is the result of summing the Wick tensor over its last index:
\begin{equation*}
    \tN_d = \sum_{\alpha=1}^3 \tW_d(:,:,:,:,\alpha) = \sum_{\alpha=1}^3 \tW_d(:,:,:,:,\alpha) \Delta^{(1)}(\alpha).
\end{equation*}
In a diagram, the Gaussian fourth-moment tensor is the contraction of the Wick tensor with a copy tensor of order one:
\begin{equation*}
    \tN_d = \E[\vg^{\otimes 4}]=\tnshow{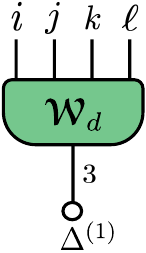}{.1}
\end{equation*}

\subsubsection{General case: A Gaussian tensor} \label{sec:gaussian_tensor}
\RAMdone
\ENEdone 
We now turn our attention to a tensor $\tG \in \bbR^{d_1\times d_2 \times \cdots \times d_t}$ of general dimensions \(d_1\times\cdots\times d_t\) populated with iid $\cN_{\bbR}(0,1)$ entries.
Up to scaling, the site tensors of an \rmps take this form with \(t\in\{2,3\}\).
Our task is to compute the fourth-moment tensor of $\tG$, defined as $\tN_{d_1\times \cdots \times d_t} \coloneqq \E[\tG^{\otimes 4}]$.
To proceed, let $\vi = (i_1,\ldots,i_t)$, $\vj = (j_1,\ldots,j_t)$, $\vk = (k_1,\ldots,k_t)$, and $\vl = (\ell_1,\ldots,\ell_t)$ denote multi-indices, so that
\begin{equation*}
    \tN_{d_1\times \cdots \times d_t}(\vi,\vj,\vk,\vl) \coloneqq\E[\tG(\vi)\tG(\vj)\tG(\vk)\tG(\vl)].
\end{equation*}
Applying the Isserlis--Wick theorem again, we expand the entries of the fourth-moment tensor as
\begin{equation*}
\E[\tG(\vi)\tG(\vj)\tG(\vk)\tG(\vl)]= \indicator_{\{\vi = \vj\}} \indicator_{\{\vk = \vl\}} + \indicator_{\{\vi = \vk\}} \indicator_{\{\vj = \vl\}} + \indicator_{\{\vi = \vl\}} \indicator_{\{\vj = \vk\}}.
\end{equation*}
Each indicator factors as a product.
For instance,
\begin{equation*}
    \indicator_{\{\vi = \vj\}}
    = \indicator_{\{(i_1,\ldots,i_t) = (j_1,\ldots,j_t)\}}
    = \indicator_{\{i_1=j_1\}} \times \cdots \times \indicator_{\{i_t=j_t\}}.
\end{equation*}
Inserting these identities and rearranging, we obtain
\begin{align*}
    \E[\tG(\vi)\tG(\vj)\tG(\vk)\tG(\vl)] &= \prod_{p=1}^t \indicator_{\{i_p = j_p\}}\indicator_{\{k_p = \ell_p\}} + \prod_{p=1}^t \indicator_{\{i_p = k_p\}}\indicator_{\{j_p = \ell_p\}} + \prod_{p=1}^t \indicator_{\{i_p = \ell_p\}}\indicator_{\{j_p = k_p\}} \\
    &= \prod_{p=1}^t \tW_{d_p}^{(1)}(i_p,j_p,k_p,\ell_p) + \prod_{p=1}^t \tW_{d_p}^{(2)}(i_p,j_p,k_p,\ell_p) + \prod_{p=1}^t \tW_{d_p}^{(3)}(i_p,j_p,k_p,\ell_p)\\
    &= \sum_{\alpha=1}^3\prod_{p=1}^t \tW_{d_p}(i_p,j_p,k_p,\ell_p,\alpha) 
\end{align*}
This expression contains a sum of tensors, so we can express it concisely using a copy tensor:
\[
    \tN_{d_1 \times \cdots \times d_t}(\vi,\vj,\vk,\vl)
    = \sum_{\alpha_1,\ldots,\alpha_t=1}^3
    \prod_{p=1}^t \tW_{d_p}(i_p,j_p,k_p,\ell_p,\alpha_p)
    \Delta^{(t)}(\alpha_1,\ldots,\alpha_t).
\]
That is, we are contracting \(t\) distinct Wick tensors with a single copy tensor.
Written using a tensor diagram,
\begin{equation*}
\tN_{d_1\times \cdots \times d_t}
    = \tnshow{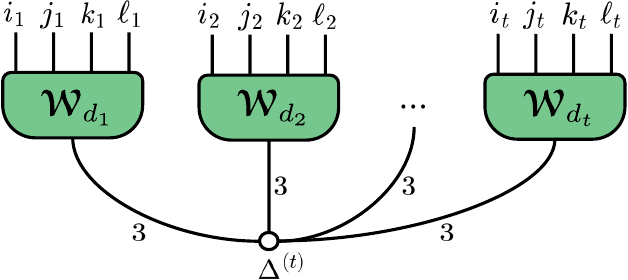}{.45}
\end{equation*}
By the fusion property (\cref{impthm:copy-tensors-fuse}), we can reformat this tensor in a suggestive way: 
\begin{equation} \label{eq:gaussian-tensor-diagram}
\tN_{d_1\times \cdots \times d_t}
   = \tnshow{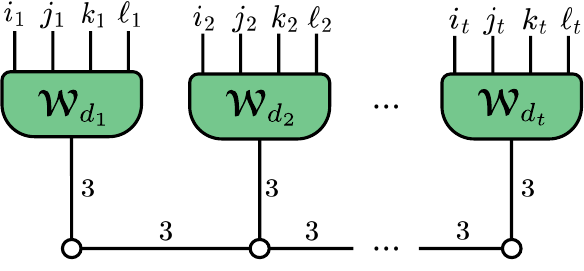}{.425}
\end{equation}

\subsubsection{Consequence: Tensor product of standard Gaussian vectors}
\label{sec:fourth-moment-of-kr-gaussian}

In anticipation of what is to follow, we now analyze the fourth-moment tensor for the vector $\vh = \vg^{(1)} \otimes \cdots \otimes \vg^{(t)}$ that is a tensor product of $t$ independent standard Gaussian vectors with dimensions $d^{s_1},\ldots,d^{s_t}$.
To start, consider a single base Gaussian vector $\vg^{(q)}\in\bbR^{d^{s_q}}$, which is a vector of exponentially large dimension $d^{s_q}$ and iid entries.
We can identify $\vg^{(q)}$ with a standard Gaussian tensor in $(\bbR^d)^{\otimes s_q}$.
Consequently, by \cref{eq:gaussian-tensor-diagram}, the fourth-moment tensor of $\vg^{(q)}$ is 
\begin{equation*}
\E[(\vg^{(q)})^{\otimes 4}]=\tnshow{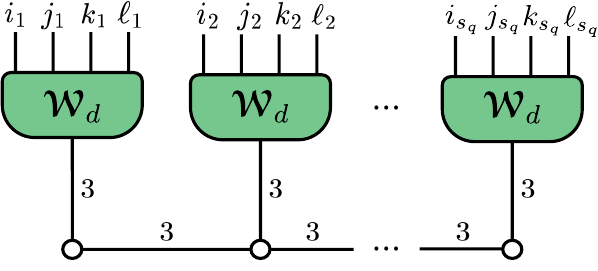}{0.425}
\end{equation*}
The fourth-moment tensor of $\vh = \vg^{(1)} \otimes \cdots \otimes \vg^{(t)}$ is then the tensor product of the fourth-moment tensors of the vectors $\vg^{(q)}$, corresponding to the following tensor diagram 
\begin{equation}
    \E[\vh^{\otimes 4}] =
  \tnshow{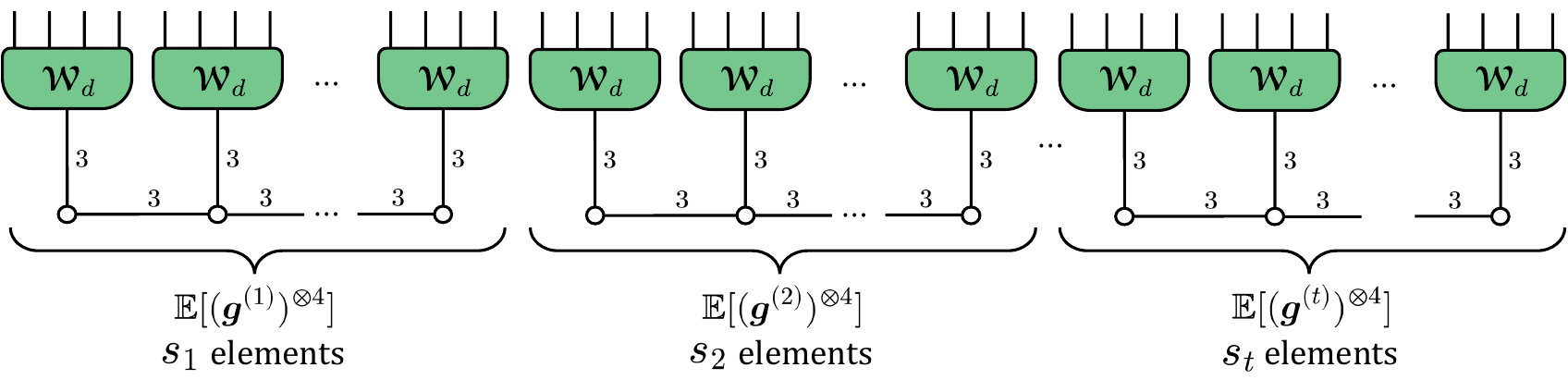}{0.9}
    \label{eq:fourth-moment-of-h}
\end{equation}
This same tensor diagram will appear in our computation of the fourth-moment tensor of an \rmps in \cref{sec:fourth_moment_interpretation}.

\subsection{The fourth moment of an \rmps}
\label{ssec:rmps_fourth_moment}

\RAMdone
\ENEdone 
\CCdone
Recall that a Gaussian \rmps is a tensor network of the form
\begin{equation*}
    \vomega(i_1,\ldots,i_n) = \sum_{\alpha_1,\ldots,\alpha_{n-1} = 1}^\chi \tG_1(i_1,\alpha_1)\tG_2(\alpha_1,i_2,\alpha_2)\cdots \tG_{n-1}(\alpha_{n-2},i_{n-1},\alpha_{n-1})\tG_n(\alpha_{n-1},i_n).
\end{equation*}
Each \emph{site tensor} $\tG_p$ is populated with iid $\cN_{\bbR}(0,\sigma^2)$ entries with $\sigma^2 = \chi^{-1/2}$ when $p \in \{1,n\}$ and $\sigma^2 = \chi^{-1}$ otherwise.
Graphically, we have
\begin{equation*}
\label{eq:mps_diagram}
  \vomega = 
  \tnshow[5ex]{E5}{.5}.
  \vspace{-2em}
\end{equation*}
In this section, we compute the fourth-moment tensor $\tM \coloneqq \E[\vomega^{\otimes 4}]$ of \vomega.
We start this calculation in \cref{sec:fourth_moment_computation}.
Then, in \cref{sec:fourth_moment_interpretation}, we recognize that this fourth-moment tensor coincides with the fourth-moment tensor of a probabilistic mixture of Gaussian--Kronecker vectors.

\subsubsection{The fourth moment of an \rmps as a tensor network} \label{sec:fourth_moment_computation}

\begin{figure}[t]
    \centering
    \captionsetup[subfigure]{justification=centering}
\begin{subfigure}[b]{0.3\textwidth}
        \centering
\(\tnshow{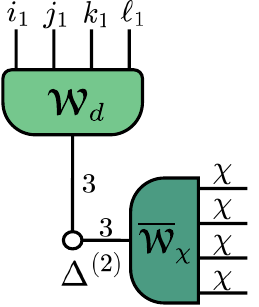}{.65}\)
        \caption{First site, \(\E[\tG_1^{\otimes 4}]\)\\~}
        \label{fig:fourth-moment-first-site}
    \end{subfigure}
    \hfill
\begin{subfigure}[b]{0.3\textwidth}
        \centering
\(\tnshow{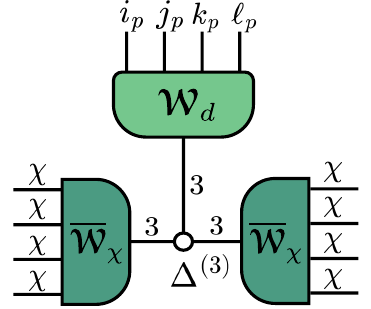}{.9}\)
        \caption{Interior sites, \(\E[\tG_p^{\otimes 4}]\),\\\(p\in\{2,\ldots,n-1\}\)}
        \label{fig:fourth-moment-middle-site}
    \end{subfigure}
    \hfill
\begin{subfigure}[b]{0.3\textwidth}
        \centering
\(\tnshow{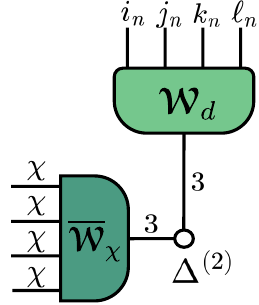}{.65}\)
        \caption{Last site, \(\E[\tG_n^{\otimes 4}]\)\\~}
        \label{fig:fourth-moment-last-site}
    \end{subfigure}
    
    \caption{
        Tensor networks for the first, interior, and last sites of the fourth moment of a Gaussian \rmps.
}
    \label{fig:fourth-moment-of-sites}
\end{figure}

\RAMdone
\ENEdone
\CCdone
The fourth-moment tensor of a Gaussian \rmps follows from the Gaussian fourth-moment calculations in \cref{sec:gaussian}.
Our analysis had assumed that our Gaussian tensor had entries with variance 1.
To account for the fact that the entries of an \rmps' site tensors \(\tG_1,\ldots,\tG_n\) have variance either \(\chi^{-1}\) or \(\chi^{-1/2}\), we introduce the \emph{scaled Wick tensor}
\begin{equation*}
    \overline{\tW}_\chi \coloneqq \frac{1}{\chi}\tW_\chi 
\end{equation*}
Recall the (unscaled) Wick tensor $\tW_\chi$ was defined in \cref{sec:gaussian-vector}.
We can use a combination of scaled and unscaled Wick tensors to express the fourth moment of the site tensors of an \rmps, as shown in \cref{fig:fourth-moment-of-sites}.
Recall that the first site tensor has dimension \(\tG_1\in\bbR^{d \times \chi}\), the interior site tensors have dimension \(\tG_p \in \bbR^{\chi \times d \times \chi}\), and the last site tensor has dimension \(\tG_n\in\bbR^{\chi \times d}\).

The fourth-moment tensor of an \rmps is then given by contracting fourth-moment tensors for the individual sites together.
We find that 
\begin{align*}
    \tM &= \tnshow{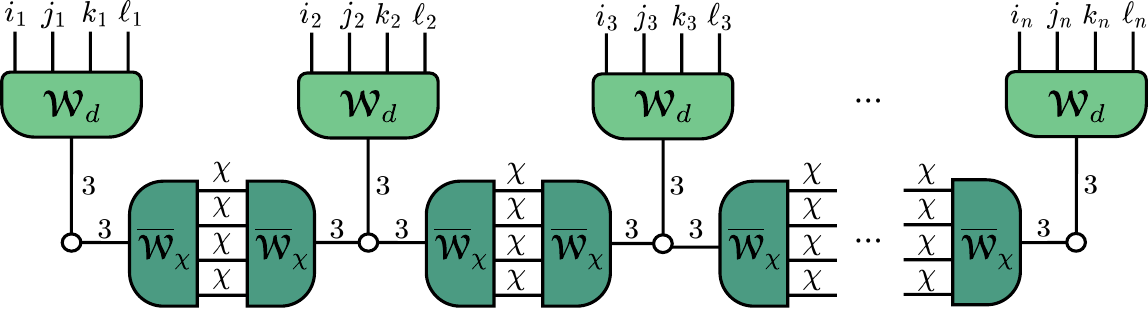}{.8}
\end{align*}
We can simplify this tensor network further by contracting each pair of scaled Wick tensors \(\bar\tW_\chi\) together, resulting in the small matrix
\begin{equation*}
    \tnshow[.72ex]{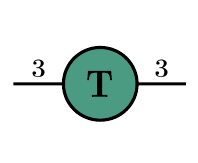}{.145}
    =\tnshow[.75ex]{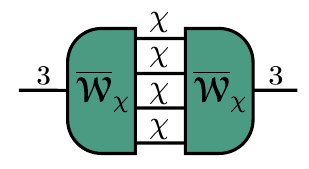}{.22} \in \bbR^{3\times 3}.
\end{equation*}
The entries of \mT are \(\mT(\alpha,\beta)=\frac{1}{\chi^2} \langle \tW_\chi^{(\alpha)},\tW_\chi^{(\beta)}\rangle\).
Direct computation
yields the formula
\begin{equation} \label{eq:T_def}
    \mT = \begin{bmatrix}
        1 & 1/\chi & 1/\chi \\
        1/\chi & 1 & 1/\chi \\
        1/\chi & 1/\chi & 1
    \end{bmatrix}.
\end{equation}
We have established the following theorem.

\begin{theorem}[Fourth-moment tensor of an \rmps]
    \label{thm:rmps-moment-tensor}
    When $\bbF = \bbR$, the fourth-moment tensor of a length $n$ \rmps with bond dimension $\chi$ is 
\begin{equation*}
        \tM = \tnshow{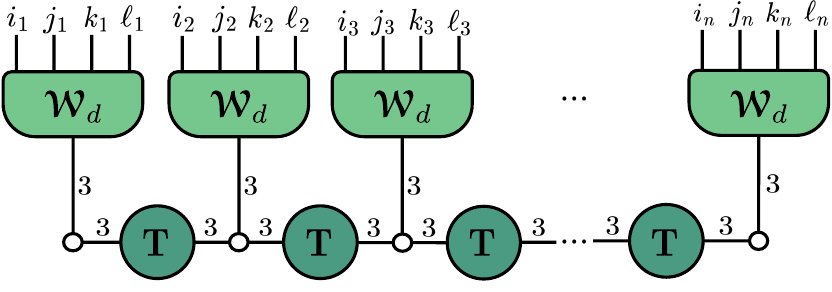}{.55},
\end{equation*}
where \(\tW_d\) is the Wick tensor \cref{eq:wick-tensor} and \(\mT\) is defined in \cref{eq:T_def}.
\end{theorem}

\subsubsection{Fourth moment of an \rmps as a probabilistic mixture} \label{sec:fourth_moment_interpretation}

\ENEdone 
\CCdone
In \cref{thm:rmps-moment-tensor}, we have obtained an exact description of the fourth-moment tensor of an \rmps as a tensor diagram.
However, it is not obvious how to use this diagram to bound the variance of the quadratic form \(\vomega^\top\mA\vomega\).
In this section, we use this tensor diagram to construct a random vector \(\vh_{\vb}\) with two crucial properties.
First, the fourth-moment tensor for \(\vh_{\vb}\) will coincide with that of an \rmps \vomega, so that $\Var(\vomega^\top\mA\vomega) = \Var(\vh_{\vb}^\top\mA\vh_{\vb}^{\vphantom{\top}})$.
Second, by a direct argument, we can bound the variance of the quadratic form \(\vh_{\vb}^\top\mA\vh_{\vb}^{\vphantom{\top}}\).

We begin by analyzing the diagram in \cref{thm:rmps-moment-tensor} for the fourth-moment of an \rmps.
Observe that the matrix \mT can be realized as a convex combination
\begin{equation}
    \mT = \frac{1}{\chi} \mT_1 + \left( 1 - \frac{1}{\chi} \right)\mT_0
    \label{eq:T-matrix-convex-comb}
\end{equation}
where
\begin{equation}
    \mT_1 \coloneqq \vone\vone^\top =\tnshow{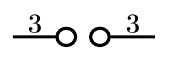}{.155} \quad \text{and} \quad \mT_0 \coloneqq \Id = \tnshow{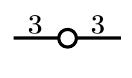}{.125}.
    \label{eq:t-matrices-diagrams}
\end{equation}
The convex combination in \cref{eq:T-matrix-convex-comb} suggests a probabilistic interpretation where we view \mT as the expected value $\mT = \E[\mS]$ of a random matrix \mS that takes values $\mT_0$ and $\mT_1$ with appropriate probabilities:
\begin{equation*}
    \mS = \begin{cases}
        \mT_1 & \text{with probability } 1/\chi, \\
        \mT_0 & \text{with probability } 1-1/\chi.
    \end{cases}
\end{equation*}
Such a random matrix can be generated by drawing a random variable $b \sim \bernoulli(1/\chi)$ and setting $\mS \coloneqq \mT_b$.
The variable $b$ represents a biased coin flip that has probability $1/\chi$ of being heads.

Let us now apply this insight at the level of the entire fourth-moment tensor.
Draw random variables
\begin{equation} \label{eq:random_cuts}
    \vb = (b_1,\ldots,b_{n-1}) \quad \text{with } b_1,\ldots,b_{n-1} \sim \bernoulli(1/\chi) \text{ iid,}
\end{equation}
and consider the random tensor: 
\begin{equation} \label{eq:tM_b_def}
    \tM_{\vb} \coloneqq \tnshow{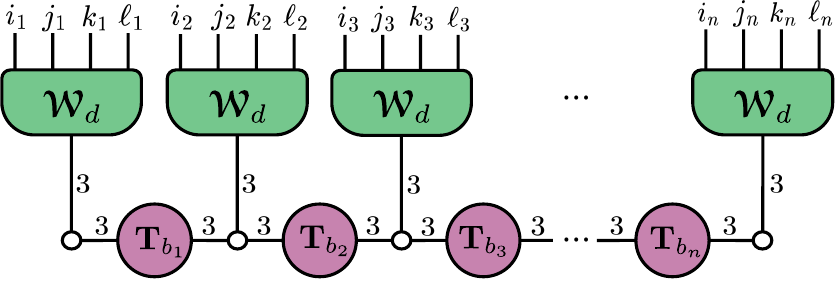}{.55}
\end{equation}
By linearity of expectation,
\begin{equation*}
    \E[\tM_{\vb}] = \tM,
    \label{eq:expected-tMb-is-tM}
\end{equation*}
where the expectation is taken over the randomness in the biased coin flips $\vb$.

Now, fix a realization \(\vb\in\{0,1\}^{n-1}\), and let us examine the resulting tensor \(\tM_{\vb}\).
Whenever \(b_p=1\), we have \(\mT_{b_p} = \mT_1\).
As shown above in \cref{eq:t-matrices-diagrams}, \(\mT_1\) is the tensor product of two copy tensors.
Substituting \(\mT_{b_p} = \mT_1\) into \cref{eq:tM_b_def} has the effect of ``cutting'' the network into two separate parts between sites $p$ and $p+1$: 
\begin{equation*}
\tnshow{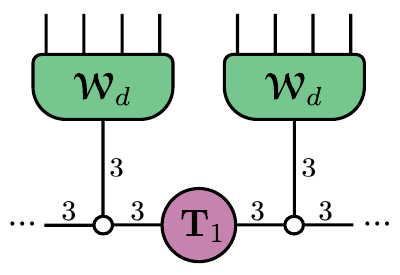}{.275}
=
\tnshow{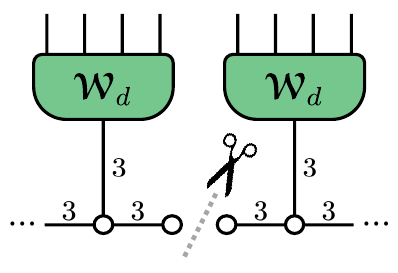}{.275}
=
\tnshow{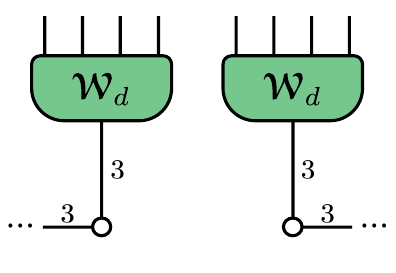}{.275}
\end{equation*}
We have used the copy tensor fusion property (\cref{impthm:copy-tensors-fuse}) to simplify the diagram.

When we have \(b_p=0\), the matrix \(\mT_{b_p} = \mI\) is the identity matrix, and the tensor network remains connected between sites $p$ and $p+1$: 
\begin{equation*}
\tnshow{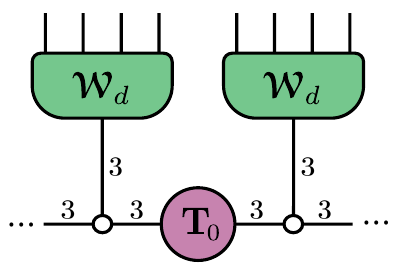}{.275}
=
\tnshow{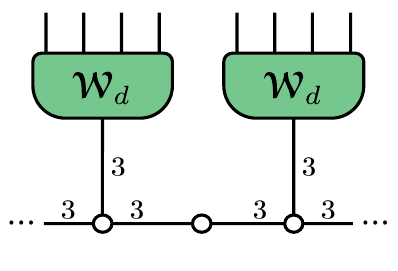}{.275}
=
\tnshow{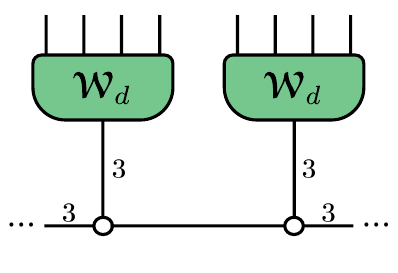}{.275}
\end{equation*}
We have again used the fusion property.

Repeating this process for each matrix $\mT_{b_p}$, we see that \(\tM_{\vb}\) splits into $t = 1 + |\{p : b_p = 1\}|$ pieces.
For instance, with $n = 6$ and $\vb = (0,0,1,0,1)$, we have
\begin{equation}
    \tM_{\vb}=\tnshow{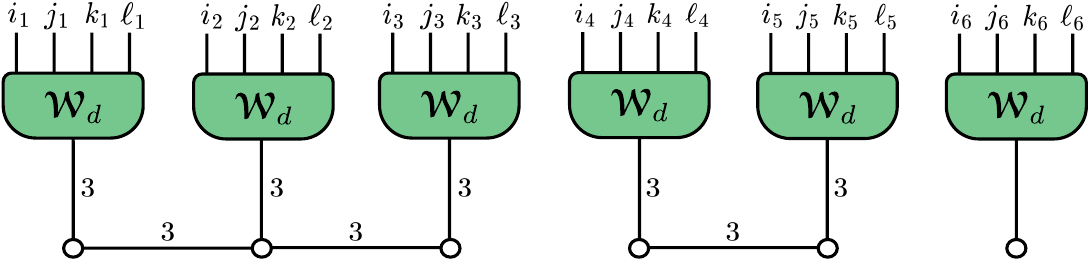}{0.7}
    \label{eq:example-cut-tensor}
\end{equation}
Recalling the tensor network diagram shown in \cref{eq:fourth-moment-of-h}, we recognize that the network in \cref{eq:example-cut-tensor} equals the fourth-moment tensor of a Gaussian Kronecker vector \(\vh = \vg^{(1)} \otimes \vg^{(2)} \otimes \vg^{(3)}\) where the base vectors have dimensions \(\vg^{(1)}\in\bbR^{d^3}\), and \(\vg^{(2)}\in\bbR^{d^2}\), and \(\vg^{(3)} \in \bbR^{d}\).

Let us generalize and formalize this observation.
For any \(t\in\{1,\ldots,n\}\), consider any vector \(\vb\in\{0,1\}^{n-1}\) with \(t-1\) entries equal to 1.
Let \(p_1,\ldots,p_{t-1}\) be the indices where \(\vb\) is nonzero.
Then, we can interpret \((p_1,\ldots,p_{t-1})\) as instructions on how to cut the sequence \((1,\ldots,n)\) into \(t\) disjoint subsequences:
\begin{equation*}
    (1,\ldots,n)
    = \underbrace{(1,\ldots,p_1)}_{s_1 \text{ elements}}
    \cup \underbrace{(p_1+1,\ldots,p_2)}_{s_2 \text{ elements}} \cup
    \cdots
    \cup \underbrace{(p_{t-1}+1,\ldots,n)}_{s_{t} \text{ elements}}.
\end{equation*}
Here, \(s_q\) denotes the length of the \(q\)th subsequence.
We can then express \(\tM_{\vb}\) in terms of this composition as 
\begin{equation*}
    \tnshow{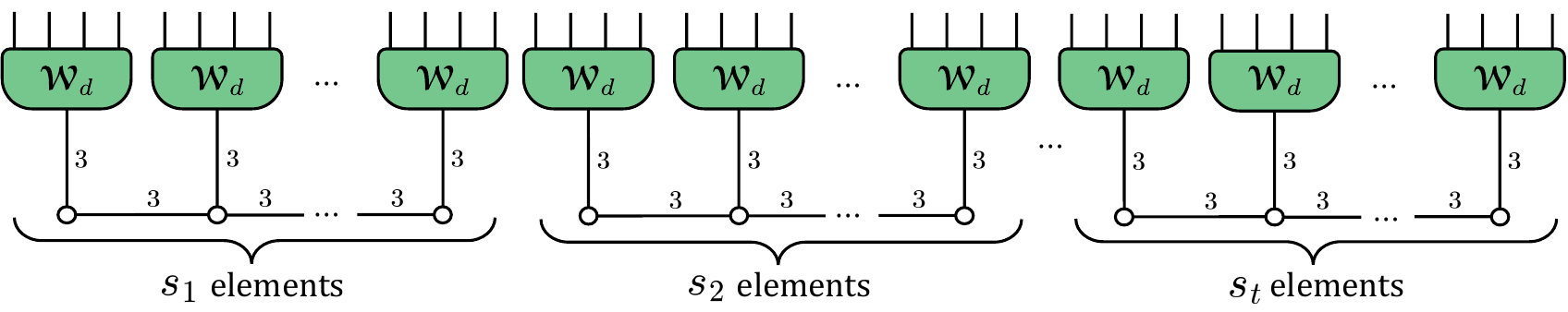}{1}
\end{equation*}

The previous display coincides with the tensor diagram \cref{eq:fourth-moment-of-h} for a tensor product of Gaussians of lengths $d^{s_1},\ldots,d^{s_t}$.
We conclude that \(\tM_{\vb}\) is the fourth-moment tensor for the Gaussian-Kronecker vector
\begin{equation} \label{eq:cut_gaussian}
    \E\big[\vh_{\vb}^{\otimes 4} \mid \vb\big] = \tM_{\vb}^{\vphantom{\otimes 4}} \quad \text{where }\vh_{\vb} \coloneqq \vg^{(1)} \otimes \vg^{(2)} \otimes \cdots \otimes \vg^{(t)} \text{ and } \vg^{(q)} \sim \cN_{\bbR}(\vzero,\Id_{d^{s_q}}).
\end{equation}
Since both the number of terms $t$ and the sizes of the individual terms $\vg^{(q)}$ are random, we call $\vh_{\vb}$ a \emph{randomly fractured Gaussian vector}.
Combining this expression with the identity \(\E[\tM_{\vb}] = \tM\) and the tower rule, we conclude that
\begin{equation*}
\E\big[\vh_{\vb}^{\otimes 4}\big] = \E\big[\tM_{\vb}^{\vphantom{\otimes 4}}\big] = \tM.
\end{equation*}
We have shown that the \rmps \vomega has the same fourth-moment tensor as the fractured Gaussian vector \(\vh_{\vb}\).
\begin{theorem}[\rmps as a randomly fractured Gaussian]
    \label{thm:gaussian-surrogate}
    Let \(\vomega\) be an \rmps of bond dimension \(\chi\).
    Sample \(b_1,\ldots,b_{n-1} \sim \bernoulli(1/\chi)\), let \(\vs=(s_1,\ldots,s_t)\) be the corresponding composition, and construct the fractured Gaussian vector
    \[
        \vh_{\vb}^{\vphantom{(1)}} \coloneqq \vg^{(1)} \otimes \vg^{(2)} \otimes \cdots \otimes \vg^{(t)} \quad \text{where } \vg^{(q)} \sim \cN_{\bbR}(\vzero,\Id_{d^{s_q}}).
    \]
    Then, the first four moments of \vomega and \(\vh_{\vb}\) agree:
    \[
        \E[\vomega]
        =\E[\vh_{\vb}] 
        = \vec0,
        \quad
        \E[\vomega\vomega^\top]
        =\E[\vh_{\vb}^{\vphantom{\top}}\vh_{\vb}^\top] 
        = \mI, \quad 
        \E[\vomega^{\otimes 3}]
        =\E[\vh_{\vb}^{\otimes 3}] 
        = \vec0,
        \quad
\E[\vomega^{\otimes 4}]
        =\E[\vh_{\vb}^{\otimes 4}].
    \]
    In particular, we have \(\Var(\vomega^\top\mA\vomega) = \Var(\vh_{\vb}^\top\mA\vh_{\vb}^{\vphantom{*}})\) for all matrices \(\mA\in\bbR^{d^n \times d^n}\).
\end{theorem}

\Cref{thm:gaussian-surrogate} justifies our claim that an \rmps behaves like a standard Gaussian vector when, and only when, \(\chi \gtrsim n\).
Indeed, the fractured Gaussian vector \(\vh_{\vb}\) is cut between each pair of adjacent sites with probability \(1/\chi\).
So, with probability \((1-1/\chi)^n\), no cuts occur.
We identify three cases:
\begin{itemize}
    \item When \(\chi \ll n\), the probability of having no cuts tends to zero in the limit $n\to\infty$.
    As such, an \rmps acts like a Gaussian--Kronecker vector in this regime.
    \item When \(\chi \asymp n\), there is a constant probability of having no cuts, in which case \(\vh_{\vb}\) has a standard Gaussian distribution.  In this case, an \rmps often acts like a standard Gaussian vector.
    \item When \(\chi \asymp n/\eps\), the probability of having no cuts is \(1-\eps\).  Thus, \(\vh_{\vb}\) has a standard Gaussian distribution a majority of the time; the \rmps acts even more like a standard Gaussian vector.
\end{itemize}

\subsection{The quadratic form induced by a random MPS}

\ENEdone
\CCdone
We have computed the exact variance of a Gaussian \rmps and identified a surrogate vector \(\vh_{\vb}\) with the same fourth-moment tensor.
Variance bounds for the quadratic form $\vomega^\top\mA\vomega$ (\cref{thm:rmps_variance_intro}) now follow almost immediately.
First, we import a variance inequality for tensor products, due to Meyer \& Avron \cite{meyer_hutchinsons_2025}:
\begin{importedtheorem}[Variance bound for tensor products \protect{\cite[Thm.~1.1]{meyer_hutchinsons_2025}}] \label{impthm:meyer_avron}
    Let $\vh \coloneqq \vg^{(1)}\otimes \cdots \otimes \vg^{(t)}$ be a tensor product of independent Gaussian vectors $\vg^{(q)} \sim \cN_{\bbR}(\vzero,\Id_{d_q})$, and let $\mA$ be a square matrix.
    Then
\begin{equation} \label{eq:meyer-avron}
        \Var(\vh^\top \mA\vh) \le 2^t \norm{\mA}_{\rm F}^2 + (3^t - 2^t - 1) \norm{\mA}_*^2.
    \end{equation}
\end{importedtheorem}

Meyer \& Avron's original work contains an \emph{exact} expression for the variance $\Var(\vh^\top \mA\vh)$.
This expression involves a weighted sum of the squared Frobenius norm of all possible partial traces of a symmetrized version of the matrix \mA.
We obtain a simple bound by noting that the Frobenius norm of any partial trace is bounded by the nuclear norm.
Meyer \& Avron's paper also focused on the case of Gaussian vectors with common dimension $d_1 = \cdots = d_t$; the same proof carries over to general dimensions.

\Cref{impthm:meyer_avron} yields a variance bound for the quadratic form determined by a Gaussian \rmps.

\begin{theorem}[Variance bound for a Gaussian \rmps] \label{thm:rmps_variance}
    Let $\vomega$ be a Gaussian \rmps of bond dimension \(\chi\), and let $\mA$ be a real square matrix of dimension $d^n$.
    Then
\begin{equation*}
        \Var(\vomega^\top \mA \vomega) \le 2\left(1 + \frac{1}{\chi}\right)^{n-1} \norm{\mA}_{\rm F}^2 + \left[3 \left(1 + \frac{2}{\chi}\right)^{n-1} - 2 \left(1 + \frac{1}{\chi}\right)^{n-1} - 1\right] \norm{\mA}_*^2.
    \end{equation*}
\end{theorem}

\begin{proof}
    Introduce the fractured Gaussian vector \(\vh_{\vb}\) from \cref{thm:gaussian-surrogate}.
    Let $\vs = (s_1,\ldots,s_t)$ denote the composition associated with \(\vh_{\vb}\).
    The size $t$ of the composition is a random variable, and $t-1$ satisfies a binomial distribution with $n-1$ trials and success probability $1/\chi$.
    In particular,
\begin{equation} \label{eq:binomial_power_expectation}
        \E[a^t] = a\left(1 + \frac{a-1}{\chi} \right)^{n-1} \quad \text{for any } a \in\bbR.
    \end{equation}
By \cref{impthm:meyer_avron}, we have the bound
\begin{equation*}
        \Var(\vh_{\vb}^\top \mA \vh_{\vb}^{\vphantom{\top}} \mid \vb) \le 2^t \norm{\mA}_{\rm F}^2 + (3^t - 2^t - 1) \norm{\mA}_*^2.
    \end{equation*}
The theorem follows by applying the law of total variance with respect to $\vb$ and invoking \cref{eq:binomial_power_expectation}.
\end{proof}

\subsection{The complex case} \label{sec:complex}

\ENEdone 
\RAMdone
\CCdone
The complex case is similar to the real case, up to a few small modifications.
Let $\tG_{\bbC}$ be a standard Gaussian tensor over the complex field of dimensions $d_1\times \cdots \times d_t$.
Its fourth-moment tensor is defined as
\begin{equation} \label{eq:complex-fourth-moment}
    \tN^{\bbC}_{d_1\times \cdots \times d_t} \coloneqq \E[\tG_{\bbC}^{\otimes 2} \otimes \overline{\tG}_{\bbC}^{\otimes 2}].
\end{equation}
By the Isserlis--Wick theorem, the entries of \(\tN_{d_1\times \cdots \times d_t}^{\bbC}\) are
\begin{equation*}
    \tN^{\bbC}_{d_1\times \cdots \times d_t}(\vi,\vj,\vk,\vl) = \E[\tG(\vi)\tG(\vj)\overline{\tG(\vk) \tG(\vl)}] = \indicator_{\{\vi = \vk\}} \indicator_{\{\vj = \vl\}} + \indicator_{\{\vi = \vl\}} \indicator_{\{\vj = \vk\}}.
\end{equation*}
Observe that there are now two terms, rather than the three in the expression \cref{eq:gaussian-tensor-indicators} in the real case.
Proceeding as we did in \cref{sec:gaussian_tensor}, we realize $\tN^{\bbC}_{d_1\times \cdots \times d_t}$ as a tensor diagram 
\begin{equation*}
    \tN^{\bbC}_{d_1 \times \cdots \times d_t}
    = \tnshow{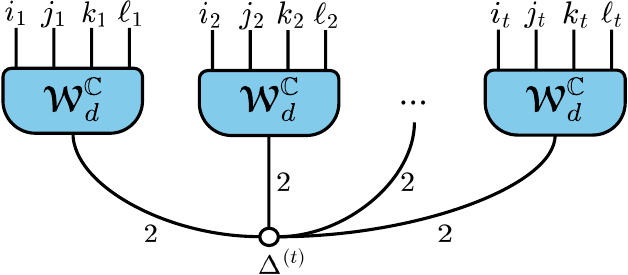}{0.4}= \tnshow{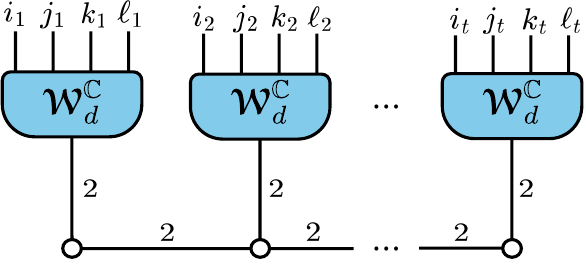}{0.4},
\end{equation*}
where the complex Wick tensor $\tW_d^{\bbC}\in \{0,1\}^{d\times d\times d\times d\times 2}$ is now the stacking of just two tensors:
\begin{equation*}
\tW_d^{\bbC}(i,j,k,\ell,1) = \indicator_{\{i = k\}} \indicator_{\{j = \ell\}} \quad \text{and} \quad \tW_d^{\bbC}(i,j,k,\ell,2) = \indicator_{\{i = \ell\}} \indicator_{\{j = k\}}.
\end{equation*}
Observe that $\tW_d^{\bbC}$ has fifth mode of dimension two, not three.
The computation of the fourth-moment tensor of a complex \rmps 
\begin{equation*}
    \tM^{\bbC} = \E[\vomega^{\otimes 2} \otimes \overline{\vomega}^{\otimes 2}]
\end{equation*}
proceeds as in the real case.
Since the Wick tensors have final dimension two instead of three, the matrix $\mT^{\bbC}$ now has dimension $2 \times 2$:
\begin{equation*}
    \mT^{\bbC} = \begin{bmatrix}
        1 & 1/\chi \\
        1/\chi & 1
    \end{bmatrix}.
\end{equation*}
The rest of the argument is identical.

\begin{theorem}[Fourth-moment tensor of an \rmps, complex case]
    When $\bbF = \bbC$, the fourth-moment tensor of an \rmps with bond dimension \(\chi\) is
\begin{equation*}
        \tM^{\bbC} = \tnshow{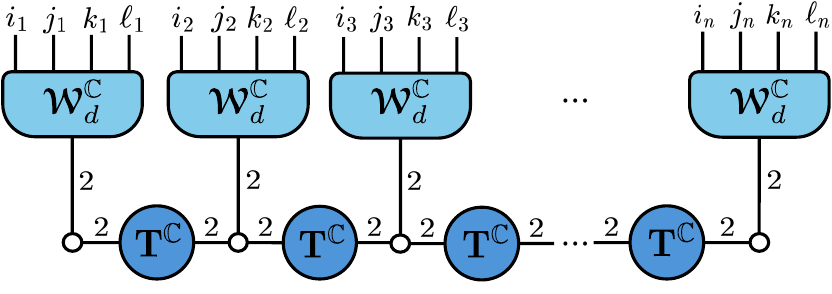}{0.55}.
    \end{equation*}
This fourth-moment tensor admits an interpretation as a probabilistic mixture of complex Gaussian--Kronecker vectors.
    Namely, sample \(b_1,\ldots,b_{n-1} \sim \bernoulli(1/\chi)\), let \(\vs=(s_1,\ldots,s_t)\) be the corresponding composition, and construct the vector
    \[
        \vh_{\vb} \coloneqq \vg^{(1)} \otimes \vg^{(2)} \otimes \cdots \otimes \vg^{(t)} \quad \text{where } \vg^{(q)} \sim \cN_{\bbC}(\vzero,\Id_{d^{s_q}}).
    \]
Then, the first four moments of \vomega and \(\vh_{\vb}\) agree:
    \begin{align*}
        &\E[\vomega]
        =\E[\vh_{\vb}] 
        = \vec0,
        \quad
        &&\E[\vomega\vomega^*]
        =\E[\vh_{\vb}^{\vphantom{*}}\vh_{\vb}^*] 
        = \mI, \\ 
        &\E[\vomega^{\otimes 3}]
        =\E[\vh_{\vb}^{\otimes 3}] 
        = \vec0,
        \quad
        &&\E[\vomega^{\otimes 2} \otimes \overline{\vomega}^{\otimes 2}]
        =\E[\vh_{\vb}^{\otimes 2} \otimes \overline{\vh}_{\vb}^{\otimes 2}].
    \end{align*}
    Consequently, for any matrix \mA of dimension $d^n$,
    \[
        \Var(\vomega^*\mA\vomega)
        \leq \norm{\mA}_{\rm F}^2 + 2\left[ \left(1 + \frac{1}{\chi}\right)^{n-1} - 1\right] \norm{\mA}_*^2.
    \]
\end{theorem}
 \section{Randomized linear algebra with \trps and beyond}
\label{sec:quadratic-form-condition}

\ENEdone
\CCdone
In the previous section, we proved \cref{thm:rmps_variance_intro}, establishing bounds on the variance of the quadratic form $\vomega^*\mA\vomega$ induced by an \rmps $\vomega$.
In this section, we employ this bound to analyze several randomized matrix algorithms.
All of these arguments remain valid for a generic random test vector \vomega,
provided a similar estimate for the quadratic form.

\subsection{The quadratic form condition}
\CCdone
\ENEdone 
To state results in generality, consider a test matrix of the form
\begin{equation} \label{eq:general-omega}
\mOmega \coloneqq \frac{1}{\sqrt{k}}
    \begin{bmatrix}
    \mid & \mid &        & \mid  \\
    \vomega_1 & \vomega_2 & \cdots & \vomega_k \\
    \mid & \mid &        & \mid  
    \end{bmatrix} 
    \in\bbF^{N \times k},
\end{equation}
assembled from iid copies $\vomega_1,\ldots,\vomega_k$ of an isotropic random vector $\vomega$ that meets the following condition.

\begin{definition}[Quadratic form condition] \label{def:qf-cond}
    An isotropic random vector $\vomega$ satisfies the \emph{quadratic form condition} with parameters $a, b \geq 0$ when
    it admits the bound
\begin{equation*}
        \Var(\vomega^* \mA\vomega) \le a \norm{\mA}_{\rm F}^2 + b \tr(\mA)^2 \quad \text{for every psd matrix } \mA.
    \end{equation*}
A test matrix \mOmega satisfies the quadratic form condition if it takes the form \cref{eq:general-omega} for iid vectors $\vomega_1,\ldots,\vomega_k$ satisfying the quadratic form condition.
\end{definition}

In the limit as \(N \to \infty\), any isotropic random vector must have parameters \(a + b \geq 2\) when \(\bbF=\bbR\) and \(a+b\geq1\) when \(\bbF=\bbC\) \cite{epperly23}.
Since \(\norm{\mA}_{\rm F} \leq \tr(\mA)\) for psd \mA,
the condition is less sensitive to the
size of $a$ than the size of $b$.
Gaussian vectors are near-optimal choices for the quadratic form condition,
as their parameters satisfy \begin{equation*}
    a = \begin{cases}
        2 & \text{if } \bbF = \bbR \\
        1 & \text{if } \bbF = \bbC
    \end{cases} \quad \text{and} \quad b = 0.
\end{equation*}
Meanwhile, \cref{thm:rmps_variance} shows that a real-valued \rmps satisfies the quadratic form condition with
\begin{equation*}
    a = 2 \left( 1 + \frac{1}{\chi} \right)^{n-1} \quad \text{and} \quad b = 3\left(1 + \frac{2}{\chi}\right)^{n-1} - 2\left(1 + \frac{1}{\chi}\right)^{n-1} - 1.
\end{equation*}
When $\chi / n$ is large, an \rmps has quadratic form properties similar with those of a standard Gaussian vector (cf.\ the discussion in \cref{sec:trace-guarantees}).
Concretely, in the real field, \(\chi \asymp n\) yields \(a \asymp b \asymp 1\) while \(\chi \asymp n/\eps\) yields \(a \leq 2+\eps\) and \(b \leq \eps\).
Analogous statements hold in the complex field.

\subsection{Subspace injections} \label{sec:osi}

\ENEdone 
\CCdone
Any test matrix $\mOmega$ with isotropic columns~\eqref{eq:general-omega}, such as an \rmpstest, can be applied to reduce the dimension of a vector while preserving its (squared) length, on average:
\begin{equation*}
    \E \big[ \norm{\mOmega^* \vx}_2^2 \big] = \norm{\vx}_2^2 \quad \text{for any } \vx \in \bbF^N.
\end{equation*}
But understanding how a test matrix acts on a single vector is not enough to understand its role
as primitive for linear algebra algorithms. Instead, we must study how the test matrix acts on a \emph{linear subspace}.

Let $\set{L} \subseteq \bbF^N$ be a linear subspace of dimension $r$.
For a parameter $\alpha > 0$, we say that the map $\mOmega^*$ is a \emph{subspace injection} on $\set{L}$ if 
\begin{equation} \label{eq:subpace-injection}
    \norm{\mOmega^*\vx}_2^2 \ge \alpha \cdot \norm{\vx}_2^2 \quad \text{ for every } \vx \in \set{L}. 
\end{equation}
The condition \cref{eq:subpace-injection} can be reformulated in terms of singular values:
\[
    \sigma_{\min}(\mOmega^*\mQ)^2 \geq \alpha
    \quad
    \text{where \mQ is orthonormal with range $\set{L}$}.
\]
The mapping $\mOmega^* : \set{L} \to \bbF^k$ is an injective function if and only if $\alpha > 0$, giving the parameter $\alpha$ an interpretation as a quantitative measure of injectivity.
We call an isotropic test matrix $\mOmega$ an \emph{oblivious subspace injection} (OSI) of dimension $r$ if the subspace injection property \cref{eq:subpace-injection} holds with probability at least $0.99$ for any $r$-dimensional subspace $\set{L}$.
The authors have developed
a theory of subspace injections and their use in linear algebra algorithms~\cite{cam25}; see also \cite{halko11,mahoney16,Drineas17} for earlier works that, in effect, exploit the subspace injection property \cref{eq:subpace-injection}.

The key technical observation states that the quadratic form condition implies
the OSI property.  In particular,
an \rmpstest is an OSI.

\begin{proposition}[Quadratic form condition implies subspace injectivity] \label{prop:quadratic-form-to-osi}
    Let $\mOmega \in \bbF^{N\times k}$ be a test matrix that satisfies the quadratic form condition with parameters $a,b\in\bbR$. 
    Then $\mOmega$ is an OSI with subspace dimension $r$ and with injectivity $\alpha$ when the embedding dimension $k\gtrsim (1-\alpha)^{-2} \cdot (a+b) \cdot r$.
    In particular, an \rmpstest is an OSI with injectivity $\alpha\geq \nicefrac{1}{2}$ when $k \gtrsim r$ and $\chi \gtrsim n$.
\end{proposition}

Provided that $\chi\gtrsim n$, an \rmpstest achieves the OSI property with an embedding dimension $k$ proportional to the subspace dimension $r$.
It is easy to see that $k\ge r$ is necessary to be an OSI, so indeed \trps with $\chi \gtrsim n$ are subspace injections whose embedding dimension is optimal, up to a universal multiplicative constant.
\Cref{prop:quadratic-form-to-osi} follows immediately from the next theorem.

\begin{importedtheorem}[Fourth-moment bound implies injectivity, \protect{\cite[Thm.~5.2 and Rem.~5.1]{tropp25}}] \label{impthm:fourth-to-osi}
    Consider a test matrix $\mOmega \in \bbF^{N\times k}$ of the form \cref{eq:general-omega} whose columns $\vomega_i \sim \vomega$ are drawn iid from a distribution that satisfies the bound
\begin{equation} \label{eq:fourth-moment-bound-osi}
        \E [ |\langle \vx, \vomega\rangle |^4 ] \le h \quad \text{for any unit vector } \vx \in \bbF^N.
    \end{equation}
Then \mOmega is an OSI of dimension $r$ with injectivity $\alpha$ when $k\gtrsim (1-\alpha)^{-2} \cdot h \cdot r$.
\end{importedtheorem}

The parameters $a,b\in\bbR_+$ in the quadratic form condition always satisfy \(a+b\geq1/3\) for any $N\ge2$ \cite{epperly23}.
By taking $\mA = \vx\vx^*$ in \cref{def:qf-cond}, we learn that the fourth-moment bound \cref{eq:fourth-moment-bound-osi} holds with parameter $h = 1 + a + b \le 4(a+b)$.
Thus, \cref{impthm:fourth-to-osi} immediately implies \cref{prop:quadratic-form-to-osi}.

\subsection{Linear algebra with OSIs}
\CCdone
\ENEdone

When we implement a randomized low-rank approximation algorithm using an OSI test matrix $\mOmega$,
the method enjoys rigorous error bounds.

\begin{importedtheorem}[OSIs for low-rank approximation (\emph{informal}), \protect{\cite[Thm.~1.3]{cam25}}]
    Fix matrix \(\mA\in\bbF^{N \times M}\) and target rank \(r \leq \min\{N,M\}\).
Using OSI test matrices with subspace dimension \(\cO(r)\) and injectivity \(\alpha \geq \nicefrac12\),
 both the randomized SVD and the generalized Nystr{\"o}m algorithms produce a rank-$\mathcal{O}(r)$ approximation $\hat{\mA}$ that satisfies \(\snorm{\mA - \hat\mA}_{\rm F} \lesssim \snorm{\mA - \lra \mA r}_{\rm F}\) with high probability.
    In particular, \trps suffice to achieve this guarantee when \(k \gtrsim r\) and \(\chi \gtrsim n\).
\end{importedtheorem}
See Cama\~no et al.~\cite{cam25} for formal statements of these theoretical guarantees and the required embedding dimensions.
The same paper also proves a guarantee for the Nystr\"om approximation, which we will discuss in the next subsection.
Additionally, the OSI technology leads to an error bound for the sketch-and-solve method for least-squares regression:
\begin{importedtheorem}[OSIs for least squares, \protect{\cite[Thm.~1.4]{cam25}}]
    Fix \(\mA\in\bbF^{N \times M}\) and \(\mB \in \bbF^{N \times P}\).
    Let \(\mOmega\in\bbF^{N \times k}\) be an OSI with subspace dimension \(M\) and injectivity \(\alpha \geq \nicefrac12\).
    Then, with high probability, the sketch-and-solve solution
    \(
        \tilde\mX \in \argmin_{\mX \in \bbF^{M \times P}} \norm{\mOmega^*\mA\mX-\mOmega^*\mB}_{\rm F}
    \)
    provides a near-optimal solution to the full least squares problem
\[
    \snorm{\mA\tilde\mX-\mB}_{\rm F} \lesssim \min_{\mX\in\bbF^{M \times P}}\norm{\mA\mX-\mB}_{\rm F}.
    \]
    In particular, this guarantee holds for an \rmps test matrix \(\mOmega\)
when \(k \gtrsim M\) and \(\chi \gtrsim n\).
\end{importedtheorem}

These theorems show that the OSI property, which follows from the quadratic form condition,
justifies several fundamental algorithms.
By exploiting the full power of the quadratic form condition, we can prove stronger or sharper results.
In \cref{sec:nystrom-analysis}, we will develop Frobenius-norm error bounds for the Nystr\"om approximation.
\Cref{sec:nystrompp-analysis} exploits these Frobenius-norm bounds to analyze the \nystrompp estimator,
verifying that this method achieves optimal sample complexity bounds. 

\subsection{Nystr\"om approximation}
\label{sec:nystrom-analysis}
\CCdone
We can use the quadratic form condition and subspace injectivity property to analyze the randomized Nystr\"om approximation,
introduced in \cref{sec:nystrom}. 

\begin{theorem}[Nystr\"om approximation from quadratic form]
    \label{thm:nystrom-frob-generic}
    Fix a psd matrix \(\mA\in\bbR^{N \times N}\) and a target rank \(r \leq N\). Let \(\mOmega \in \bbR^{N \times k}\) be a test matrix that satisfies the quadratic form condition with parameters $a,b \in \bbR_+$, and introduce the Nystr\"om approximation \(\mA\langle\mOmega\rangle\).
    Provided that \(k \gtrsim (a+b)r\), \begin{align}
        \norm{\mA - \mA\langle\mOmega\rangle}_*
            &\lesssim \norm{\mA-\lra{\mA}r}_*,\quad\text{and} \label{eq:nystrom-trace-qf} \\
        \snorm{\mA - \mA\langle\mOmega\rangle}_{\rm F}^2 
        &\lesssim \norm{\mA-\lra\mA r}_{\rm F}^2 + \frac{1}{r}\norm{\mA-\lra{\mA}r}_*^2, \label{eq:nystrom-fro-qf} 
    \end{align}
    with probability at least \(0.98\).
    In particular, these bounds hold for an \rmpstest when the embedding
    dimension $k \asymp r$ and the bond dimension $\chi \asymp n$.
\end{theorem}

This result holds for any test matrix that satisfies the quadratic form condition, and it implies our bound for Nystr\"om approximation with an \trp (\cref{thm:nystrom}) as a corollary.
The nuclear-norm bound \cref{eq:nystrom-trace-qf} follows directly from the OSI property (\cref{prop:quadratic-form-to-osi}); see \cite[Cor.~2.5]{cam25}.
The Frobenius-norm bound \cref{eq:nystrom-fro-qf} is new; the proof employs the quadratic form condition. 
The latter bound plays a role in the analysis of variance-reduced trace estimation algorithms; see \cref{sec:nystrompp-analysis}.
Let us turn to the proof of~\eqref{eq:nystrom-fro-qf}.

\subsubsection{The randomized SVD and the Gram correspondence} \label{sec:rsvd}
\CCdone
The Nystr\"om approximation is closely related to the approximation produced by the \emph{randomized SVD} algorithm \cite{halko11}.
This relation states that an error bound for one approximation implies a bound for the other.
We make use of this connection in our analysis, so we shall briefly review the randomized SVD.

Given $\mA\in \bbF^{N\times M}$ and a random test matrix $\mOmega \in \bbF^{M\times k}$, the randomized SVD approximation is 
\begin{equation*}
    \hat{\mA} = \mQ(\mQ^*\mA) \quad \text{where } \mQ = \operatorname{orth}(\mA\mOmega).
\end{equation*}
The function $\operatorname{orth}(\cdot)$ returns a matrix with orthonormal columns than span the range of its input.
The approximation $\hat\mA$ may also be called a randomized \QB approximation or a randomized projection approximation.
The randomized SVD approximation and the randomized Nystr\"om approximation are closely connected by the \emph{Gram correspondence} \cite{gittens11,gittens16}; see \cite[sec.~2.6]{mtmepperly25} for discussion.

\begin{importedlemma}[Gram correspondence, \protect{\cite[Lem.~1]{gittens16}}]
    \label{impthm:gram-correspondence}
    Fix a psd matrix \(\mA \in \bbR^{N \times N}\) and a test matrix \(\mOmega\in\bbR^{N \times k}\).
    Introduce the Nystr\"om approximation $\mA\langle\mOmega\rangle$, and let \(\mQ = \orth(\mA^{1/2}\mOmega)\in\bbR^{N \times k}\) define a randomized RSVD approximation $\mQ\mQ^*  \mA^{1/2}$ to the square root $\mA^{1/2}$.
    Then, for all \(p \geq 1\),
    \[
        \snorm{\mA-\mA\langle\mOmega\rangle}_{(p)} = \snorm{\mA^{1/2} - \mQ\mQ^*  \mA^{1/2}}_{(2p)}^2.
    \]
\end{importedlemma}

Under this correspondence, \cref{thm:nystrom-frob-generic} is equivalent to the following result for the randomized SVD.
Notice the change in the order of the Schatten norm: to bound the error of the Nystr\"om approximation in the Frobenius norm we must bound the error of the randomized SVD in the Schatten 4-norm.

\begin{theorem}[Randomized SVD from quadratic form]
    \label{thm:rsvd-qf}
    Fix a matrix \(\mA\in\bbR^{N \times M}\) and a target rank \(r \leq \min\{N,M\}\), and let \(\mOmega \in \bbR^{M \times k}\) be a test matrix satisfying the quadratic form condition with parameters $a,b \in \bbR_+$.
    Introduce the randomized SVD approximation $\hat{\mA} \coloneqq \mQ\mQ^*\mA$ induced by an orthonormal matrix \(\mQ = \orth(\mA\mOmega)\).
    Provided that \(k \gtrsim (a+b)r\),
\begin{align}
        \snorm{\mA - \hat\mA}_{\rm F}^2
            &\lesssim \snorm{\mA - \lowrank{\mA}_r}_{\rm F}^2, \quad\text{and} \label{eq:rsvd-fro-qf} \\
        \snorm{\mA - \hat\mA}_{(4)}^4 
        &\lesssim \norm{\mA-\lra\mA r}_{(4)}^4 + \frac{1}{r}\snorm{\mA - \lowrank{\mA}_r}_{\rm F}^4, \label{eq:rsvd-s4-qf} 
    \end{align}
    with probability at least $0.98$.
    In particular, these bounds hold for an \rmpstest with embedding dimension $k \asymp r$ and bond dimension $\chi \asymp n$.
\end{theorem}

\subsubsection{Proof of \texorpdfstring{\cref{thm:nystrom-frob-generic}}{Theorem \ref*{thm:nystrom-frob-generic}}}

In this section, we establish the Frobenius-norm error bound from \cref{thm:nystrom-frob-generic} by confirming the Schatten 4-norm error bound stated in \cref{thm:rsvd-qf}.
We begin with the following error decomposition:

\begin{importedtheorem}[Randomized SVD error, \protect{\cite[Prop.~8.5]{tropp2023}}]
    \label{impthm:rsvd-structural}
    Fix a matrix \(\mA\in\bbF^{N \times M}\) and target rank \(r \leq \min\{N,M\}\).
    Introduce a compact SVD of \mA \[\mA = \mU\sbmat{\mSigma_1 \\ & \mSigma_2} \sbmat{\mV_1^* \\ \mV_2^*}\] where \(\mSigma_1\in\bbR^{r \times r}\) and \(\mV_1\in\bbF^{M \times r}\) contain the top singular values and right singular vectors of \mA, while \(\mSigma_2\in\bbR^{(L-r) \times (L-r)}\) and \(\mV_2\in\bbF^{M \times (L-r)}\) contain the remaining nonzero singular values and vectors.
    Let \(\mOmega\in\bbF^{M \times k}\) be a matrix, and assume that $\operatorname{rank}(\mV_1^*\mOmega) = r$. 
    Then with \(\mQ=\orth(\mA\mOmega)\), it holds that
    \[
        \snorm{\mA - \mQ\mQ^*\mA}_{(4)}^2 \leq \snorm{\mA - \lra\mA r}_{(4)}^2 + \snorm{\mSigma_2^{\vphantom{*}}(\mV_2^*\mOmega)(\mV_1^*\mOmega)^+}_{(4)}^2.
    \]
\end{importedtheorem}

We use this decomposition to prove the Schatten 4-norm bound in \cref{thm:rsvd-qf}.

\begin{proof}[Proof of Schatten 4-norm bound in \cref{thm:rsvd-qf}]
    Instate the notation of \cref{impthm:rsvd-structural}, and assume that $k\gtrsim (a+b)r$.
    Our goal is to establish the bound
\begin{equation} \label{eq:want-to-prove}
    \norm{\mSigma_2^{\vphantom{*}}\mV_2^* \mOmega(\mV_1^* \mOmega)^+}_{(4)}^4 \lesssim \norm{\mA - \lra\mA r}_{(4)}^4 + \frac1r\norm{\mA - \lra\mA r}_{\rm F}^4.
    \end{equation}
    By \cref{prop:quadratic-form-to-osi} and the hypothesis $k\gtrsim (a+b)r$, the embedding \(\mOmega\) is an OSI with dimension $r$ and injectivity $\alpha = \nicefrac{1}{2}$.
    Therefore, with probability at least \(0.99\), we have
    \[
        \norm{(\mV_1^* \mOmega)^+}_2^2
        = \frac1{\sigma_{\rm min}(\mV_1^* \mOmega)^2} \leq 2.
    \]
    By the submultiplicative property, it holds that
    \begin{equation} \label{eq:schatten-4-first-step}
        \snorm{\mSigma_2^{\vphantom{*}}\mV_2^* \mOmega(\mV_1^* \mOmega)^+}_{(4)}^4 \le \snorm{\mSigma_2^{\vphantom{*}}\mV_2^* \mOmega}_{(4)}^4 \cdot \snorm{(\mV_1^* \mOmega)^+}_2^4
        \leq 4 \snorm{\mSigma_2^{\vphantom{*}}\mV_2^* \mOmega}_{(4)}^4.
    \end{equation}
    We now bound $\snorm{\mSigma_2^{\vphantom{*}}\mV_2^* \mOmega}_{(4)}^4$.
    Define the psd matrix \(\mB = \mV_2^{\vphantom{*}}\mSigma_2^2\mV_2^* \), so that
    \begin{equation} \label{eq:schatten-4-intermediate}
        \E\bigl[\snorm{\mSigma_2\mV_2^* \mOmega}_{(4)}^4\bigr]
        = \E\bigl[\snorm{\mOmega^* \mB\mOmega}_{\rm F}^2\bigr]
        = \frac1{k^2}\sum_{i,j=1}^k \E\bigl[|\vomega_i^* \mB\vomega_j^{\vphantom{*}}|^2\bigr].
    \end{equation}
    We tackle the on-diagonal and off-diagonal terms separately.
    For \(i \neq j\), the isotropy of \vomega gives
    \[
        \E[|\vomega_i^* \mB\vomega_j^{\vphantom{*}}|^2]
        = \E[\tr(\vomega_i^{\vphantom{*}}\vomega_i^* \mB\vomega_j^{\vphantom{*}}\vomega_j^* \mB)] = \E[\tr(\Id \cdot \mB\cdot \Id\cdot \mB)]
        = \norm{\mB}_{\rm F}^2.
    \]
    For \(i = j\), the quadratic form condition gives
    \[
        \E[|\vomega_i^* \mB\vomega_i^{\vphantom{*}}|^2]
        \leq a \norm{\mB}_{\rm F}^2 + (1+b)\tr(\mB)^2.
    \]
    Returning to \cref{eq:schatten-4-intermediate}, we obtain
\[
        \E\bigl[\snorm{\mSigma_2^{\vphantom{*}}\mV_2^* \mOmega}_{(4)}^4\bigr]
        \leq \left(1+\frac{a-1}{k}\right)\norm{\mB}_{\rm F}^2 + \frac{1+b}{k}\tr(\mB)^2.
    \]
    The matrix \mB satisfies $\norm{\mB}_{\rm F}^2 = \norm{\mA-\lra\mA r}_{(4)}^4$ and $\tr(\mB) = \norm{\mA - \lra\mA r}_{\rm F}^2$, and we have selected the parameter \(k \gtrsim (1+a+b)r\).
    Therefore,
    \[
        \E\bigl[\snorm{\mSigma_2^{\vphantom{*}}\mV_2^* \mOmega}_{(4)}^4\bigr]
        \lesssim \norm{\mA-\lra\mA r}_{(4)}^4 + \frac{1}{r}\norm{\mA - \lra\mA r}_{\rm F}^4.
    \]
    Substituting this result into \cref{eq:schatten-4-first-step} and applying Markov's inequality,
    we reach the desired conclusion \cref{eq:want-to-prove}.
\end{proof}

\subsection[Nystr\"om++]{\nystrompp}
\label{sec:nystrompp-analysis}
\CCdone
In this section, we analyze the \nystrompp estimator under the quadratic form condition.
Recall that the \nystrompp estimator takes the form
\[
    \nystromppest = \tr(\hat\mA) + \GH_\ell(\mA-\hat\mA),
\]
where \(\hat\mA = \mA\langle\mOmega\rangle\) is a Nystr\"om approximation to \mA generated using an test matrix $\mOmega = k^{-1/2}[\vomega_1,\ldots,\vomega_k] \in \bbF^{N\times k}$ and where \(\GH_\ell(\mB) = \frac1\ell \sum_{i=1}^\ell \vpsi_i^* \mB\vpsi_i^{\vphantom{*}}\) denotes the Girard--Hutchinson estimator.
The vectors $\vomega_i$ and \(\vpsi_i\) are iid copies of \vomega.

We prove the following result for a general embedding \mOmega that satisfies the quadratic form condition.
This guarantee implies \cref{thm:nystrompp-spectral-rmps}.

\begin{theorem}[\nystrompp: Spectral accuracy]
    \label{thm:nystrompp-generic}
    Fix a psd matrix \(\mA\in\bbF^{N \times N}\) and a target rank $r\le N$.
    Suppose \nystrompp is implemented with iid vectors that satisfy the quadratic form condition with parameters $a,b \in \bbR_+$.  Suppose that $k$ matvecs are used for low-rank approximation and $\ell$ matvecs for trace estimation.
    Then the \nystrompp estimator is unbiased $\E[\nystromppest] = \tr(\mA)$, and its error is controlled by the magnitude of the best rank-$r$ approximation error.  Provided that \(k \gtrsim (a+b)r\),
    \begin{equation} \label{eq:nystrompp-spectral}
        |\nystromppest - \tr(\mA)|^2 \lesssim \frac{a}{\ell} \snorm{\mA - \lowrank\mA_r}_{\rm F}^2 + \left(\frac{a}{\ell r} + \frac{b}{\ell} \right) \norm{\mA - \lowrank{\mA}_r}_*^2
\end{equation}
    with probability at least 0.98.
    As a consequence, when $k = \ell$ and the probe vector is an \rmps with bond dimension $\chi \asymp n$,
\begin{equation} \label{eq:nystrompp-spectral-2}
        |\nystromppest - \tr(\mA)| \lesssim \frac{1}{\sqrt{k}} \norm{\smash{\mA - \lowrank{\mA}_{\lfloor \rc k\rfloor}}}_* \quad \text{for an absolute constant } 0 < \rc < 1.
    \end{equation}
\end{theorem}

\begin{proof}
    Conditional on $\mOmega$, the Girard--Hutchinson estimator $\GH_\ell(\mA - \hat\mA)$ is an unbiased estimator for $\tr(\mA - \hat\mA)$.
    By the quadratic form condition, it has variance
\begin{equation*}
        \Var(\GH_\ell(\mA - \hat\mA) \mid \mOmega) \le \frac{a}{\ell} \norm{\smash{\mA - \hat\mA}}_{\rm F}^2 + \frac{b}{\ell} \norm{\smash{\mA - \hat\mA}}_*^2.
    \end{equation*}
By \cref{thm:nystrom-frob-generic} and a union bound, it holds with probability at least 0.98 that 
\begin{equation*}
        |\nystromppest - \tr(\mA)|^2 \lesssim \frac{a}{\ell} \left( \snorm{\mA - \lowrank\mA_r}_{\rm F}^2 + \frac{1}{r} \norm{\smash{\mA - \lowrank{\mA}_r}}_*^2 \right) + \frac{b}{\ell} \norm{\smash{\mA - \lowrank{\mA}_r}}_*^2.
    \end{equation*}
This is the first conclusion \cref{eq:nystrompp-spectral}.

    To establish the second conclusion \cref{eq:nystrompp-spectral-2}, recall that an \rmps with $\chi \asymp n$ meets the quadratic form condition with parameters $a,b\lesssim 1$.
    Consequently, when \nystrompp is implemented with such an \rmps and $k=\ell$, the conditions for the bound \cref{eq:nystrompp-spectral} hold with $r = \lfloor \rc k \rfloor$ for some universal constant $0 < \rc < 1$.
    Consequently, applying \cref{eq:nystrompp-spectral}, we obtain
\begin{equation*}
        |\nystromppest - \tr(\mA)|^2 \lesssim \frac{1}{k} \left( \snorm{\mA - \lowrank\mA_{\lfloor \rc k\rfloor}}_{\rm F}^2 + \frac{1}{k} \norm{\smash{\mA - \lowrank\mA_{\lfloor \rc k\rfloor}}}_*^2 \right) + \frac{1}{k} \norm{\smash{\mA - \lowrank\mA_{\lfloor \rc k\rfloor}}}_*^2.
    \end{equation*}
Bounding the Frobenius norm of the psd matrix $\mA - \lowrank\mA_{\lfloor \rc k\rfloor}$ by its trace yields the desired bound \cref{eq:nystrompp-spectral-2}.
\end{proof}

The spectral accuracy bound \cref{thm:nystrompp-generic} shows that \nystrompp is effective when applied to a matrix that is very near to a low-rank matrix.
But how does \nystrompp perform on a general input matrix, perhaps chosen to be a hard instance for the algorithm?
For such cases, we provide the following relative-error bound, which implies \cref{thm:nystrompp}.

\begin{corollary}[\nystrompp: Relative-error guarantee] \label{cor:nystrompp-worst-case}
Fix a psd matrix $\mA \in \bbF^{N\times N}$.  Apply \nystrompp using random vectors that satisfy the quadratic form condition with parameters $a$ and $b$.
    With probability at least 0.98, the \nystrompp estimator satisfies
\begin{equation} \label{eq:nystrompp-worst-case}
        |\nystromppest - \tr(\mA)| \lesssim 
        \sqrt{\frac{a(a+b)}{k\ell} + \frac{b}{\ell}} \cdot \tr(\mA)
\end{equation}
when $k$ matvecs are apportioned to low-rank approximation and $\ell$ matvecs are apportioned to trace estimation.
    In particular, the estimator achieves relative error $\varepsilon$ when $\ell \gtrsim (a(a+b)/k + b)/\varepsilon^2$.
    Consequently, for \rmps with bond dimension $\chi \asymp n / \varepsilon$, setting $k = \ell \asymp 1/\varepsilon$ suffices to compute the trace to relative error $\varepsilon$.
\end{corollary}

To prove this result, we need the following standard comparison between the best rank-$r$ approximation error in the Frobenius norm and the trace norm.

\begin{importedlemma}[Frobenius to trace, \protect{\cite[Lem.~7]{gilbert2007one}}] \label{implem:l1l2}
    Let $\mA$ be a psd matrix and fix a rank $r\ge 1$.
    Then
\begin{equation*}
        \norm{\mA - \lowrank{\mA}_r}_{\rm F}^2 \le \frac{1}{2r} \tr(\mA)^2.
    \end{equation*}
\end{importedlemma}

With this result at hand, we can prove \cref{cor:nystrompp-worst-case}.

\begin{proof}[Proof of \cref{cor:nystrompp-worst-case}]
    By \cref{thm:nystrompp-generic}, it holds with probability $0.98$ that
\begin{equation*}
        |\nystromppest - \tr(\mA)|^2 \lesssim \frac{a}{\ell} \snorm{\mA - \lowrank\mA_r}_{\rm F}^2 + \left(\frac{a}{\ell r} + \frac{b}{\ell} \right) \norm{\smash{\mA - \lowrank{\mA}_r}}_*^2
    \end{equation*}
for some $r = \lfloor \rc \cdot k/(a+b) \rfloor$ with absolute constant $0 < \rc < 1$.
    Apply \cref{implem:l1l2} and the trivial comparison $\norm{\smash{\mA - \lowrank{\mA}_r}}_* \le \tr(\mA)$ to obtain
\begin{equation*}
        |\nystromppest - \tr(\mA)|^2 \lesssim \left(\frac{a}{\ell r} + \frac{b}{\ell} \right) \tr(\mA)^2 \lesssim \left(\frac{a(a+b)/k+b}{\ell} \right) \tr(\mA)^2.
    \end{equation*}
This is the first conclusion \cref{eq:nystrompp-worst-case}.

    For an \rmps, setting $\chi \asymp n /\varepsilon$ yields $a \lesssim 1$ and $b\lesssim \varepsilon$.
    Setting $k = \ell$ and applying \cref{eq:nystrompp-worst-case} gives
\begin{equation*}
        |\nystromppest - \tr(\mA)| \lesssim \sqrt{\left(\frac{1}{k^2} + \frac{\varepsilon}{k}\right)} \cdot \tr(\mA)
    \end{equation*}
Thus, setting $k = \ell \asymp 1/\varepsilon$ suffices to achieve relative error $\varepsilon$.
\end{proof}

\section*{Acknowledgments} 
We thank Gil Goldshlager, Zhen Huang, Michael Ragone, Lin Lin, Diyi Liu, and Kevin Stubbs for helpful discussions.
ENE acknowledges support from the Miller Institute for Basic Research in Science, University of California Berkeley.
CC has been supported by the Caltech Kortschak Scholars Program and the NSF GRFP Award 2139433.
RAM would like to acknowledge the Caltech S2I Center and the DARPA DIAL and DARPA AIQ programs for providing partial support of this work.
JAT has been supported by the Caltech Carver Mead New Adventures Fund and ONR Award N00014-24-1-2223.

\section*{AI Statement}

This work used a large language model coding assistant to support code profiling and implementation optimization.

\appendix 

\section{Details for experiments}

\label{app:experiment_details}
All experiments were run on the UC Berkeley SCF cluster, using 8 cores of a dual-socket AMD EPYC
7543 (2.8 GHz) CPU. 
Our code is written in \texttt{python} $
\texttt{v3.12.12}$ with some choice operations implemented in \texttt{C}. 

\paragraph{\Cref{fig:rnla_fails_at_scale}}

We represent the PTFIM Hamiltonian $\cref{eq:ptfim}$ as a bond-dimension $5$ \mpo by adding a bond-dimension $2$ periodic-boundary \mpo to the standard bond-dimension $3$ open-chain TFIM \mpo \cite[p.~7]{pirvu2010matrix}.
In our experiments, we use interaction strength $J=1$, transverse field strength to $h=8$, and inverse temperature $\beta=0.5$. 
Following \cite{epperly24trace}, we improve numerical stability by applying XNysTrace to the shifted operator $\e^{-\beta(\mH+b\mI)}$ for $b=n(|J|+|h|)$ and scaling the result to remove the effect of the shift.

The action $\vomega\mapsto \e^{-\beta\mH}\vomega$ is approximated using GSE-TDVP1 \cite{Haegeman_2011,Haeg16,YA20} with \mpo--\mps products evaluated using the SRC method with adaptive truncation \cite[App.~C]{camano25}.
We use $\lceil 5\sqrt n\,\rceil$ imaginary-time steps and set the truncation cutoff to $10^{-14}$ for subspace enrichment, SRC, and TDVP.
Mirroring the implementation of GSE-TDVP1 in the ITensor software library, global subspace enrichment is performed once at the start of time evolution. 
We use three steps of subspace enrichment with $\chi = 1$ and two steps of enrichment when $\chi > 1$, which were the minimal values we found the led to satisfactory results.

We test \rmpss with $\chi\in\{2,4,8\}$ and report $\chi=4$, which gave the lowest runtime. Although this value is below the theoretically supported scaling $\chi\asymp n$, the smaller probe bond dimension substantially reduces the cost of TDVP1-GSE and \GramNystrom, offsetting the additional GSE-TDVP1 queries needed to reach the target accuracy.

\newcommand{\nA}{n_{\scalebox{0.6}{$\scriptscriptstyle\rm A$}}}
\paragraph{\Cref{fig:nystrom}}
Given an \mps $\vpsi \in \bbC^{d^n}$ of order $n$ with bond dimension $\xi$ and contiguous subsystems of sizes $n_{\rm A} + n_{\rm B} = n$, it is straightforward to represent the reduced density matrix $\mrho_{\rm A} \coloneqq \tr_{\rm B}(\vpsi\vpsi^*) \in \bbC^{d^{\nA} \times d^{\nA}} \cong \bbC^{d^{2\nA}}$ as an \mps of order $2n_{\rm A}$ and bond dimension $\xi$ (not an \mpo!).
The eigenvalues of $\mrho_{\rm A}$ are readily obtained by placing this \mps into an appropriate canonical form, and matrix--\mps products $\mrho_{\rm A}\vomega$ are obtained by applying \vomega with the $n_{\rm A}$ rightmost sites and contracting down to an \mps of order $n_{\rm A}$ and bond dimension $\xi$.
In our experiments, we use this \mps representation of $\mrho_{\rm A}$ both to extract the true eigenvalues and to perform matrix--\mps products.

\bibliographystyle{siamplain}

\end{document}